\documentclass[11pt]{article}

\usepackage{fancyhdr}
\usepackage{amsfonts}
\usepackage{amsmath}
\usepackage{amssymb}
\usepackage{amsthm}
\usepackage{tikz}

\usepackage{amscd}
\usepackage{amstext}

\usepackage{oldgerm} 

\usepackage{fancyhdr}
\usepackage{amsfonts}
\usepackage{amsmath}
\usepackage{amssymb}
\usepackage{amsthm}

\usepackage{amscd}
\usepackage{amstext}

\usepackage{graphics}
\usepackage{graphicx}

\newlength{\fixboxwidth}
\newcommand{\re}{\mathbb{R}}\newcommand{\N}{\mathbb{N}}
\newcommand{\zz}{\mathbb{Z}}

\newcommand{\R}{{\re}^d}

\newcommand{\cs}{{\mathcal S}}

\newcommand{\cf}{{\mathcal F}}
\newcommand{\cfi}{{\cf}^{-1}}

\newcommand{\supp}{{\rm supp \, }}

\newcommand{\dist}{{\rm dist \, }}

\newcommand{\be}{\begin{equation}}
\newcommand{\ee}{\end{equation}}
\newcommand{\beq}{\begin{eqnarray}}
\newcommand{\beqq}{\begin{eqnarray*}}
\newcommand{\eeq}{\end{eqnarray}}
\newcommand{\eeqq}{\end{eqnarray*}}

\newtheorem{satz}{Theorem}

\newtheorem{rem}{Remark}
\newtheorem{defi}{Definition}
\newtheorem{lem}{Lemma}
 
\newtheorem{prop}{Proposition}

\begin{document}


\title{Truncation in Besov-Morrey and Triebel-Lizorkin-Morrey spaces}
\author{Marc Hovemann}
\date{\today}
\maketitle

{\scriptsize  Institute of Mathematics, Friedrich-Schiller-University Jena, Ernst-Abbe-Platz 2, 07743 Jena, Germany }

\vspace{0.3 cm }


\vspace{0.3 cm }


\vspace{0.3 cm }

\textbf{Abstract.} We will prove that under certain conditions on the parameters the operators $T^{+}f = \max(f,0)$ and $ Tf = |f|  $ are bounded mappings on the Triebel-Lizorkin-Morrey and Besov-Morrey spaces. Moreover we will show that some of the conditions we mentioned before are also necessary. Furthermore we prove that for $p < u  $ in many cases the Triebel-Lizorkin-Morrey spaces do not have the Fubini property.


\section{Introduction and main results}\label{sec_1}


Nowadays the Besov spaces $ B^s_{p,q}(\R)  $ and the Triebel-Lizorkin spaces $F^s_{p,q} (\R)$ are well-established tools to describe the regularity of functions and distributions. These function spaces have been investigated in detail in the famous books of Triebel, see \cite{Tr83}, \cite{Tr92} and \cite{Tr06}. In the recent years a growing number of authors worked with a generalization of Besov and Triebel-Lizorkin spaces where the $ L_{p}(\R) $-quasinorm was replaced by a Morrey-quasinorm. So the Besov-Morrey spaces $ \mathcal{N}^{s}_{u,p,q}(\R) $ and the Triebel-Lizorkin-Morrey spaces $ \mathcal{E}^{s}_{u,p,q}(\R) $ attracted some attention. Besov-Morrey spaces have been introduced by Kozono and Yamazaki in 1994, see \cite{KoYa}, whereas the spaces $ \mathcal{E}^{s}_{u,p,q}(\R) $ appeared the first time in a paper written by Tang and Xu in 2005, see \cite{TangXu}. Later, with a different notation, the Triebel-Lizorkin-Morrey spaces also showed up in \cite{Tr14} and in \cite{yy1} as well as in \cite{yy2}. In this paper we will study the mapping properties of the truncation operator $ T^{+}  $ given by 
\begin{align*}
(T^{+}f)(x) = \max(f(x),0) , \qquad x \in \R,  
\end{align*} 
in which $ f $ is a real-valued function from a Besov-Morrey space or a Triebel-Lizorkin-Morrey space. The operator $ T^{+} $ is one member of a bigger class of operators that is called composition operators, see chapter 5.3 in \cite{RS}. Those operators play an important role in the theory of nonlinear partial differential equations, see for example chapter 8 in \cite{GiT}. Throughout this paper we will answer several questions concerning the operator $T^{+}$. Let $ \mathbb{N}^{s}_{u,p,q}(\mathbb{R}^{d})  $ be the real part of $ \mathcal{N}^{s}_{u,p,q}(\mathbb{R}^{d})  $ and $ \mathbb{E}^{s}_{u,p,q}(\mathbb{R}^{d})  $ be the real part of $ \mathcal{E}^{s}_{u,p,q}(\mathbb{R}^{d})  $. We write $ \mathbb{A}^{s}_{u,p,q}(\mathbb{R}^{d})  $ when we mean either $ \mathbb{N}^{s}_{u,p,q}(\mathbb{R}^{d})  $ or $ \mathbb{E}^{s}_{u,p,q}(\mathbb{R}^{d})  $. Then we want to know under which conditions on the parameters $ s,p,u,q $ and $ d$ we have $ T^{+}(\mathbb{A}^{s}_{u,p,q}(\R)) \subset \mathbb{A}^{s}_{u,p,q}(\R)   $. This is the so-called acting property. Moreover we will investigate under which conditions on the parameters the operator $ T^{+} : \mathbb{A}^{s}_{u,p,q}(\R) \rightarrow \mathbb{A}^{s}_{u,p,q}(\R)   $ is bounded on $  \mathbb{A}^{s}_{u,p,q}(\R)  $. $ T^{+} $ is strongly connected with the operator $  T $ given by
\begin{align*}
(Tf)(x) = |f(x)| , \qquad x \in \R.  
\end{align*}
The operators $ T^{+} $ and $ T  $ have many properties in common. Therefore in this paper we also will study the behaviour of $ T $. For the original Besov spaces as well as for the original Triebel-Lizorkin spaces results concerning truncation operators are already known. Let $ \mathbb{B}^{s}_{p,q}(\R)  $ be the real part of $ B^{s}_{p,q}(\R) $ and $ \mathbb{F}^{s}_{p,q}(\R)  $ be the real part of $ F^{s}_{p,q}(\R) $. Then the following result is known since many years. 

\begin{satz}\label{Hist_Res}
Let $ \mathbb{A} \in \{ \mathbb{B}, \mathbb{F} \}  $. Let $ 1 \leq p < \infty $, $ 1 \leq q \leq \infty $ and $ 0 < s < 1 + 1/p $. For $ \mathbb{A} = \mathbb{F} $ in the case $ p = 1  $ we assume $ s \not = 1   $. Then there is a constant $ C > 0  $ independent of $ f \in \mathbb{A}^{s}_{p,q}(\R)  $ such that
\begin{align*}
\Vert T^{+}f  \vert A^{s}_{p,q}(\mathbb{R}^{d})  \Vert \leq C \Vert f \vert A^{s}_{p,q}(\mathbb{R}^{d})  \Vert
\end{align*}
holds for all $ f \in \mathbb{A}^{s}_{p,q}(\R)  $. Moreover in the formulation of theorem \ref{Hist_Res} one can replace the operator $ T^{+} $ by $ T $.
\end{satz}

This result can be found in \cite{Tr01}, see theorem 25.8 in chapter 25. For earlier contributions we refer to chapter 5.4.1. in \cite{RS} as well as to \cite{Bou93}, \cite{BouMey} and \cite{Os1}. Now let us look at the case $ p < u  $. Then in this paper we will prove the following result for the Besov-Morrey spaces. 

\begin{satz}\label{MR_N_d>1}
Let $ 1 \leq p < u < \infty  $, $ 1 \leq q \leq \infty   $ and $ s > 0  $. We assume
\[
\left\{ \begin{array}{lll}
\frac{1}{p} -  \frac{1}{u} > 1 - \frac{1}{d}  & \quad & \mbox{in the case} \qquad 1 \leq s < \min  ( 1 + \frac{1}{p} , 1 + \frac{d}{u} ) \ \mbox{and} \ d > 1 ;
\\  
q \not = \infty & \quad & \mbox{in the case}\qquad s = \min  ( 1 + \frac{1}{p} , 1 + \frac{d}{u} ) \ \mbox{and} \ d > 1  .
\end{array}
\right.
\]
Then $ T^{+}  $ acts on $  \mathbb{N}^{s}_{u,p,q}(\mathbb{R}^{d}) $ and there is a constant $ C > 0 $ independent of $ f \in \mathbb{N}^{s}_{u,p,q}(\mathbb{R}^{d})  $ such that we have 
\begin{equation}\label{MR_N_eq2}
\Vert T^{+}f  \vert \mathcal{N}^{s}_{u,p,q}(\mathbb{R}^{d})  \Vert \leq C \Vert f \vert \mathcal{N}^{s}_{u,p,q}(\mathbb{R}^{d})  \Vert
\end{equation}
if and only if 
\begin{equation}\label{MR_N_eq1}
s < \min \Big ( 1 + \frac{1}{p} , 1 + \frac{d}{u} \Big ). 
\end{equation}
Moreover in the formulation of theorem \ref{MR_N_d>1} one can replace the operator $ T^{+} $ by $ T $.
\end{satz}

It turns out that the critical border $ s = 1 + 1/p $ we know for the spaces $ \mathbb{B}^{s}_{p,q}(\mathbb{R}^{d})  $ is replaced by $ s = \min  ( 1 + 1/p , 1 + d/u  )  $ in the case of the spaces $  \mathbb{N}^{s}_{u,p,q}(\mathbb{R}^{d})  $. There is the surprising new phenomenon that for $ p < u  $ the critical border also depends on the dimension $d$. For $ p = u  $ this is not the case. Here we always have $ \min  ( 1 + 1/p , 1 + d/u ) =  1 + 1/p  $. So we recover the original result. Moreover for $ d = 1  $ because of $ p \leq u   $ we obtain $ \min  ( 1 + 1/p , 1 + 1/u  ) = 1 + 1/u  $. Hence the condition concerning the parameter $ s $ becomes much more easy in this case. The additional condition $  1/p -  1/u > 1 - 1/d  $ we need in the case $ d > 1  $ seems to be of technical nature. Maybe it can be left away using another method for the proof. Now let us look at the Triebel-Lizorkin-Morrey spaces. There is the following result.  

\begin{satz}\label{MR_E_d>1}
Let $ 1 \leq p < u < \infty  $, $ 1 \leq q \leq \infty   $ and $ s > 0  $. We assume
\[
\left\{ \begin{array}{lll}
p \not = 1,  q \not = \infty \ \mbox{and} \ \frac{1}{p} -  \frac{1}{u} > 1 - \frac{1}{d}  & \! & \mbox{in the case}\quad s=1 ;
\\
\frac{1}{p} -  \frac{1}{u} > 1 - \frac{1}{d}  & \! & \mbox{in the case} \quad 1 < s < \min  ( 1 + \frac{1}{p} , 1 + \frac{d}{u} ) \ \mbox{and} \ d > 1 ;
\\  
\frac{u}{p} \leq d & \! & \mbox{in the case}\quad s = \min  ( 1 + \frac{1}{p} , 1 + \frac{d}{u} ).
\end{array}
\right.
\]
Then $ T^{+}  $ acts on $  \mathbb{E}^{s}_{u,p,q}(\mathbb{R}^{d}) $ and there is a constant $ C > 0 $ independent of $ f \in \mathbb{E}^{s}_{u,p,q}(\mathbb{R}^{d})  $ such that we have 
\begin{equation}\label{MR_E_eq2}
\Vert T^{+}f  \vert \mathcal{E}^{s}_{u,p,q}(\mathbb{R}^{d})  \Vert \leq C \Vert f \vert \mathcal{E}^{s}_{u,p,q}(\mathbb{R}^{d})  \Vert
\end{equation}
if and only if
\begin{align*}
s < \min \Big ( 1 + \frac{1}{p} , 1 + \frac{d}{u} \Big ). 
\end{align*}
Moreover in the formulation of theorem \ref{MR_E_d>1} one can replace the operator $ T^{+} $ by $ T $.
\end{satz}
The special case $ p > 1  $, $ s = 1  $ and $ q = 2  $ in theorem \ref{MR_E_d>1} refers to the so-called Sobolev-Morrey spaces, see proposition \ref{s=1_E_MR} below. It turns out that also in the case of the Triebel-Lizorkin-Morrey spaces the critical border $ s = 1 + 1/p  $ we know for the spaces $ \mathbb{F}^{s}_{p,q}(\mathbb{R}^{d})  $ is replaced by $ s = \min  ( 1 + 1/p , 1 + d/u  )  $ for $ p < u  $. When you want to prove theorem \ref{MR_E_d>1} there is a big difference between the cases $ p = u  $ and $ p < u  $. So for $ d > 1  $ in most of the cases the Triebel-Lizorkin-Morrey spaces do not have the so-called Fubini property for $ p \not = u  $, see lemma \ref{res_no_Fub}. 

This paper is organized in the following way. In section \ref{sec_1} that you read at the moment the main results are formulated. In section \ref{sec_2} the spaces $  \mathbb{N}^{s}_{u,p,q}(\mathbb{R}^{d})  $ and $  \mathbb{E}^{s}_{u,p,q}(\mathbb{R}^{d})  $ are defined. Moreover some important properties of them are collected here. In section \ref{sec_3} we will prove the main results. So in subsection \ref{sec_3.1} we will deal with the simple case $ 0 < s < 1  $. In subsection \ref{sec_3.2} we deal with the Triebel-Lizorkin-Morrey spaces and look at the case $ s > 1 $ and $ d = 1  $. Here our main tool will be a Hardy-type inequality. In subsection \ref{sec_3.3} we will prove the results for $ s \geq 1   $ and $ d > 1 $. For that purpose we will apply the so-called Morrey characterization for the Triebel-Lizorkin-Morrey spaces. In subsection \ref{sec_T_BMS} we will investigate the boundedness properties of the operator $ T $ in the context of Besov-Morrey spaces. In subsection \ref{subsec_Nec} we prove that some of the conditions concerning the parameter $s$ that appear in the main results are also necessary. At the end of this paper in section \ref{sec_5} we will discuss some further properties of the operator $ T^{+} $ like continuity or Lipschitz continuity. But at first we will fix some notation.       

\section*{Notation}

As usual $\N$ denotes the natural numbers, $\N_0$ the natural numbers including $0$, $\zz$ the integers and $\re$ the real numbers. $\R$ denotes the $d$-dimensional  Euclidean space. We put
\[
 B(x,t) := \{y\in \R: \quad |x-y|< t\}\, , \qquad x \in \R\, , \quad t>0.
\]
Let $\mathcal{S}(\R)$ be the collection of all Schwartz functions on $\R$ endowed with the usual topology and denote by $\mathcal{S}'(\R)$ its topological dual, namely the space of all bounded linear functionals on $\mathcal{S}(\R)$ endowed with the weak $\ast$-topology. With $ \mathbb{S}(\R)  $ we denote the set of all real-valued Schwartz functions. The symbol $\cf$ refers to  the Fourier transform,
$\cfi$ to its inverse transform, both defined on $\cs'(\R)$.
By $C^\infty_0(\R)$ we mean the set of all infinitely often differentiable real-valued functions on $\R$ with compact support. For $ 1 \leq p \leq \infty  $ with $ L_{p}(\R)  $ we denote the Lebesgue spaces and $ L_{p}^{loc}(\R)   $ are the local Lebesgue spaces. By $ \mathbb{L}_{p}(\R)  $ and $ \mathbb{L}_{p}^{loc}(\R)  $ we mean the real parts of $ L_{p}(\R)   $ and $ L_{p}^{loc}(\R)   $. For two quasi-Banach spaces $ X $ and $ Y $ we write $ X \hookrightarrow Y $ if $ X \subset Y $ and the natural embedding of $ X $ into $ Y $ is continuous. The symbols  $C, C_1, c, c_{1} \ldots $ denote  positive constants that depend only on the fixed parameters $d,s,u,p,q$ and probably on auxiliary functions. When we write $ A \sim B   $ we mean that there exist two constants $ C_{1}, C_{2} > 0   $ such that $ A \leq C_{1} B \leq C_{2} A  $. 

\section{Definition and basic properties of Besov-Morrey spaces and Triebel-Lizorkin-Morrey spaces}\label{sec_2}

The Besov-Morrey spaces and also the Triebel-Lizorkin-Morrey spaces are function spaces that are built upon Morrey spaces. Because of this at first we want to recall the definition of the Morrey spaces. 

\begin{defi}\label{def_mor}

Let $ 1 \leq p \leq u < \infty$. Then the real Morrey space $ \mathbb{M}^{u}_{p}(\R)  $ is defined to be the set of all real-valued functions $ f \in \mathbb{L}_{p}^{loc}(\R) $ such that 
\begin{align*}
\Vert f \vert \mathcal{M}^{u}_{p}(\R) \Vert := \sup_{y \in \R, r > 0} \vert B(y,r) \vert^{\frac{1}{u}-\frac{1}{p}} \Big ( \int_{B(y,r)} \vert f(x) \vert^{p} dx      \Big )^{\frac{1}{p}} < \infty.
\end{align*} 

\end{defi}

The Morrey spaces are Banach spaces. They have many connections to the Lebesgue spaces. So for $ p \in [1,\infty) $ we have $ \mathbb{M}^{p}_{p}(\R) = \mathbb{L}_{p}(\R)$. Moreover for $ 1 \leq p_{2} \leq p_{1} \leq u < \infty $ we have
\begin{equation}\label{Lu_Mo_Em}
\mathbb{L}_{u}(\R) = \mathbb{M}^{u}_{u}(\R) \hookrightarrow   \mathbb{M}^{u}_{p_{1}}(\R)  \hookrightarrow  \mathbb{M}^{u}_{p_{2}}(\R).
\end{equation} 
In what follows we will need a so-called smooth  dyadic decomposition of the unity. Let $\varphi_0 \in C_0^{\infty}({\R})$ be a non-negative function such that $\varphi_0(x) = 1$ if $|x|\leq 1$ and $ \varphi_0 (x) = 0$ if $|x|\geq 3/2$. For $k\in \N$ we define $  \varphi_k(x) := \varphi_0(2^{-k}x)-\varphi_0(2^{-k+1}x) $. Then because of $ \sum_{k=0}^\infty \varphi_k(x) = 1 $ for all $ x\in \R $ and $ \supp \varphi_k \subset \{ x\in \R: \: 2^{k-1}\le |x|\le 3 \cdot 2^{k-1} \} $ with $ k \in \N $ we call the system $(\varphi_k)_{k\in \N_0 }$ a smooth dyadic decomposition of the unity on $\R$. Because of the Paley-Wiener-Schwarz theorem $\cfi[\varphi_{k}\, \cf f]$ with $k \in \N_0$ is a smooth function for all $f\in \cs'(\R)$. Now we are able to define the Besov-Morrey spaces and the Triebel-Lizorkin-Morrey spaces. 

\begin{defi}\label{def_tlm}

Let $ 1 \leq p \leq u < \infty $, $ 1 \leq q \leq \infty $ and $ s > 0 $. Let $ (\varphi_{k})_{k\in \N_0 }$ be a smooth dyadic decomposition of the unity.

\begin{itemize}
\item[(i)]

The real Besov-Morrey space $  \mathbb{N}^{s}_{u,p,q}(\mathbb{R}^{d}) $ is defined to be the set of all real-valued functions $ f \in \mathbb{L}_{1}^{loc}(\R) \cap \mathcal{S}'(\mathbb{R}^{d})  $ such that
\begin{align*} 
\Vert f \vert \mathcal{N}^{s}_{u,p,q}(\mathbb{R}^{d})  \Vert :=  \Big ( \sum_{k = 0}^{\infty} 2^{ksq}   \Vert \mathcal{F}^{-1}[\varphi_{k} \mathcal{F}f]  \vert \mathcal{M}^{u}_{p}(\R)   \Vert  ^{q} \Big   )^{\frac{1}{q}} < \infty .
\end{align*}

In the case $ q = \infty $ the usual modifications are made.

\item[(ii)]

The real Triebel-Lizorkin-Morrey space $  \mathbb{E}^{s}_{u,p,q}(\mathbb{R}^{d}) $ is defined to be the set of all real-valued functions $ f \in\mathbb{L}_{1}^{loc}(\R) \cap  \mathcal{S}'(\mathbb{R}^{d})  $ such that
\begin{align*} 
\Vert f \vert \mathcal{E}^{s}_{u,p,q}(\mathbb{R}^{d})  \Vert := \Big  \Vert \Big ( \sum_{k = 0}^{\infty} 2^{ksq} \vert \mathcal{F}^{-1}[\varphi_{k} \mathcal{F}f](x) \vert ^{q} \Big   )^{\frac{1}{q}} \Big \vert \mathcal{M}^{u}_{p}(\R)   \Big \Vert < \infty .
\end{align*}

In the case $ q = \infty $ the usual modifications are made.

\end{itemize}

\end{defi}

\begin{rem}\label{rem_BM_TLM_rf}
In the literature often the Besov-Morrey spaces and the Triebel-Lizorkin-Morrey spaces are defined in a more general way, see for example \cite{TangXu} or section 1.3.3 in \cite{ysy}. So usually for the parameters we allow $ s \in \mathbb{R}  $, $ 0 < p \leq u < \infty $ and $ 0 < q \leq \infty $. Moreover instead of $ f \in \mathbb{L}_{1}^{loc}(\R) \cap \mathcal{S}'(\mathbb{R}^{d})   $ in the original definition we have $ f \in \mathcal{S}'(\mathbb{R}^{d})   $. These original Besov-Morrey spaces $ \mathcal{N}^{s}_{u,p,q}(\mathbb{R}^{d})  $ and Triebel-Lizorkin-Morrey spaces $ \mathcal{E}^{s}_{u,p,q}(\mathbb{R}^{d})  $ contain complex-valued functions and sometimes also singular distributions. In this paper we will investigate the properties of the operator $(T^{+}f)(x) = \max(f(x),0)$. $ T^{+} $ neither makes sense for complex-valued functions nor for singular distributions. Therefore we work with the real function spaces $ \mathbb{N}^{s}_{u,p,q}(\mathbb{R}^{d})  $ and $ \mathbb{E}^{s}_{u,p,q}(\mathbb{R}^{d})  $. The restrictions $ s > 0  $ and $ p \geq 1  $ make sure that our function spaces do not contain singular distributions, see theorem 3.3 in \cite{HaMoSk} and theorem 3.4 in \cite{HaMoSk2}.  
\end{rem}

In what follows sometimes we write $ \mathbb{A}^{s}_{u,p,q}(\mathbb{R}^{d})  $. Then we mean either $ \mathbb{N}^{s}_{u,p,q}(\mathbb{R}^{d})  $ or $ \mathbb{E}^{s}_{u,p,q}(\mathbb{R}^{d})  $. By $ \mathcal{A}^{s}_{u,p,q}(\mathbb{R}^{d})   $ we mean either $ \mathcal{N}^{s}_{u,p,q}(\mathbb{R}^{d})   $ or $ \mathcal{E}^{s}_{u,p,q}(\mathbb{R}^{d})   $. Of course we always have $ \mathbb{A}^{s}_{u,p,q}(\R) \subset  \mathcal{A}^{s}_{u,p,q}(\R)  $. Because of this many results that are known for the usual Besov-Morrey and Triebel-Lizorkin-Morrey spaces have obvious counterparts for the spaces $  \mathbb{N}^{s}_{u,p,q}(\mathbb{R}^{d}) $ and $ \mathbb{E}^{s}_{u,p,q}(\mathbb{R}^{d})  $. Hereinafter we want to collect some basic properties of the Besov-Morrey spaces and the Triebel-Lizorkin-Morrey spaces. Most of them will be used later.

\begin{lem}\label{l_bp1}
Let $ 1 \leq p \leq u < \infty $, $ 1 \leq q \leq \infty $ and $ s > 0 $. Then we know the following.

\begin{itemize}
\item[(i)]

The spaces $ \mathbb{A}^{s}_{u,p,q}(\mathbb{R}^{d})  $ are independent of the chosen smooth dyadic decomposition of the unity in the sense of equivalent norms. 

\item[(ii)]

The spaces $ \mathbb{A}^{s}_{u,p,q}(\mathbb{R}^{d})  $ are Banach spaces. 

\item[(iii)]

It holds $ \mathbb{S}(\R)   \hookrightarrow  \mathbb{A}^{s}_{u,p,q}(\mathbb{R}^{d})   \hookrightarrow   \mathcal{S}'(\mathbb{R}^{d})$.

\item[(iv)]

We have $ \mathbb{N}^{s}_{p,p,q}(\mathbb{R}^{d}) = \mathbb{B}^{s}_{p,q}(\R)  $ and $ \mathbb{E}^{s}_{p,p,q}(\mathbb{R}^{d}) = \mathbb{F}^{s}_{p,q}(\R)  $.



\end{itemize}

\end{lem} 

\begin{proof}

(i) was proved in \cite{TangXu}, see theorem 2.8. The proof of (ii) is standard, see corollary 2.6. in \cite{KoYa} and lemma 2.1 in \cite{ysy}. (iii) can be found in \cite{SawTan}, see theorem 3.2. and with slightly different formulation in \cite{ysy}, see proposition 2.3. (iv) is obvious. 
\end{proof}

In many cases the Besov-Morrey spaces and the Triebel-Lizorkin-Morrey spaces are embedded into the Morrey spaces. For us the following result will be important, see theorem 3.2 in \cite{HaMoSk2}. 

\begin{lem}\label{E_Mor_em}
Let $ 1 \leq p \leq u < \infty  $, $ s > 0 $ and $ 1 \leq q \leq \infty  $. Then there is the embedding $ \mathbb{E}^{s}_{u,p,q}(\mathbb{R}^{d}) \hookrightarrow \mathbb{M}^{u}_{p}(\R) $.
\end{lem} 

The spaces $ \mathbb{A}^{s}_{u,p,q}(\mathbb{R}^{d})  $ have the so-called Fatou property, see lemma 3.5 in \cite{SawTan}.

\begin{lem}\label{l_fatou}
Let $ 1 \leq p \leq u < \infty $, $ 1 \leq q \leq \infty $ and $ s > 0 $. Suppose that $ (f_{k})_{k \in \mathbb{N}_{0} }  $ is a bounded sequence in $ \mathbb{A}^{s}_{u,p,q}(\mathbb{R}^{d})  $. The limit $ f = \lim_{k \rightarrow \infty} f_{k}  $ exists in $ \mathcal{S}'(\mathbb{R}^{d}) $. Then we have $ f \in  \mathbb{A}^{s}_{u,p,q}(\mathbb{R}^{d}) $ and 
\begin{align*}
\Vert f \vert \mathcal{A}^{s}_{u,p,q}(\mathbb{R}^{d})  \Vert \leq C \sup_{k \in \mathbb{N}_{0}} \Vert f_{k} \vert \mathcal{A}^{s}_{u,p,q}(\mathbb{R}^{d})  \Vert.
\end{align*}

\end{lem}                      

Let $ j \in \{ 1, 2, \ldots , d \}  $ and $ m \in \mathbb{N} $. Then with $\partial^{m}_{j}f   $ we denote the derivative of order $ m $ in direction $ e_{j} $ of a distribution $f \in \mathcal{S}'(\mathbb{R}^{d}) $. Using this we can state the next result. 

\begin{lem}\label{l_deriv}
Let $ 1 \leq p \leq u < \infty $, $ 1 \leq q \leq \infty $, $ s > 0 $ and $ m \in \mathbb{N} $. Then there are constants $ C_{1}, C_{2} > 0  $ independent of $ f \in \mathbb{A}^{s+m}_{u,p,q}(\mathbb{R}^{d}) $ such that
\begin{align*}
C_{1} \Vert f \vert \mathcal{A}^{s + m}_{u,p,q}(\mathbb{R}^{d})  \Vert \leq \Vert f \vert \mathcal{A}^{s}_{u,p,q}(\mathbb{R}^{d})  \Vert + \sum_{j = 1}^{d} \Vert \partial^{m}_{j}f \vert \mathcal{A}^{s}_{u,p,q}(\mathbb{R}^{d})  \Vert \leq C_{2} \Vert f \vert \mathcal{A}^{s + m}_{u,p,q}(\mathbb{R}^{d})  \Vert.
\end{align*}
\end{lem} 

\begin{proof}
This result can be found in \cite{SawTan}, see corollary 3.4. One may also consult theorem 2.15 in \cite{TangXu}.
\end{proof}

For the spaces $ \mathbb{A}^{s}_{u,p,q}(\mathbb{R}^{d}) $ there exist some useful multiplier theorems. So on the one hand there is the following pointwise multiplier theorem. Let $ C(\R) $ be the space of all real-valued bounded uniformly continuous functions on $ \R $. For $  m \in \mathbb{N} $ we put 
\begin{align*}
C^{m}(\R) = \lbrace f : D^{\alpha}f \in C(\R) \ \forall \ \vert \alpha \vert \leq m \rbrace  \quad \mbox{and} \quad \Vert f \vert  C^{m}(\R)  \Vert = \sum_{\vert \alpha \vert \leq m}  \Vert D^{\alpha} f \vert L_{\infty}(\R) \Vert .
\end{align*}

\begin{lem}\label{l_bmD2}

Let $ 1 \leq p \leq u < \infty $, $ 1 \leq q \leq \infty $ and $ s > 0  $. Let $ m \in \mathbb{N} $ be sufficiently large. Then there exists a positive constant $ C(m) $ such that for all $ g \in C^{m}(\R) $ and for all $ f \in \mathbb{A}^{s}_{u,p,q}(\R) $ we have

\begin{center}
$\Vert f \cdot g \vert  \mathcal{A}^{s}_{u,p,q}(\R)   \Vert \leq C(m) \Vert g \vert  C^{m}(\R)  \Vert \; \Vert f  \vert  \mathcal{A}^{s}_{u,p,q}(\R)   \Vert.$
\end{center}

\end{lem}

\begin{proof}
This result can be found in \cite{HaSkSM}, see theorem 2.6. More details can be found in \cite{Saw}. A related result can be found in \cite{ysy}, see theorem 6.1.      
\end{proof}

On the other hand there is the following Fourier multiplier theorem.

\begin{lem}\label{Four_Mult}
Let $ s > 0  $, $ 1 \leq p \leq u < \infty   $ and $ 1 \leq q \leq \infty  $. If $ m \in \mathbb{N}  $ is sufficiently large, then there exists a constant $ C > 0  $ such that for all $ g \in C^{\infty}(\R)   $ and $ f \in \mathbb{A}^{s}_{u,p,q}(\R)  $ we have
\begin{align*}
\Vert  \mathcal{F}^{-1}[g \mathcal{F}f]   \vert  \mathcal{A}^{s}_{u,p,q}(\R)   \Vert \leq C  \sup_{\vert \gamma  \vert \leq m } \sup_{x \in \R} ( 1 + \vert x \vert^{2})^{\frac{ \vert \gamma \vert}{2}} \vert D^{\gamma} g(x)   \vert   \   \    \Vert  f  \vert  \mathcal{A}^{s}_{u,p,q}(\R)   \Vert .
\end{align*}
\end{lem}

\begin{proof}
To prove this result we follow the proof of theorem 2.3.7. in \cite{Tr83}. Here the assertion was proved for the special case $ p = u  $. Fortunately almost everything that is done there also can be used for $p < u  $. Instead of formula (2.3.6.20) from \cite{Tr83} we apply proposition 2.12 from \cite{TangXu}. Then the desired result follows in the same way as is \cite{Tr83}. 
\end{proof}

\begin{rem}\label{rem_Fo_mul}
Notice that a result similar to lemma \ref{Four_Mult} also can be found in \cite{TangXu}, see proposition 2.14. For the case $ p > 1 $ an alternative proof of lemma \ref{Four_Mult} also can be found in \cite{Tr14}, see formula (3.259) from chapter 3.5.2. and theorem 3.50. One may also consult theorem 4.1 in \cite{yyMF} or theorem 3 in \cite{Mah et al}.  
\end{rem}

With the help of real interpolation we can find the following connection between Morrey spaces, Besov-Morrey spaces and Triebel-Lizorkin-Morrey spaces.

\begin{lem}\label{re_inter}
Let $ 0 < \theta < 1 $, $ s_{1} > 0  $, $ 1 \leq p \leq u < \infty   $ and $ 1 \leq q, q_{1} \leq \infty   $. Then we have
\begin{align*}
\mathbb{N}^{\theta s_{1}}_{u,p,q}(\mathbb{R}^{d}) = \Big ( \mathbb{M}^{u}_{p}(\R) ,  \mathbb{E}^{ s_{1}}_{u,p,q_{1}}(\mathbb{R}^{d})  \Big )_{\theta , q}
\end{align*}
in the sense of equivalent norms.
\end{lem}

\begin{proof}
This result was proved in \cite{Si_sur2}, see corollary 2.3. The general background concerning interpolation theory can be found in \cite{Tr78} and \cite{BL}.
\end{proof}

In many cases it is possible to describe the Besov-Morrey spaces and the Triebel-Lizorkin-Morrey spaces in terms of differences. Let $ f : \mathbb{R}^d \rightarrow \mathbb{R}$ be a function. Then for $ x, h \in \mathbb{R}^d $ we define the difference of the first order by $ \Delta_{h}^{1}f (x) := f ( x + h ) - f (x) $. Let $ N \in \mathbb{N}$ with $ N > 1 $. Then we define the difference of the order $ N $ by $\Delta_{h}^{N}f (x) :=   (\Delta_{h} ^1  ( \Delta_{h} ^{N-1}f   ) ) (x) $. Using this notation we can formulate the following very useful result for the spaces $  \mathbb{N}^{s}_{u,p,q}(\R)    $.

\begin{prop}\label{pro_N_diff}
Let $ 1 \leq p \leq u < \infty $ and $ 1 \leq q \leq \infty $. Let $ 1 \leq v \leq \infty $ and $ s ~ > ~ d \max  ( 0, 1/p - 1/v   ) $. Let $ 1 \leq a \leq \infty  $ and $ N \in \mathbb{N} $ with $ N > s $. Then a function $ f~ \in ~ \mathbb{L}_{p}^{loc}(\R)$ belongs to $ \mathbb{N}^{s}_{u,p,q}(\R)   $ if and only if $ f \in \mathbb{L}_{v}^{loc}(\R)$ and 
\begin{align*}
& \underbrace{\Vert f \vert \mathcal{M}^{u}_{p}( \mathbb{R}^d)   \Vert + \Big ( \int_{0}^{a} t^{-sq-d \frac{q}{v}} \Big \Vert  \Big ( \int_{B(0,t)}\vert \Delta^{N}_{h}f(x) \vert^{v} dh \Big )^{\frac{1}{v}} \Big \vert  \mathcal{M}^{u}_{p}( \mathbb{R}^d) \Big \Vert^{q}  \frac{dt}{t} \Big )^{\frac{1}{q}}  } < \infty  \\
& \hspace{4,5 cm} \Vert f \vert \mathcal{N}^{s}_{u,p,q}(\R) \Vert^{(v,a)} :=  
\end{align*}
with modifications if $ q = \infty $ and/or $ v = \infty$. The norms $ \Vert f \vert \mathcal{N}^{s}_{u,p,q}(\R) \Vert  $ and $ \Vert f \vert \mathcal{N}^{s}_{u,p,q}(\R) \Vert^{(v,a)} $  are equivalent for $ f \in \mathbb{L}_{p}^{loc}(\R)$.

\end{prop}

\begin{proof}
This result was proved in \cite{HoBM}, see theorem 3.
\end{proof}

For the Triebel-Lizorkin-Morrey spaces there is the following characterisation in terms of differences. 

\begin{prop}\label{pro_E_diff}

Let $ 1 \leq p \leq u < \infty $, $ 1 \leq q \leq \infty $, $ 1 \leq v \leq \infty $, $ N \in \mathbb{N} $ and $   1 \leq a \leq \infty $. Let $ d \max  ( 0, 1/p - 1/v, 1/q - 1/v  ) < s < N $. Then a function $ f \in \mathbb{L}_{\min(p,q)}^{loc}(\R)$ belongs to $ \mathbb{E}^{s}_{u,p,q}(\R)   $ if and only if $ f \in \mathbb{L}_{v}^{loc}(\R)$ and (modifications if $ q = \infty $ and/or $ v = \infty $)  
\begin{align*}
& \underbrace{\Vert f \vert \mathcal{M}^{u}_{p}(\R)  \Vert  +   \Big \Vert \Big (  \int_{0}^{a}  t^{-sq} \Big ( t^{-d} \int_{B(0,t)}\vert \Delta^{N}_{h}f(x) \vert^{v} dh \Big )^{\frac{q}{v}}  \frac{dt}{t}  \Big )^{\frac{1}{q}} \Big \vert \mathcal{M}^{u}_{p}( \mathbb{R}^d)  \Big \Vert } < \infty. \\
& \hspace{4cm} \Vert f \vert \mathcal{E}^{s}_{u,p,q}(\R) \Vert^{(v,a)} := 
\end{align*}
The norms $ \Vert f \vert \mathcal{E}^{s}_{u,p,q}(\R) \Vert  $ and $ \Vert f \vert \mathcal{E}^{s}_{u,p,q}(\R) \Vert^{(v,a)}    $  are equivalent for $ f ~ \in ~  \mathbb{L}_{\min(p,q)}^{loc}(\R)$.

\end{prop}

\begin{proof}
This result was proved in \cite{HoTLM}, see theorem 7.
\end{proof}

\section{On the boundedness of the operators $ T^{+} $ and $ T $ }\label{sec_3}

Now for real-valued $ f $ we are ready to turn our attention to the operators $ T^{+}  $ and $ T $ given by 
\begin{equation}\label{op_T+}
(T^{+}f)(x) = \max(f(x),0) \qquad \mbox{and} \qquad  (Tf)(x) = |f(x)|
\end{equation}
in the context of Besov-Morrey and Triebel-Lizorkin-Morrey spaces. Both operators are representatives of composition operators $ T_{g} : f \mapsto g \circ f   $. Theory concerning composition operators or even more general Nemytzkij operators can be found in chapter 5 in \cite{RS} and in \cite{AZ}. One may also consult \cite{BouLan}, \cite{BouMou}, \cite{BouMouSi} and \cite{BouSi}. In the theory of nonlinear partial differential equations the operator $ T^{+} $ often is called a truncation operator. It turns out that the operators $ T^{+}  $ and $ T $ have many properties in common. The reason for this is that $ 2 \max(f(x),0) =  f(x) +   |f(x)|  $ for real-valued functions $ f $. Because of this formula we learn that whenever $ T $ is bounded on $ \mathbb{A}^{s}_{u,p,q}(\R)   $ also $ T^{+}  $ is bounded on $ \mathbb{A}^{s}_{u,p,q}(\R)   $. Moreover it is not difficult to see that when we know whether the operator $ T $ is bounded on a space $ \mathbb{A}^{s}_{u,p,q}(\R)  $ or not we also know whether we have $ T(\mathbb{A}^{s}_{u,p,q}(\R)) \subset \mathbb{A}^{s}_{u,p,q}(\R)    $ or $ T(\mathbb{A}^{s}_{u,p,q}(\R)) \not \subset \mathbb{A}^{s}_{u,p,q}(\R)    $. So in what follows it will be enough to investigate the boundedness properties of $ T $. Results concerning the boundedness of the mapping $ f \mapsto |f|  $ in the setting of the original Besov and Triebel-Lizorkin spaces can be found in \cite{BouMey} as well as in \cite{Os1}, in chapter 5.4.1 in \cite{RS} and in chapter 25 in \cite{Tr01}. Early results for Sobolev spaces can be found in \cite{MaMiz}.

\subsection{The boundedness of the operator $ T $ in the case $ 0 < s < 1  $}\label{sec_3.1}

For a start we will look at the case $ 0 < s < 1   $. Here it is very easy to prove that the operator $T$ is bounded.

\begin{prop}\label{pro_0<s<1}
Let $ 1 \leq p \leq u < \infty $, $ 1 \leq q \leq \infty $ and $ 0 < s < 1  $. Then there is a constant $ C > 0 $ independent of $ f \in \mathbb{A}^{s}_{u,p,q}(\R)  $ such that 
\begin{align*}
\Vert Tf   \vert \mathcal{A}^{s}_{u,p,q}(\mathbb{R}^{d})  \Vert \leq C \Vert f \vert \mathcal{A}^{s}_{u,p,q}(\mathbb{R}^{d})  \Vert
\end{align*}
holds for all $ f \in  \mathbb{A}^{s}_{u,p,q}(\R)  $.
\end{prop}

\begin{proof}
\textit{Step 1.} At first we will deal with the case $ \mathcal{A} = \mathcal{E}  $. Because of $ 0 < s < 1 $ and $ p, q \geq 1  $ we can use proposition \ref{pro_E_diff} with $ v = 1  $, $ a = \infty  $ and $ N = 1  $. So we have to work with 
\begin{align*}
 \Vert \ |f| \ \vert \mathcal{M}^{u}_{p}(\R)  \Vert  +   \Big \Vert \Big (  \int_{0}^{\infty}  t^{-sq} \Big ( t^{-d} \int_{B(0,t)}\vert \Delta^{1}_{h}|f|(x) \vert dh \Big )^{q}  \frac{dt}{t}  \Big )^{\frac{1}{q}} \Big \vert \mathcal{M}^{u}_{p}( \mathbb{R}^d)  \Big \Vert. 
\end{align*}
Now because of the triangle inequality we have 
\begin{align*}
\vert \Delta^{1}_{h}|f|(x) \vert  = \vert |f(x+h)| - |f(x)| \vert  \leq \vert f(x+h) - f(x)   \vert  =  \vert \Delta^{1}_{h}f(x) \vert. 
\end{align*}
When we use proposition \ref{pro_E_diff} again the proof for the case $ \mathcal{A} = \mathcal{E}   $ is complete.

\textit{Step 2.} In the case $ \mathcal{A} = \mathcal{N}   $ the proof can be done in the same way. Here instead of proposition \ref{pro_E_diff} we have to use proposition \ref{pro_N_diff}. We omit the details. 
\end{proof}

\subsection{On the boundedness of the operator $ T $ in the case $ s > 1  $ and $ d = 1  $ }\label{sec_3.2}

In this section we will study the properties of the mapping $ T : f \rightarrow |f|  $ for functions from Triebel-Lizorkin-Morrey spaces in the case $ s > 1   $ and dimension $  d = 1 $. To this end the following Hardy-type inequality will be an important tool. For $ x \in \mathbb{R}^{d}   $ and a set $ A \subset \mathbb{R}^{d}  $ we will write $ \dist(x,A) = \inf_{y \in A} \vert x - y  \vert  $. By $ A^{c} $ we mean $ \mathbb{R}^{d} \setminus A   $.

\begin{lem}\label{Hardy}
Let $ 1 \leq p < \infty  $, $ 1 \leq q \leq \infty  $ and $ 0 < s < 1/p  $. Let $ d = 1$. Then there exists a constant $ C > 0 $ such that
\begin{align*}
\int_{I} |f(x)|^{p} \dist(x,I^{c})^{-sp} dx \leq C \int_{I} \Big ( \int_{0}^{\infty} r^{-sq} \Big ( \int_{\substack{-1 < h < 1 \\ x-rh \in I}} \vert f(x) - f(x-rh)\vert dh \Big )^{q} \frac{dr}{r} \Big )^{\frac{p}{q}} dx
\end{align*}
holds for all intervals $ I $ and all $ f \in  \mathbb{S}(\mathbb{R}) $ satisfying $ \int_{I} f(x) dx = 0  $ if $ I $ is bounded.

\end{lem}

\begin{proof}
This result can be found in \cite{BouMey}, see lemma 1. One may also consult \cite{Bou93} or chapter 3.1 in \cite{mir}. A detailed proof can be found in \cite{RS}, see lemma 1 in chapter 5.4.1.
\end{proof}

It is also possible to prove a version of the Hardy-type inequality for Morrey spaces. 

\begin{lem}\label{Hardy_Mor}
Let $ 1 \leq p \leq u < \infty  $, $ 1 \leq q \leq \infty  $ and $ 0 < s < 1/u  $. Let $ d = 1$. Then there exists a constant $ C > 0 $ such that
\begin{align*}
& \sup_{\substack{a,b \in \mathbb{R} \\ a<b }} \vert a - b \vert^{\frac{1}{u}-\frac{1}{p}} \Big (  \int_{I \cap (a,b)} |f(x)|^{p} \dist(x,I^{c})^{-sp} dx \Big )^{\frac{1}{p}} \\
& \qquad \leq C \sup_{\substack{a,b \in \mathbb{R} \\ a<b }} \vert a - b \vert^{\frac{1}{u}-\frac{1}{p}} \Big ( \int_{I \cap (a,b)} \Big ( \int_{0}^{\infty} r^{-sq} \Big ( \int_{\substack{-1 < h < 1 \\ x-rh \in I}} \vert f(x) - f(x-rh)\vert dh \Big )^{q} \frac{dr}{r} \Big )^{\frac{p}{q}} dx \Big )^{\frac{1}{p}}
\end{align*}
holds for all intervals $ I $ and all $ f \in  \mathbb{S}(\mathbb{R}) $ satisfying $ \int_{I} f(x) dx = 0  $ if $ I $ is bounded.

\end{lem}

\begin{proof}
To prove this result we can use the methods that are described in the proof of lemma 1 from chapter 5.4.1. in \cite{RS}, see also lemma 1 in \cite{BouMey} and \cite{Bou93}. Only a few modifications have to be made.

\textit{Step 1.} At first we look at the case $ I = (0,\infty) $. For $ x > 0 $ we put 
\begin{align*}
g(x) = \frac{1}{x} \int_{0}^{x} \Big ( f(x) - f(y)   \Big ) dy \qquad  \mbox{and}  \qquad h(x) = g(x) - \int_{x}^{\infty} g(y) \frac{dy}{y}.
\end{align*}
For these functions because of $ f  $ is smooth we observe
\begin{align*}
g'(x) = f'(x) - \Big ( - \frac{1}{x^2} \int_{0}^{x} f(y) dy + \frac{f(x)}{x}   \Big ) \qquad \mbox{and} \qquad h'(x) = g'(x) + \frac{g(x)}{x} = f'(x).
\end{align*}
Since $ f \in  \mathbb{S}(\mathbb{R})  $ we find $ \lim_{x \rightarrow \infty} f(x) = \lim_{x \rightarrow \infty} g(x) =  \lim_{x \rightarrow \infty} h(x) = 0   $ and therefore can conclude $ f = h $. When we use this identity we obtain
\begin{align*}
& \sup_{\substack{a,b \in \mathbb{R} \\ a<b }} \vert a - b \vert^{\frac{1}{u}-\frac{1}{p}} \Big (  \int_{(0,\infty) \cap (a,b)} |f(x)|^{p} x^{-sp} dx \Big )^{\frac{1}{p}} \\
& \qquad \leq \sup_{\substack{a,b \in \mathbb{R} \\ a<b }} \vert a - b \vert^{\frac{1}{u}-\frac{1}{p}} \Big (  \int_{(0,\infty) \cap (a,b)} |g(x)|^{p} x^{-sp} dx \Big )^{\frac{1}{p}}  \\
& \qquad \qquad + \sup_{\substack{a,b \in \mathbb{R} \\ a<b }} \vert a - b \vert^{\frac{1}{u}-\frac{1}{p}} \Big (  \int_{(0,\infty) \cap (a,b)} \Big \vert \int_{x}^{\infty} g(y) \frac{dy}{y}  \Big \vert^{p} x^{-sp} dx \Big )^{\frac{1}{p}}. 
\end{align*}
Now at first we look at the second term. We put $ y = x \xi  $ and get
\begin{align*}
& \sup_{\substack{a,b \in \mathbb{R} \\ a<b }} \vert a - b \vert^{\frac{1}{u}-\frac{1}{p}} \Big (  \int_{(0,\infty) \cap (a,b)} \Big \vert \int_{x}^{\infty} g(y) \frac{dy}{y}  \Big \vert^{p} x^{-sp} dx \Big )^{\frac{1}{p}} \\
& \qquad \leq \sup_{\substack{a,b \in \mathbb{R} \\ a<b }} \vert a - b \vert^{\frac{1}{u}-\frac{1}{p}} \int_{1}^{\infty} \xi^{-1} \Big ( \int_{(0,\infty) \cap (a,b)}  \vert g(x \xi) x^{-s} \vert^{p}  dx \Big )^{\frac{1}{p}} d \xi.
\end{align*}
We put $ x \xi = z  $ and obtain
\begin{align*}
& \sup_{\substack{a,b \in \mathbb{R} \\ a<b }} \vert a - b \vert^{\frac{1}{u}-\frac{1}{p}} \Big (  \int_{(0,\infty) \cap (a,b)} \Big \vert \int_{x}^{\infty} g(y) \frac{dy}{y}  \Big \vert^{p} x^{-sp} dx \Big )^{\frac{1}{p}} \\
& \qquad \leq \sup_{\substack{a,b \in \mathbb{R} \\ a<b }} \int_{1}^{\infty} \xi^{-1-\frac{1}{u}+s} \vert \xi a - \xi b \vert^{\frac{1}{u}-\frac{1}{p}}  \Big ( \int_{(0,\infty) \cap (\xi a, \xi b)}  \vert g(z) z^{-s}    \vert^{p}  dz \Big )^{\frac{1}{p}} d \xi .
\end{align*}
Next we have to use $  s < 1/u $. Then we find
\begin{align*}
& \sup_{\substack{a,b \in \mathbb{R} \\ a<b }} \vert a - b \vert^{\frac{1}{u}-\frac{1}{p}} \Big (  \int_{(0,\infty) \cap (a,b)} \Big \vert \int_{x}^{\infty} g(y) \frac{dy}{y}  \Big \vert^{p} x^{-sp} dx \Big )^{\frac{1}{p}} \\
& \qquad \leq \sup_{\substack{a,b \in \mathbb{R} \\ a<b }}  \sup_{1 \leq \rho \leq \infty } \vert \rho a - \rho b \vert^{\frac{1}{u}-\frac{1}{p}}  \Big ( \int_{(0,\infty) \cap (\rho a, \rho b)}  \vert g(z) z^{-s}    \vert^{p}  dz \Big )^{\frac{1}{p}}  \int_{1}^{\infty} \xi^{-1-\frac{1}{u}+s} d \xi   \\
& \qquad \leq C_{1}    \sup_{\substack{a,b \in \mathbb{R} \\ a<b }} \vert  a -  b \vert^{\frac{1}{u}-\frac{1}{p}}  \Big ( \int_{(0,\infty) \cap ( a,  b)}  \vert g(z) z^{-s}  \vert^{p}  dz \Big )^{\frac{1}{p}}.
\end{align*}
In what follows we will use the abbreviation
\begin{align*}
u(r,x) = \int_{\substack{-1 < y < 1 \\ x-ry > 0}} \vert f(x) - f(x-ry)\vert dy .
\end{align*}
Then because of
\begin{align*}
g(x) = \frac{y}{x} \int_{0}^{ \frac{x}{y}  } \Big ( - f(x) + f (x - ry)    \Big ) \  dr \ , \qquad \qquad y \not = 0 \ ,
\end{align*}
we can apply H\"older's inequality to get
\begin{align*}
|g(x)| & = 2 \Big | \int_{\frac{1}{2}}^{1} \frac{y}{x} \int_{0}^{ \frac{x}{y}  } \Big ( - f(x) + f (x - ry)    \Big ) \  dr \   dy   \Big | \leq C_{2} \Big ( \int_{0}^{2x} |u(r,x)|^{q}  x^{-1} dr \Big )^{\frac{1}{q}}.
\end{align*}
Using this estimate we obtain
\begin{align*}
& \sup_{\substack{a,b \in \mathbb{R} \\ a<b }} \vert a - b \vert^{\frac{1}{u}-\frac{1}{p}} \Big (  \int_{(0,\infty) \cap (a,b)} |g(x)|^{p} x^{-sp} dx \Big )^{\frac{1}{p}}  \\
& \qquad \leq C_{3} \sup_{\substack{a,b \in \mathbb{R} \\ a<b }} \vert a - b \vert^{\frac{1}{u}-\frac{1}{p}} \Big (  \int_{(0,\infty) \cap (a,b)}  \Big ( \int_{0}^{2x} |u(r,x)|^{q}  x^{-1-sq} dr \Big )^{\frac{p}{q}}   dx \Big )^{\frac{1}{p}}  \\
& \qquad \leq C_{4} \sup_{\substack{a,b \in \mathbb{R} \\ a<b }} \vert a - b \vert^{\frac{1}{u}-\frac{1}{p}} \Big (  \int_{(0,\infty) \cap (a,b)}  \Big ( \int_{0}^{\infty} |u(r,x)|^{q}  r^{-1-sq} dr \Big )^{\frac{p}{q}}   dx \Big )^{\frac{1}{p}}.
\end{align*}
In the last step we used $ r \leq 2x $. So step 1 of the proof is complete.

\textit{Step 2.}
Next we look at the case $ I = ( 0 , 1 ) $. At first we observe
\begin{align*}
& \sup_{\substack{a,b \in \mathbb{R} \\ a<b }} \vert a - b \vert^{\frac{1}{u}-\frac{1}{p}} \Big (  \int_{(0,1) \cap (a,b)} |f(x)|^{p} \dist(x,(0,1)^{c})^{-sp} dx \Big )^{\frac{1}{p}} \\
& \qquad \leq \sup_{\substack{a,b \in \mathbb{R} \\ a<b }} \vert a - b \vert^{\frac{1}{u}-\frac{1}{p}} \Big (  \int_{(0,1) \cap (a,b)} |f(x)|^{p} x^{-sp} dx \Big )^{\frac{1}{p}} \\
& \qquad \qquad +  \sup_{\substack{a,b \in \mathbb{R} \\ a<b }} \vert a - b \vert^{\frac{1}{u}-\frac{1}{p}} \Big (  \int_{(0,1) \cap (a,b)} |f(1-y)|^{p} y^{-sp} dy \Big )^{\frac{1}{p}}.
\end{align*}
Notice that we have $ f(1 - \cdot ) \in  \mathbb{S}(\mathbb{R})  $ and $ \int_{0}^{1} f(1-y) dy = \int_{0}^{1} f(y) dy = 0  $. So we can proceed for both terms simultaneously. Thanks to a transformation of the coordinates at the end of the calculations we will do now we can see that it is enough to deal with
\begin{align*}
\sup_{\substack{a,b \in \mathbb{R} \\ a<b }} \vert a - b \vert^{\frac{1}{u}-\frac{1}{p}} \Big (  \int_{(0,1) \cap (a,b)} |f(x)|^{p} x^{-sp} dx \Big )^{\frac{1}{p}} .
\end{align*}
Now let $ \eta \in C_{0}^{\infty}(\mathbb{R})  $ be a cut-off function with $ 0 \leq \eta(x) \leq 1  $ for all $x$ and $ \eta(x) = 1 $ if $ 0 \leq x \leq 1/2  $ and $ \supp \eta \subset [-\frac{1}{4},\frac{3}{4}]   $. Then we find 
\begin{align*}
& \sup_{\substack{a,b \in \mathbb{R} \\ a<b }} \vert a - b \vert^{\frac{1}{u}-\frac{1}{p}} \Big (  \int_{(0,1) \cap (a,b)} |f(x)|^{p} x^{-sp} dx \Big )^{\frac{1}{p}}  \\
& \qquad \leq \sup_{\substack{a,b \in \mathbb{R} \\ a<b }} \vert a - b \vert^{\frac{1}{u}-\frac{1}{p}} \Big (  \int_{(0,\frac{1}{2}) \cap (a,b)} |f(x) \eta(x)|^{p} x^{-sp} dx \Big )^{\frac{1}{p}} \\
& \qquad \qquad + C_{1}  \sup_{\substack{a,b \in \mathbb{R} \\ a<b }} \vert a - b \vert^{\frac{1}{u}-\frac{1}{p}} \Big (  \int_{(\frac{1}{2},1) \cap (a,b)} |f(x)|^{p} dx \Big )^{\frac{1}{p}}.
\end{align*}
Here for the first term because of $ (0,\frac{1}{2}) \subset (0 , \infty)  $ we can use the result from step 1. When we introduce the abbreviations $ J_{1}   $ and $ J_{2}   $ like it is done below we obtain
\begin{align*}
& \sup_{\substack{a,b \in \mathbb{R} \\ a<b }} \vert a - b \vert^{\frac{1}{u}-\frac{1}{p}} \Big (  \int_{(0,1) \cap (a,b)} |f(x)|^{p} x^{-sp} dx \Big )^{\frac{1}{p}}  \\
& \qquad \leq C_{1}  \sup_{\substack{a,b \in \mathbb{R} \\ a<b }} \vert a - b \vert^{\frac{1}{u}-\frac{1}{p}} \Big (  \int_{(0,1) \cap (a,b)} |f(x)|^{p} dx \Big )^{\frac{1}{p}}  + C_{2} \sup_{\substack{a,b \in \mathbb{R} \\ a<b }} \vert a - b \vert^{\frac{1}{u}-\frac{1}{p}} \\
& \qquad  \Big ( \int_{(0,\infty) \cap (a,b)} \Big ( \int_{0}^{\infty} r^{-sq} \Big ( \int_{\substack{-1 < h < 1 \\ x-rh > 0}} \vert f(x) \eta(x) - f(x-rh) \eta(x-rh)\vert dh \Big )^{q} \frac{dr}{r} \Big )^{\frac{p}{q}} dx \Big )^{\frac{1}{p}} \\
& \qquad = C_{3} ( J_{1} + J_{2}  ).
\end{align*}
If we replace $ f $ by $ f \chi_{[0,1]} $ we can split up $ J_{2} $ in the following way:
\begin{align*}
& J_{2}  \leq C_{4} \sup_{\substack{a,b \in \mathbb{R} \\ a<b }} \vert a - b \vert^{\frac{1}{u}-\frac{1}{p}} \Big ( \int_{ \substack{(0,1) \\ \cap (a,b)}} \Big ( \int_{0}^{\infty} r^{-sq} \Big ( \int_{\substack{-1 < h < 1 \\ 0 < x-rh < 1}} \vert f(x)  - f(x-rh) \vert dh \Big )^{q} \frac{dr}{r} \Big )^{\frac{p}{q}} dx \Big )^{\frac{1}{p}} \\
&  + C_{4} \sup_{\substack{a,b \in \mathbb{R} \\ a<b }} \vert a - b \vert^{\frac{1}{u}-\frac{1}{p}}  \Big ( \int_{ \substack{(1,\infty) \\ \cap (a,b)}} \Big ( \int_{0}^{\infty} r^{-sq} \Big ( \int_{\substack{-1 < h < 1 \\ 0 < x-rh < \frac{3}{4}}} \vert f(x-rh) \eta(x-rh)\vert dh \Big )^{q} \frac{dr}{r} \Big )^{\frac{p}{q}} dx \Big )^{\frac{1}{p}} \\ 
& + C_{4} \sup_{\substack{a,b \in \mathbb{R} \\ a<b }} \vert a-b \vert^{\frac{1}{u}-\frac{1}{p}}  \Big ( \int_{\substack{(0,\infty) \\ \cap (a,b)}} \Big ( \int_{0}^{\infty}  \Big ( \int_{\substack{-1 < h < 1 \\ x-rh > 0}} \vert f \chi_{[0,1]}(x) \vert  \vert  \Delta^{1}_{rh} \eta(x-rh)\vert dh \Big )^{q} \frac{dr}{r^{sq+1}} \Big )^{\frac{p}{q}} dx \Big )^{\frac{1}{p}} \\
& = C_{4} ( J_{21} + J_{22} + J_{23}  ).
\end{align*}
Now $ J_{21}  $ is what we want to have. If we use $ 0 < s < 1  $ and that $ \eta $ is smooth like in the proof of lemma 1 from chapter 5.4.1. in \cite{RS} for $ J_{23} $ we find
\begin{align*}
J_{23} \leq C_{5} \sup_{\substack{a,b \in \mathbb{R} \\ a<b }} \vert a-b \vert^{\frac{1}{u}-\frac{1}{p}}  \Big ( \int_{(0,1)  \cap (a,b)}  \vert f (x) \vert^{p} dx \Big )^{\frac{1}{p}}. 
\end{align*}
For $ J_{22} $ with $ x - rh = z  $ we obtain
\begin{align*}
J_{22} & \leq \sup_{\substack{a,b \in \mathbb{R} \\ a<b }} \vert a - b \vert^{\frac{1}{u}-\frac{1}{p}}  \Big ( \int_{(1,\infty) \cap (a,b)} \Big ( \int_{x - \frac{3}{4}}^{\infty} r^{-sq} \Big ( \int_{\substack{-1 < h < 1 \\ \frac{x}{r} - \frac{3}{4r}  < h < \frac{x}{r} }} \vert f\chi_{[0,1]}(x-rh) \vert dh \Big )^{q} \frac{dr}{r} \Big )^{\frac{p}{q}} dx \Big )^{\frac{1}{p}} \\ 
& \qquad \leq \sup_{\substack{a,b \in \mathbb{R} \\ a<b }} \vert a - b \vert^{\frac{1}{u}-\frac{1}{p}}  \Big ( \int_{(1,\infty) \cap (a,b)} \Big ( \int_{x - \frac{3}{4}}^{\infty} r^{-sq-q-1}  dr \Big )^{\frac{p}{q}} dx \Big )^{\frac{1}{p}} \Big ( \int_{0}^{1} \vert f(z) \vert dz \Big ) \\ 
& \qquad \leq C_{6} \sup_{\substack{a,b \in \mathbb{R} \\ a<b }} \vert a - b \vert^{\frac{1}{u}-\frac{1}{p}}  \Big ( \int_{(\frac{1}{4},\infty) \cap (a,b)}  x^{-(s+1)p} dx \Big )^{\frac{1}{p}} \Big ( \int_{0}^{1} \vert f(z) \vert dz \Big ). 
\end{align*}
Now because of $ 0 < s < 1  $ and $ u \geq 1  $ we find $ x^{-(s+1)} \chi_{(\frac{1}{4},\infty)}(x) \in \mathbb{L}_{u}(\mathbb{R})   $. With $  \mathbb{L}_{u}(\mathbb{R})  \hookrightarrow   \mathbb{M}^{u}_{p}(\mathbb{R}) $ we get $ x^{-(s+1)} \chi_{(\frac{1}{4},\infty)}(x) \in \mathbb{M}^{u}_{p}(\mathbb{R}) $. So H\"older's inequality yields
\begin{align*}
J_{22}  \leq C_{7} \Big ( \int_{0}^{1} \vert f(z) \vert^{p} dz \Big )^{\frac{1}{p}} \leq C_{7} \sup_{\substack{a,b \in \mathbb{R} \\ a<b }} \vert a-b \vert^{\frac{1}{u}-\frac{1}{p}}  \Big ( \int_{(0,1)  \cap (a,b)}  \vert f (z) \vert^{p} dz \Big )^{\frac{1}{p}}. 
\end{align*}
Next we use $ \int_{0}^{1} f(x) dx = 0  $. This leads to
\begin{align*}
& \sup_{\substack{a,b \in \mathbb{R} \\ a<b }} \vert a-b \vert^{\frac{1}{u}-\frac{1}{p}}  \Big ( \int_{(0,1)  \cap (a,b)}  \vert f (z) \vert^{p} dz \Big )^{\frac{1}{p}} \\
& \qquad = \sup_{\substack{a,b \in \mathbb{R} \\ a<b }} \vert a-b \vert^{\frac{1}{u}-\frac{1}{p}}  \Big ( \int_{(0,1)  \cap (a,b)} \Big \vert f (z) - \int_{0}^{1} f(y) dy \Big \vert^{p} dz \Big )^{\frac{1}{p}} \\
& \qquad \leq \sup_{\substack{a,b \in \mathbb{R} \\ a<b }} \vert a-b \vert^{\frac{1}{u}-\frac{1}{p}}  \Big ( \int_{(0,1)  \cap (a,b)} \Big ( \int_{0}^{1} \vert  f (z) -  f(y) \vert dy \Big )^{p} dz \Big )^{\frac{1}{p}} \\
& \qquad \leq C_{8} \sup_{\substack{a,b \in \mathbb{R} \\ a<b }} \vert a-b \vert^{\frac{1}{u}-\frac{1}{p}}  \Big ( \int_{(0,1)  \cap (a,b)} \Big ( \int_{1}^{2} r^{-s-\frac{1}{q}} \Big ( \int_{0}^{1} \vert  f (z) -  f(y) \vert dy \Big ) dr \Big )^{p} dz \Big )^{\frac{1}{p}}.
\end{align*}
When we use $ q \geq 1 $ and $ y = z - rh $ we find
\begin{align*}
& \sup_{\substack{a,b \in \mathbb{R} \\ a<b }} \vert a-b \vert^{\frac{1}{u}-\frac{1}{p}}  \Big ( \int_{(0,1)  \cap (a,b)}  \vert f (z) \vert^{p} dz \Big )^{\frac{1}{p}} \\
&  \leq C_{9} \sup_{\substack{a,b \in \mathbb{R} \\ a<b }} \vert a-b \vert^{\frac{1}{u}-\frac{1}{p}}  \Big ( \int_{(0,1)  \cap (a,b)} \Big ( \int_{1}^{2} r^{-sq-1} \Big ( \int_{0}^{1} \vert  f (z) -  f(y) \vert dy \Big )^{q} dr \Big )^{\frac{p}{q}} dz \Big )^{\frac{1}{p}} \\
&  \leq C_{10} \sup_{\substack{a,b \in \mathbb{R} \\ a<b }} \vert a-b \vert^{\frac{1}{u}-\frac{1}{p}}  \Big ( \int_{\substack{ (0,1)  \\ \cap  (a,b)}} \Big ( \int_{0}^{\infty} r^{-sq-1} \Big ( \int_{\substack{-1 < h < 1 \\ 0 < z-rh < 1}} \vert  f (z) -  f(z - rh) \vert dh \Big )^{q} dr \Big )^{\frac{p}{q}} dz \Big )^{\frac{1}{p}} .
\end{align*}
So this step of the proof is complete.

\textit{Step 3.}
At last we have to deal with any other interval $ I = (c,d)   $ with $ -\infty \leq c < d \leq \infty $. But then the desired inequality follows from the model cases $ I = (0, \infty)  $ and $ I = (0,1)   $ and an appropriate transformation of the coordinates. So the proof is complete.   
\end{proof}

Now we are well-prepared to prove the following result for the Triebel-Lizorkin-Morrey spaces in the case of dimension $ d = 1  $. 

\begin{prop}\label{pro_d1s>1}
Let $ 1 \leq p \leq u < \infty  $, $ 1 \leq q \leq \infty  $ and $ 1 < s < 1 + 1/u  $. Let $ d = 1$. Then there is a constant $ C > 0 $ independent of $ f \in \mathbb{E}^{s}_{u,p,q}(\mathbb{R})  $ such that we have 
\begin{equation}\label{d1s>1_eq1}
\Vert Tf  \vert \mathcal{E}^{s}_{u,p,q}(\mathbb{R})  \Vert \leq C \Vert f \vert \mathcal{E}^{s}_{u,p,q}(\mathbb{R})  \Vert
\end{equation}
for all $ f \in \mathbb{E}^{s}_{u,p,q}(\mathbb{R})   $.
\end{prop}

\begin{proof}
To prove this result we follow the ideas from the proof of the theorem in chapter 5.4.1. in \cite{RS}, see also theorem 1 in \cite{BouMey}.

\textit{Step 1.} At first we prove \eqref{d1s>1_eq1} for real-valued $ f \in C_{0}^{\infty}(\mathbb{R}) $. We use lemma \ref{l_deriv} with $ m = 1  $. Then we find
\begin{align*}
\Vert \  |f| \ \vert \mathcal{E}^{s}_{u,p,q}(\mathbb{R})  \Vert \leq C_{1} \Vert \  |f| \ \vert \mathcal{E}^{s-1}_{u,p,q}(\mathbb{R})  \Vert + C_{1} \Vert \partial^{1}  |f| \  \vert \mathcal{E}^{s-1}_{u,p,q}(\mathbb{R})  \Vert.
\end{align*}
In general $ \partial^{1}  |f|  $ is a distributional derivative. Since $ f \in C_{0}^{\infty}(\mathbb{R})  $ we find that $ |f|  $ is a Lipschitz continuous function. So the classical derivative exists almost everywhere and coincides with the distributional one almost everywhere. Hence in our case we can also understand $ \partial^{1}  |f|  $ as a classical derivative. Let us look at $ \Vert \  |f| \ \vert \mathcal{E}^{s-1}_{u,p,q}(\mathbb{R})  \Vert $. Because of $  1 < s < 1 + 1/u < 2  $ we have $ 0 < s - 1 < 1  $. So we can apply proposition \ref{pro_0<s<1} and obtain
\begin{align*}
\Vert \  |f| \ \vert \mathcal{E}^{s-1}_{u,p,q}(\mathbb{R})  \Vert \leq C_{2} \Vert f \vert \mathcal{E}^{s-1}_{u,p,q}(\mathbb{R})  \Vert \leq C_{2} \Vert f \vert \mathcal{E}^{s}_{u,p,q}(\mathbb{R})  \Vert.
\end{align*}
Now we want to work with $  \Vert \partial^{1}  |f| \  \vert \mathcal{E}^{s-1}_{u,p,q}(\mathbb{R})  \Vert $. Because of $ p \geq 1  $ and $ q \geq 1 $ we can apply proposition \ref{pro_E_diff} with $ v = 1, a = \infty $ and $ N = 1  $. We get
\begin{align*}
&\Vert \partial^{1}  |f| \  \vert \mathcal{E}^{s-1}_{u,p,q}(\mathbb{R})  \Vert \\
& \qquad \leq C_{3}  \Vert \partial^{1}  |f| \  \vert \mathcal{M}^{u}_{p}(\mathbb{R})  \Vert  + C_{3}   \Big \Vert \Big (  \int_{0}^{\infty}  t^{-(s-1)q-q} \Big ( \int_{-t}^{t}\vert \Delta^{1}_{h} \partial^{1} |f|(x) \vert dh \Big )^{q}  \frac{dt}{t}  \Big )^{\frac{1}{q}} \Big \vert \mathcal{M}^{u}_{p}( \mathbb{R})  \Big \Vert. 
\end{align*}
At first we look at $ \Vert \partial^{1} |f| \  \vert \mathcal{M}^{u}_{p}(\mathbb{R})  \Vert  $. Because of $ f $ is real-valued we can define the sets $ \Omega_{f} = \{ x \in \mathbb{R} : f(x) \geq 0   \} $ and 
\begin{align*}
SC_{f} & =  \left\{ x \in \mathbb{R} : \exists \epsilon > 0 \ \mbox{with} \ f(y) < 0 \ \mbox{for} \ y \in (x - \epsilon , x) \ \mbox{and} \  f(y) > 0 \ \mbox{for} \ y \in (x , x + \epsilon)    \right\} \\
& \cup \; \left\{ x \in \mathbb{R} : \exists \epsilon > 0 \ \mbox{with} \ f(y) > 0 \ \mbox{for} \ y \in (x - \epsilon , x) \ \mbox{and} \  f(y) < 0 \ \mbox{for} \ y \in (x , x + \epsilon)    \right\}.   
\end{align*}
Now remember that we have $ f \in  C_{0}^{\infty}(\mathbb{R})  $. Hence the set $ SC_{f} $ has Lebesgue measure zero. Therefore if we work with the Morrey norm we can exclude the set $  SC_{f}  $. Because of this we find
\begin{align*}
 \Vert \partial^{1}|f| \ \vert \mathcal{M}^{u}_{p}(\mathbb{R})  \Vert 
&  \leq C_{4}  \sup_{\substack{a,b \in \mathbb{R} \\ a < b}} \vert a-b \vert^{\frac{1}{u}-\frac{1}{p}} \Big ( \int_{((a,b) \cap \Omega_{f}) \setminus  SC_{f} } \vert \partial^{1}f(x) \vert^{p} dx      \Big )^{\frac{1}{p}} \\
& \qquad \qquad + C_{4} \sup_{\substack{a,b \in \mathbb{R} \\ a < b}} \vert a-b \vert^{\frac{1}{u}-\frac{1}{p}} \Big ( \int_{((a,b) \cap \Omega_{f}^{c}) \setminus  SC_{f}} \vert - \partial^{1}f(x) \vert^{p} dx   \Big )^{\frac{1}{p}} \\
&  \leq C_{5} \Vert \partial^{1}f \vert \mathcal{M}^{u}_{p}(\mathbb{R})  \Vert.
\end{align*}
Now we can use proposition \ref{pro_E_diff} and lemma \ref{l_deriv}. Then we get
\begin{align*}
\Vert \partial^{1}f \vert \mathcal{M}^{u}_{p}(\mathbb{R})  \Vert \leq C_{6} \Vert \partial^{1}f \vert \mathcal{E}^{s-1}_{u,p,q}(\mathbb{R})  \Vert \leq C_{7} \Vert f \vert \mathcal{E}^{s}_{u,p,q}(\mathbb{R})  \Vert.
\end{align*}
So it remains to deal with 
\begin{align*}
\sup_{\substack{a,b \in \mathbb{R} \\ a < b}} \vert a-b \vert^{\frac{1}{u}-\frac{1}{p}} \Big ( \int_{(a,b)} \Big (  \int_{0}^{\infty}  t^{-(s-1)q-q} \Big ( \int_{-t}^{t}\vert \Delta^{1}_{h} \partial^{1}  |f| (x) \vert dh \Big )^{q}  \frac{dt}{t}  \Big )^{\frac{p}{q}} dx \Big )^{\frac{1}{p}}.
\end{align*}
Here we follow the ideas from the proof of the theorem in chapter 5.4.1. in \cite{RS}. At first for $ a < b  $ we write $ (a,b) = ( (a,b) \cap   \Omega_{f}  )  \cup ( (a,b) \cap \Omega_{f}^{c}   ) $
and for $ t \in \mathbb{R}   $
\begin{align*}
(-t,t) = ( (-t,t) \cap \left\{ h \in \mathbb{R} : x + h \in \Omega_{f}   \right\} )  \cup ( (-t,t) \cap  \left\{ h \in \mathbb{R} : x + h \not \in \Omega_{f}   \right\} ).
\end{align*}
Now we can use a version of the triangle inequality to split up the different cases. Then two different situations show up. On the one hand it is possible that we have $ x \in \Omega_{f}   $ and $ x + h \in  \Omega_{f}   $. So we find $ \vert \Delta^{1}_{h} \partial^{1} |f|(x) \vert = \vert  \partial^{1}f(x+h) - \partial^{1}f(x)   \vert = \vert \Delta^{1}_{h} \partial^{1}f(x)   \vert $. But then we have
\begin{align*}
& \Big \Vert \Big (  \int_{0}^{\infty}  t^{-(s-1)q-q} \Big ( \int_{-t}^{t}\vert \Delta^{1}_{h} \partial^{1}f(x) \vert dh \Big )^{q}  \frac{dt}{t}  \Big )^{\frac{1}{q}} \Big \vert \mathcal{M}^{u}_{p}( \mathbb{R})  \Big \Vert \\
& \qquad \qquad \qquad \qquad  \leq C_{8} \Vert \partial^{1}f \vert \mathcal{E}^{s-1}_{u,p,q}(\mathbb{R})  \Vert \leq C_{9} \Vert f \vert \mathcal{E}^{s}_{u,p,q}(\mathbb{R})  \Vert.
\end{align*}
The case $ x \in  \Omega_{f}^{c}  $ and $ x + h \not \in \Omega_{f}   $ leads to the same result. On the other hand we have the situation $ x \in \Omega_{f}  $ and $ x + h \not \in \Omega_{f}  $. Here we obtain
\begin{align*}
\vert \Delta^{1}_{h} \partial^{1} |f|(x) \vert = \vert  \partial^{1}f(x+h) + \partial^{1}f(x)   \vert \leq 2 \vert  \partial^{1}f(x) \vert  + \vert \Delta^{1}_{h} \partial^{1}f(x)   \vert.
\end{align*}
With $ \vert \Delta^{1}_{h} \partial^{1}f(x)   \vert  $ we can work like it is described before. So it remains to deal with
\begin{align*}
\sup_{\substack{a,b \in \mathbb{R} \\ a < b}} \vert a-b \vert^{\frac{1}{u}-\frac{1}{p}} \Big ( \int_{(a,b) \cap \Omega_{f}} \Big (  \int_{0}^{\infty}  t^{-(s-1)q-q} \Big ( \int_{\substack{(-t,t) \  \cap \\ \{ h \in \mathbb{R} \ : \ x + h  \not \in  \Omega_{f}  \} }}\vert \partial^{1}f(x) \vert dh \Big )^{q}  \frac{dt}{t}  \Big )^{\frac{p}{q}} dx \Big )^{\frac{1}{p}}.
\end{align*} 
Notice that the case $ x \in  \Omega_{f}^{c} $ and $ x + h  \in \Omega_{f}  $ leads to a similar situation. Next we observe 
\begin{align*}
& \sup_{\substack{a,b \in \mathbb{R} \\ a < b}} \vert a-b \vert^{\frac{1}{u}-\frac{1}{p}} \Big ( \int_{(a,b) \cap \Omega_{f}} \vert \partial^{1}f(x) \vert^{p} \Big (  \int_{0}^{\infty}  t^{-(s-1)q-q} \Big ( \int_{\substack{(-t,t) \ \cap \\ \{ h \in \mathbb{R} \ : \ x + h \not \in \Omega_{f}  \} }} 1 dh \Big )^{q}  \frac{dt}{t}  \Big )^{\frac{p}{q}} dx \Big )^{\frac{1}{p}} \\
& \qquad \leq \sup_{\substack{a,b \in \mathbb{R} \\ a < b}} \vert a-b \vert^{\frac{1}{u}-\frac{1}{p}} \Big ( \int_{(a,b) \cap \Omega_{f}} \vert \partial^{1}f(x) \vert^{p} \Big (  \int_{\dist(x,\Omega_{f}^{c})}^{\infty}  t^{-(s-1)q-q} \Big ( \int_{-t}^{t} 1 dh \Big )^{q}  \frac{dt}{t}  \Big )^{\frac{p}{q}} dx \Big )^{\frac{1}{p}} \\
& \qquad \leq C_{10} \sup_{\substack{a,b \in \mathbb{R} \\ a < b}} \vert a-b \vert^{\frac{1}{u}-\frac{1}{p}} \Big ( \int_{(a,b) \cap \Omega_{f}} \vert \partial^{1}f(x) \vert^{p} \Big (  \int_{\dist(x,\Omega_{f}^{c})}^{\infty}  t^{-(s-1)q}  \frac{dt}{t}  \Big )^{\frac{p}{q}} dx \Big )^{\frac{1}{p}} \\
& \qquad \leq C_{11} \sup_{\substack{a,b \in \mathbb{R} \\ a < b}} \vert a-b \vert^{\frac{1}{u}-\frac{1}{p}} \Big ( \int_{(a,b) \cap \Omega_{f}} \vert \partial^{1}f(x) \vert^{p} \dist(x,\Omega_{f}^{c})^{-(s-1)p}  dx \Big )^{\frac{1}{p}}.
\end{align*} 
Because of $  f \in C_{0}^{\infty}(\mathbb{R})   $ we can find disjoint open intervals $ I_{i}  $ such that $ \Omega_{f} = \bigcup_{i} \overline{I_{i}} $. For $ \vert  I_{i} \vert < \infty  $ we can write $ I_{i} = (c_{i},d_{i})   $ with $ c_{i} < d_{i} < c_{i+1} < d_{i+1}   $ and $ f(c_{i}) = f(d_{i}) = 0    $. So for $ \vert  I_{i} \vert < \infty  $ we observe $  \int_{I_{i}} \partial^{1}f(x) dx = f(d_{i}) - f(c_{i}) = 0 $. Now we get
\begin{align*}
& \sup_{\substack{a,b \in \mathbb{R} \\ a < b}} \vert a-b \vert^{\frac{1}{u}-\frac{1}{p}} \Big ( \int_{(a,b) \cap \Omega_{f}} \vert \partial^{1}f(x) \vert^{p} \dist(x,\Omega_{f}^{c})^{-(s-1)p}  dx \Big )^{\frac{1}{p}} \\
& \qquad = \sup_{\substack{a,b \in \mathbb{R} \\ a<b }} \vert a-b \vert^{\frac{1}{u}-\frac{1}{p}} \Big ( \sum_{i} \int_{(a,b) \cap \overline{I_{i}}} \vert \partial^{1}f(x) \vert^{p} \dist(x,I_{i}^{c})^{-(s-1)p}  dx \Big )^{\frac{1}{p}}.
\end{align*}
There are 3 different possible cases. First it is possible that the interval $ (a,b)  $ only intersects one single interval $ I_{1} $. Then because of $ s - 1 < 1/u  $ we can apply lemma \ref{Hardy_Mor} and obtain 
\begin{align*}
&  \sup_{\substack{a,b \in \mathbb{R} \\ a<b }} \vert a-b \vert^{\frac{1}{u}-\frac{1}{p}}  \Big ( \sum_{i} \int_{ (a,b) \cap \overline{I_{i}}} \vert \partial^{1}f(x) \vert^{p} \dist(x,I_{i}^{c})^{-(s-1)p}  dx \Big )^{\frac{1}{p}} \\
&  = \sup_{\substack{a,b \in \mathbb{R} \\ a<b }} \vert a-b \vert^{\frac{1}{u}-\frac{1}{p}}  \Big (  \int_{ (a,b) \cap \overline{I_{1}}} \vert \partial^{1}f(x) \vert^{p} \dist(x,I_{1}^{c})^{-(s-1)p}  dx \Big )^{\frac{1}{p}} \\
&  \leq C_{12} \sup_{\substack{a,b \in \mathbb{R} \\ a<b }} \vert a - b \vert^{\frac{1}{u}-\frac{1}{p}} \Big ( \int_{a}^{b} \Big ( \int_{0}^{\infty} r^{-(s-1)q} \Big ( \int_{-1 }^{1} \vert \partial^{1}f(x) - \partial^{1}f(x-rh)\vert dh \Big )^{q} \frac{dr}{r} \Big )^{\frac{p}{q}} dx \Big )^{\frac{1}{p}}.
\end{align*}
Second it is possible that the interval $ (a,b)  $ intersects two intervals $ I_{1}  $ and $ I_{2}  $. Then we get
\begin{align*}
&  \sup_{\substack{a,b \in \mathbb{R} \\ a<b }} \vert a-b \vert^{\frac{1}{u}-\frac{1}{p}}  \Big (  \sum_{i} \int_{ (a,b) \cap \overline{I_{i}}} \vert \partial^{1}f(x) \vert^{p} \dist(x,I_{i}^{c})^{-(s-1)p}  dx \Big )^{\frac{1}{p}} \\
& \qquad  \leq C_{13} \sum_{i \in \left\{ 1 , 2 \right\}} \sup_{\substack{a,b \in \mathbb{R} \\ a<b }} \vert a-b \vert^{\frac{1}{u}-\frac{1}{p}}  \Big (  \int_{ (a,b) \cap \overline{I_{i}}} \vert \partial^{1}f(x) \vert^{p} \dist(x,I_{i}^{c})^{-(s-1)p}  dx \Big )^{\frac{1}{p}} .
\end{align*}
Now again we can apply lemma \ref{Hardy_Mor} with the same result as before. The third case is that the interval $ (a,b)   $ intersects $  n \in \mathbb{N} $ intervals $  I_{i} $ with $ n \geq 3  $. But then we have the situation that $ (a,b) $ intersects the intervals $ I_{1}  $ and $ I_{n} $ and completely covers $ I_{2}, I_{3}, ... , I_{n-1}  $. So we find
\begin{align*}
&  \sup_{\substack{a,b \in \mathbb{R} \\ I_{2}, ... , I_{n-1} \subset (a,b) }} \vert a-b \vert^{\frac{1}{u}-\frac{1}{p}}  \Big ( \sum_{i} \int_{ (a,b) \cap \overline{I_{i}}} \vert \partial^{1}f(x) \vert^{p} \dist(x,I_{i}^{c})^{-(s-1)p}  dx \Big )^{\frac{1}{p}} \\
& \qquad  \leq C_{14} \sum_{i \in \left\{ 1 , n \right\}} \sup_{\substack{a,b \in \mathbb{R} \\ a<b }} \vert a-b \vert^{\frac{1}{u}-\frac{1}{p}}  \Big (  \int_{ (a,b) \cap \overline{I_{i}}} \vert \partial^{1}f(x) \vert^{p} \dist(x,I_{i}^{c})^{-(s-1)p}  dx \Big )^{\frac{1}{p}}  \\
& \qquad \qquad + C_{14}   \sup_{\substack{a,b \in \mathbb{R} \\ I_{2}, ... , I_{n-1} \subset (a,b) }} \vert a-b \vert^{\frac{1}{u}-\frac{1}{p}}   \Big ( \sum_{i  =  2}^{n-1}   \int_{  \overline{I_{i}}} \vert \partial^{1}f(x) \vert^{p} \dist(x,I_{i}^{c})^{-(s-1)p}  dx \Big )^{\frac{1}{p}}. 
\end{align*}
Now for the first part again we can use lemma \ref{Hardy_Mor}. For the second part because of $ s - 1 < 1/u \leq 1/p  $ we can apply lemma \ref{Hardy}. Using $ f(c_{2}) = \ldots = f(d_{n-1}) = 0   $ we obtain
\begin{align*}
& \sup_{\substack{a,b \in \mathbb{R} \\ I_{2}, ... , I_{n-1} \subset (a,b) }} \vert a-b \vert^{\frac{1}{u}-\frac{1}{p}}  \Big (  \sum_{i = 2}^{n - 1}   \int_{  \overline{I_{i}}} \vert \partial^{1}f(x) \vert^{p} \dist(x,I_{i}^{c})^{-(s-1)p}  dx \Big )^{\frac{1}{p}} \\
&  \leq C_{15} \sup_{\substack{a,b \in \mathbb{R} \\ a \leq c_{2} < d_{n-1} \leq b}} \vert a-b \vert^{\frac{1}{u}-\frac{1}{p}}  \Big (   \int_{c_{2}}^{d_{n-1}} \Big ( \int_{0}^{\infty} r^{-(s-1)q} \Big ( \int_{-1 }^{1} \vert  \Delta^{1}_{rh} \partial^{1}f(x-rh)\vert dh \Big )^{q} \frac{dr}{r} \Big )^{\frac{p}{q}} dx \Big )^{\frac{1}{p}} \\
&  \leq C_{15} \sup_{\substack{a,b \in \mathbb{R} \\ a < b}} \vert a-b \vert^{\frac{1}{u}-\frac{1}{p}}  \Big (   \int_{a}^{b} \Big ( \int_{0}^{\infty} r^{-(s-1)q} \Big ( \int_{-1 }^{1} \vert \partial^{1}f(x) - \partial^{1}f(x-rh)\vert dh \Big )^{q} \frac{dr}{r} \Big )^{\frac{p}{q}} dx \Big )^{\frac{1}{p}}.
\end{align*}
So if we use proposition \ref{pro_E_diff} and lemma \ref{l_deriv} we find
\begin{align*}
& \sup_{\substack{a,b \in \mathbb{R} \\ a < b}} \vert a-b \vert^{\frac{1}{u}-\frac{1}{p}} \Big ( \int_{(a,b) \cap \Omega_{f}} \vert \partial^{1}f(x) \vert^{p} \dist(x,\Omega_{f}^{c})^{-(s-1)p}  dx \Big )^{\frac{1}{p}} \\
&   \leq C_{16} \sup_{\substack{a,b \in \mathbb{R} \\ a < b}} \vert a-b \vert^{\frac{1}{u}-\frac{1}{p}}  \Big (   \int_{a}^{b} \Big ( \int_{0}^{\infty} r^{-(s-1)q} \Big ( \int_{-1 }^{1} \vert \partial^{1}f(x) - \partial^{1}f(x-rh)\vert dh \Big )^{q} \frac{dr}{r} \Big )^{\frac{p}{q}} dx \Big )^{\frac{1}{p}} \\
&  \leq C_{17} \Vert \partial^{1}f \vert \mathcal{E}^{s-1}_{u,p,q}(\mathbb{R})  \Vert \leq C_{18} \Vert f \vert \mathcal{E}^{s}_{u,p,q}(\mathbb{R})  \Vert.
\end{align*}
Step 1 of the proof is complete. 

\textit{Step 2.} Now we prove \eqref{d1s>1_eq1} for $ f \in  \mathbb{E}^{s}_{u,p,q}(\mathbb{R})  $. Let $ \varrho \in \mathbb{S}(\mathbb{R})  $ be a real even function with $ \varrho (x) = 1 $ if $ \vert x \vert \leq 1  $ and $ \varrho (x) = 0 $ if $ \vert x \vert \geq 3/2  $. For $ j \in \mathbb{N}_{0} $ and $ x \in \mathbb{R}  $ we put $ \varrho^{j}(x) = \varrho(2^{-j}x)  $. Moreover we define $ f_{j}(x) = \varrho^{j}(x) \mathcal{F}^{-1}[\varrho^{j} \mathcal{F}f](x) $, see step 1 of the proof of theorem 25.8 in \cite{Tr01}. Then $  f_{j} $ is real and because of the Paley-Wiener-Schwarz theorem we find $ f_{j} \in C_{0}^{\infty}(\mathbb{R})  $. So we also have $ f_{j} \in \mathbb{E}^{s}_{u,p,q}(\mathbb{R}) $. We observe 
\begin{align*}
\lim_{j \rightarrow \infty} f_{j} = f \qquad \mbox{and} \qquad \lim_{j \rightarrow \infty} \vert f_{j} \vert = \vert f \vert
\end{align*}
with convergence in $  \mathcal{S}'(\mathbb{R})  $. So we can apply the Fatou property, see lemma \ref{l_fatou}, and step 1 of this proof. Then we get
\begin{align*}
\Vert \  |f| \  \vert \mathcal{E}^{s}_{u,p,q}(\mathbb{R})  \Vert & \leq C_{1} \sup_{j \in \mathbb{N}_{0}} \Vert \  |f_{j}| \  \vert \mathcal{E}^{s}_{u,p,q}(\mathbb{R})  \Vert  \leq C_{2} \sup_{j \in \mathbb{N}_{0}} \Vert f_{j}  \vert \mathcal{E}^{s}_{u,p,q}(\mathbb{R})  \Vert.
\end{align*}
Next we use lemma \ref{l_bmD2}. Let $ m \in \mathbb{N}  $ be large enough. We find
\begin{align*}
\sup_{j \in \mathbb{N}_{0}} \Vert f_{j}  \vert \mathcal{E}^{s}_{u,p,q}(\mathbb{R})  \Vert  & \leq C_{3} \sup_{j \in \mathbb{N}_{0}}  \Big ( \sum_{\vert \alpha \vert \leq m}  \Vert D^{\alpha} \varrho^{j} \vert L_{\infty}(\mathbb{R}) \Vert  \Big ) \; \Vert \mathcal{F}^{-1}[\varrho^{j} \mathcal{F}f]  \vert  \mathcal{E}^{s}_{u,p,q}(\mathbb{R})   \Vert \\
& \leq C_{4} \sup_{j \in \mathbb{N}_{0}}   \Vert \mathcal{F}^{-1}[\varrho^{j} \mathcal{F}f]  \vert  \mathcal{E}^{s}_{u,p,q}(\mathbb{R})   \Vert.
\end{align*}
Now we apply lemma \ref{Four_Mult}. Let $ N \in \mathbb{N}  $ be large enough. Then we obtain
\begin{align*}
\sup_{j \in \mathbb{N}_{0}}   \Vert \mathcal{F}^{-1}[\varrho^{j} \mathcal{F}f]  \vert  \mathcal{E}^{s}_{u,p,q}(\mathbb{R})   \Vert & \leq C_{5} \sup_{j \in \mathbb{N}_{0}} \sup_{\vert \gamma  \vert \leq N } \sup_{x \in \mathbb{R}}  ( 1 + \vert x \vert^{2})^{\frac{ \vert \gamma \vert}{2}}  \vert D^{\gamma} \varrho^{j}(x)   \vert   \   \    \Vert  f  \vert  \mathcal{E}^{s}_{u,p,q}(\mathbb{R})   \Vert \\  
& \leq C_{6}    \Vert  f  \vert  \mathcal{E}^{s}_{u,p,q}(\mathbb{R})   \Vert.
\end{align*}
So the whole proof is complete.
\end{proof}

\subsection{The boundedness of the operator $T$ for $ s \geq 1 $ and $ d \in \mathbb{N}  $ }\label{sec_3.3}

In the case of the original Triebel-Lizorkin spaces $ \mathbb{F}^{s}_{p,q}(\R)   $ the step from $ d = 1  $ to $ d > 1  $ is very easy, see chapter 5.4.1. in \cite{RS}. The reason for this is that the spaces $ \mathbb{F}^{s}_{p,q}(\R)   $ have the so-called Fubini property. At the first moment one maybe could hope that also the Triebel-Lizorkin-Morrey spaces have the Fubini property. We use the following definition, see also chapter 2.5.13 in \cite{Tr83}.

\begin{defi}\label{def_fub_p}
Let $ s > 0  $, $ 1 \leq p \leq u < \infty  $ and $ 1 \leq q \leq \infty   $. Let $ d \geq 2  $. Then the space $ \mathbb{E}^{s}_{u,p,q}(\R)   $ has the Fubini property if the norm 
\begin{equation}\label{fub_m_eq}
\Vert f \vert \mathcal{E}^{s}_{u,p,q}(\R) \Vert^{F} = \sum_{j = 1}^{d} \Big \Vert \Vert f(x_{1}, \ldots , x_{j-1}, \cdot , x_{j+1}, \ldots , x_{d}) \vert  \mathcal{E}^{s}_{u,p,q}(\mathbb{R}) \Vert  \Big \vert \mathcal{M}^{u}_{p}(\mathbb{R}^{d-1}) \Big \Vert
\end{equation}
is equivalent to $ \Vert f \vert \mathcal{E}^{s}_{u,p,q}(\R) \Vert  $.
\end{defi} 

Unfortunately for $ p \not = u   $ there is the following result.

\begin{lem}\label{res_no_Fub}
Let $ s > 0  $, $ 1 \leq p < u < \infty  $, $ 1 \leq q \leq \infty   $ and $ d \geq 2  $. Let in addition 
\begin{align*}
p \leq \frac{d - 1}{d} u.
\end{align*}
Then the spaces $ \mathbb{E}^{s}_{u,p,q}(\R)  $ do not have the Fubini property. 
\end{lem}

\begin{proof}
To prove this result we investigate the properties of a special test function. Let $ f \in C_{0}^{\infty}(\mathbb{R}^{d-1}) $ be a function $ f : \mathbb{R}^{d-1} \rightarrow \mathbb{R}  $ that depends on $ x' = (x_{2}, x_{3}, \ldots , x_{d} )  $ and has a support in $ [0,1]^{d-1}  $. We assume $ \int_{[0,1]^{d-1}} \vert f(x') \vert^{p} dx' = 1   $. Now we define the function $ g : \mathbb{R}^{d} \rightarrow \mathbb{R}   $ by $ g(x_{1} , x') = f(x')   $. At first we prove $ \Vert g \vert \mathcal{E}^{s}_{u,p,q}(\R) \Vert^{F} = \infty   $. From lemma \ref{E_Mor_em} we learn $ \mathbb{E}^{s}_{u,p,q}(\mathbb{R}) \hookrightarrow \mathbb{M}^{u}_{p}(\mathbb{R})   $. Let $ t > 0  $. Then we find 
\begin{align*}
\Vert g \vert \mathcal{E}^{s}_{u,p,q}(\R) \Vert^{F} & \geq C_{1}  \Big \Vert \Vert g( \cdot , x') \vert  \mathcal{M}^{u}_{p}(\mathbb{R}) \Vert  \Big \vert \mathcal{M}^{u}_{p}(\mathbb{R}^{d-1}) \Big \Vert \\
& \geq C_{2}  t^{\frac{1}{u}-\frac{1}{p}}  \Big ( \int_{[0,1]^{d-1}}    \int_{0}^{t} \vert g(x_{1},x') \vert^{p} dx_{1}     dx'  \Big )^{\frac{1}{p}} \\
& = C_{2}  t^{\frac{1}{u}-\frac{1}{p}}  \Big ( \int_{0}^{t} \int_{[0,1]^{d-1}}     \vert  f(x')  \vert^{p} dx' dx_{1}       \Big )^{\frac{1}{p}} = C_{2}  t^{\frac{1}{u}}.
\end{align*}
If $ t $ tends to infinity we obtain $ \Vert g \vert \mathcal{E}^{s}_{u,p,q}(\R) \Vert^{F} = \infty   $. Now we prove that under the given conditions we have $ g \in \mathbb{E}^{s}_{u,p,q}(\R)  $. Because of proposition \ref{pro_E_diff} it is enough to show $ \Vert g \vert \mathcal{E}^{s}_{u,p,q}(\R) \Vert^{(1,1)} < \infty  $. At first we look at the Morrey norm. Since $ g $ is bounded and smooth we only have to deal with large cubes $ [0,t]^{d}   $ with $ t \geq 1 $. We observe
\begin{align*}
\sup_{t \geq 1}  t^{\frac{d}{u}-\frac{d}{p}} \Big ( \int_{0}^{t} \int_{[0,1]^{d-1}} \vert f(x')  \vert^{p} dx' dx_{1}      \Big )^{\frac{1}{p}}  =  \sup_{t \geq 1}  t^{\frac{d}{u}-\frac{d - 1}{p} } < \infty.
\end{align*} 
In the last step we used $  p \leq \frac{d - 1}{d} u $. Now we have to deal with the second part of the norm $ \Vert g \vert \mathcal{E}^{s}_{u,p,q}(\R) \Vert^{(1,1)}   $. To do so we apply the well-known formula 
\begin{equation}\label{diff_abl_c}
\vert \Delta^{N}_{h} g(x)   \vert \leq C \vert h \vert^{N} \max_{\vert \gamma \vert = N} \sup_{\vert x - y \vert \leq N \vert h \vert } \vert D^{\gamma} g(y)   \vert  .
\end{equation} 
Because of $ f \in C_{0}^{\infty}(\mathbb{R}^{d-1})  $ and $ \supp f \subset [0,1]^{d-1} $ we find
\begin{align*}
& \Big \Vert \Big (  \int_{0}^{1}  t^{-sq} \Big ( t^{-d} \int_{B(0,t)}\vert \Delta^{N}_{h}g(x) \vert dh \Big )^{q}  \frac{dt}{t}  \Big )^{\frac{1}{q}} \Big \vert \mathcal{M}^{u}_{p}( \mathbb{R}^d)  \Big \Vert \\ 
& \qquad \qquad \leq C_{3} \Big \Vert \Big (  \int_{0}^{1}  t^{-sq + Nq}      \chi_{ \{ \vert x' \vert \leq d + Nd \}}(x)  \frac{dt}{t}  \Big )^{\frac{1}{q}} \Big \vert \mathcal{M}^{u}_{p}( \mathbb{R}^d)  \Big \Vert \\ 
& \qquad \qquad \leq C_{4} \Big \Vert        \chi_{ \{ \vert x' \vert \leq d + Nd \}}(x)  \Big \vert \mathcal{M}^{u}_{p}( \mathbb{R}^d)  \Big \Vert. 
\end{align*} 
In the last step we used $ N > s  $. $  \chi_{ \{ \vert x' \vert \leq d + Nd \}}(x)   $ is a characteristic function that is zero if $ \vert x' \vert   $  is large. Since $ | \chi_{ \{ \vert x' \vert \leq d + Nd \}}(x) | \leq 1   $ for all $ x \in \R $ the supremum of the Morrey norm is realized by big cubes $ [-t,t]^{d}   $ with $ t > d + Nd   $. Because of this we obtain
\begin{align*}
\Big \Vert  \chi_{ \{ \vert x' \vert \leq d + Nd \}}(x)  \Big \vert \mathcal{M}^{u}_{p}( \mathbb{R}^d)  \Big \Vert & \leq C_{5} \sup_{t > d + Nd} t^{\frac{d}{u}-\frac{d}{p}} \Big ( \int_{-t}^{t} \int_{[-d-Nd,d+Nd]^{d-1}}  \chi_{ \{ \vert x' \vert \leq d + Nd \}}(x)  dx' dx_{1}      \Big )^{\frac{1}{p}} \\
& \leq C_{6} \sup_{t > d + Nd} t^{\frac{d}{u}-\frac{d}{p}+\frac{1}{p}} < \infty.
\end{align*} 
In the last step again we used $ p \leq \frac{d - 1}{d} u  $. The proof is complete.
\end{proof}

So it turns out that in most of the cases the Triebel-Lizorkin-Morrey spaces do not have the Fubini property. But fortunately in the case of small $ s $ for the Triebel-Lizorkin-Morrey spaces there exist so-called Morrey characterizations, see chapter 3.6.3 in \cite{Tr14}. To explain what this is we have to introduce some additional notation concerning Triebel-Lizorkin spaces on domains. Let $ \Omega \subset \mathbb{R}^{d}   $ be a bounded domain. As usual $ D(\Omega) = C_{0}^{\infty}(\Omega)  $ and $ D'(\Omega)    $ is the dual space of distributions on $ \Omega  $. Then for $  s > 0  $, $  1 \leq p < \infty  $ and $ 1 \leq q \leq \infty   $ we define for all $ f \in D'(\Omega)  $ the norm $ \Vert f \vert F^{s}_{p,q}(\Omega)  \Vert = \inf \Vert g \vert F^{s}_{p,q}( \mathbb{R}^{d} )  \Vert $, where the infimum is taken over all $ g \in \mathbb{F}^{s}_{p,q}(\mathbb{R}^{d})   $ such that we have $ f = g   $ on $  \Omega  $ in $  D'(\Omega)   $. More explanations concerning Triebel-Lizorkin spaces on domains can be found in \cite{Tr06}, see chapter 1.11. Now for $ j \in \mathbb{Z}  $ and $ m \in \mathbb{Z}^{d}   $ we put $ Q_{j,m} = 2^{-j}m + 2^{-j}(0,1)^d $. For $ c > 1 $ by $ c Q_{j,m}   $ we denote a cube concentric with $  Q_{j,m} $ that has side-length $ c 2^{-j}  $. Using this notation we can formulate the following result.

\begin{prop}\label{E_MorChar}
Let $ 1 \leq p \leq u < \infty  $ and $ 1 \leq q \leq \infty  $. Let $ 0 < s < \min  ( 1/p , d/u )$. Let $ f \in \mathcal{S}'(\mathbb{R}^{d})  $. Then $ f \in \mathbb{E}^{s}_{u,p,q}(\R)    $ if and only if 
\begin{equation}\label{MorChar_eq1}
\Vert f \vert \mathcal{E}^{s}_{u,p,q}(\R) \Vert^{MC} =  \sup_{j \in \mathbb{Z}, m \in \mathbb{Z}^{d}} 2^{j(\frac{d}{p} - \frac{d}{u})} \Vert f \vert F^{s}_{p,q}(2 Q_{j,m}) \Vert
\end{equation}
is finite. The norms $  \Vert f \vert \mathcal{E}^{s}_{u,p,q}(\R) \Vert   $ and $ \Vert f \vert \mathcal{E}^{s}_{u,p,q}(\R) \Vert^{MC}   $ are equivalent.
\end{prop}

\begin{proof}
This result is a combination of theorem 3.64. from \cite{Tr14} with theorem 3.38. from \cite{Tr14} and corollary 3.3. from \cite{ysy}. One may also consult theorem 2.29. in \cite{Tr13}.
\end{proof} 

Now we are well-prepared to prove the following result concerning the boundedness of the operator $ T $.  

\begin{prop}\label{d>1_E_MR}
Let $ 1 \leq p < u < \infty  $, $ 1 \leq q \leq \infty   $ and $ d \in \mathbb{N}  $. Let $ 1/p - 1/u > 1 - 1/d   $ and 
\begin{align*}
1 < s < \min \Big ( 1 + \frac{1}{p} ,  1 + \frac{d}{u} \Big ). 
\end{align*}
Then there is a constant $ C > 0 $ independent of $ f \in \mathbb{E}^{s}_{u,p,q}(\mathbb{R}^{d})  $ such that we have 
\begin{equation}\label{d>1_E_eq1}
\Vert Tf  \vert \mathcal{E}^{s}_{u,p,q}(\mathbb{R}^{d})  \Vert \leq C \Vert f \vert \mathcal{E}^{s}_{u,p,q}(\mathbb{R}^{d})  \Vert
\end{equation}
for all $ f \in \mathbb{E}^{s}_{u,p,q}(\mathbb{R}^{d})   $.
\end{prop}

\begin{proof}
\textit{Step 1. Preparations and Fatou property.}

This step is similar to step 2 of the proof from proposition \ref{pro_d1s>1}. Let $   f \in \mathbb{E}^{s}_{u,p,q}(\mathbb{R}^{d})  $ and $ \psi \in C^{\infty}_{0}(\R)   $ be a real function with $ \psi(x) = 1    $ for $ |x| \leq 1  $ and $ \psi(x) = 0   $ for $ |x| > 2  $. For $ j \in \mathbb{N}  $ we define $ f_{j} =  \mathcal{F}^{-1}[ \psi(2^{-j} \cdot ) (\mathcal{F}  f) ]$. Then we have $ f_{j} \in  \mathbb{E}^{s}_{u,p,q}(\mathbb{R}^{d})  $ and $ \supp \mathcal{F}  f_{j} \subset B(0,2^{j+1})  $. So because of the Paley-Wiener-Schwarz theorem $ f_{j}  $ is a $ C^{\infty} -  $ function. Moreover it is not difficult to see that for all $ \alpha \in \mathbb{N}_{0}^{d}   $ we have $ D^{\alpha} f_{j} \in L_{\infty}(\R)   $ and $  D^{\alpha} f_{j} \in \mathbb{E}^{s}_{u,p,q}(\mathbb{R}^{d})  $.  We have 
\begin{align*}
\lim_{j \rightarrow \infty} f_{j} = f \qquad \mbox{and} \qquad \lim_{j \rightarrow \infty} |f_{j}| = |f| 
\end{align*} 
with convergence in $ \mathcal{S}'(\mathbb{R}^{d})   $. Let us assume that we know \eqref{d>1_E_eq1} for all functions $ f_{j}  $. Then we can apply the Fatou property, see lemma \ref{l_fatou}, and find
\begin{align*}
\Vert \  |f| \  \vert \mathcal{E}^{s}_{u,p,q}(\mathbb{R}^d)  \Vert & \leq C_{1} \sup_{j \in \mathbb{N}} \Vert \  |f_{j}| \  \vert \mathcal{E}^{s}_{u,p,q}(\mathbb{R}^d)  \Vert \leq C_{2} \sup_{j \in \mathbb{N}} \Vert f_{j}  \vert \mathcal{E}^{s}_{u,p,q}(\mathbb{R}^d)  \Vert \leq C_{3} \Vert f  \vert \mathcal{E}^{s}_{u,p,q}(\mathbb{R}^d)  \Vert .
\end{align*}
In the last step we used lemma \ref{Four_Mult}. Therefore it is enough to prove \eqref{d>1_E_eq1} for analytic functions $ f \in \mathbb{E}^{s}_{u,p,q}(\mathbb{R}^{d})   $ that fulfill $ \supp \mathcal{F}  f \subset B(0,2^{R})   $ for some $ R \in \mathbb{N}   $. The same trick is used in step 1 of the proof of theorem 25.8 in \cite{Tr01}.

\textit{Step 2. Pick out cubes without zeros.}

In what follows $ f \in \mathbb{E}^{s}_{u,p,q}(\mathbb{R}^{d})   $ is an analytic function with $ \supp \mathcal{F}  f \subset B(0,2^{R})   $ for some $ R \in \mathbb{N}   $. At first we use lemma \ref{l_deriv}. Then we find
\begin{align*}
\Vert \  |f| \ \vert \mathcal{E}^{s}_{u,p,q}(\mathbb{R}^{d})  \Vert \leq C_{1} \Vert \  |f| \ \vert \mathcal{E}^{s-1}_{u,p,q}(\mathbb{R}^{d})  \Vert + C_{1} \sum_{i = 1}^{d} \Vert \partial^{1}_{i}  |f| \  \vert \mathcal{E}^{s-1}_{u,p,q}(\mathbb{R}^{d})  \Vert.
\end{align*}
For the first term because of $ 0 < s - 1 < 1   $ we can apply proposition \ref{pro_0<s<1}. When we use lemma \ref{l_deriv} again we get 
\begin{equation}\label{s=1_change}
\Vert \  |f| \ \vert \mathcal{E}^{s-1}_{u,p,q}(\mathbb{R}^{d})  \Vert \leq C_{2} \Vert f \vert \mathcal{E}^{s-1}_{u,p,q}(\mathbb{R}^{d})  \Vert \leq C_{2} \Vert f \vert \mathcal{E}^{s}_{u,p,q}(\mathbb{R}^{d})  \Vert.
\end{equation}
So in what follows we have to deal with $  \Vert \partial^{1}_{i} |f| \   \vert \mathcal{E}^{s-1}_{u,p,q}(\mathbb{R}^{d})  \Vert $ with $ i \in \{ 1, 2, \ldots , d  \}  $. Because of $ 0 < s - 1 < \min ( 1/p , d/u )   $ we can apply proposition \ref{E_MorChar}. Then we obtain
\begin{align*}
\Vert \partial^{1}_{i} |f| \   \vert \mathcal{E}^{s-1}_{u,p,q}(\mathbb{R}^{d})  \Vert \leq C_{3}  \sup_{j \in \mathbb{Z}, m \in \mathbb{Z}^{d}} 2^{j(\frac{d}{p} - \frac{d}{u})} \Vert \partial^{1}_{i} |f| \   \vert F^{s-1}_{p,q}(2 Q_{j,m}) \Vert.
\end{align*}
Now by $ \mathcal{Q}_{SC}(f)   $ we denote the set of all cubes of the form $ 2 Q_{j,m}  $ that have the following property. There exist $ y_{1}, y_{2} \in  2 Q_{j,m}  $ such that we have $ f(y_{1}) < 0 < f(y_{2})   $. That means $ f $ has a sign change in $ 2 Q_{j,m}  $. Using this notation we can write
\begin{align*}
\Vert \partial^{1}_{i} |f| \   \vert \mathcal{E}^{s-1}_{u,p,q}(\mathbb{R}^{d})  \Vert & \leq C_{3}  \sup_{\substack{j \in \mathbb{Z}, m \in \mathbb{Z}^{d} \\ 2 Q_{j,m} \: \in \:  \mathcal{Q}_{SC}(f)  }} 2^{j(\frac{d}{p} - \frac{d}{u})} \Vert \partial^{1}_{i} |f| \   \vert F^{s-1}_{p,q}(2 Q_{j,m}) \Vert \\
& \qquad \qquad \qquad + C_{3} \sup_{\substack{j \in \mathbb{Z}, m \in \mathbb{Z}^{d} \\ 2 Q_{j,m} \: \not \in \:  \mathcal{Q}_{SC}(f)  }} 2^{j(\frac{d}{p} - \frac{d}{u})} \Vert \partial^{1}_{i} |f| \   \vert F^{s-1}_{p,q}(2 Q_{j,m}) \Vert.
\end{align*}
In each cube that fulfills $ 2 Q_{j,m} \: \not \in \:  \mathcal{Q}_{SC}(f)  $ the function $ f $ is either positive everywhere or negative everywhere. Therefore for $ 2 Q_{j,m} \: \not \in \:  \mathcal{Q}_{SC}(f)  $ it does not matter whether we write $ \Vert \partial^{1}_{i} |f| \   \vert F^{s-1}_{p,q}(2 Q_{j,m}) \Vert  $ or $  \Vert \partial^{1}_{i} f   \vert F^{s-1}_{p,q}(2 Q_{j,m}) \Vert  $. Hence a combination of proposition \ref{E_MorChar} and lemma \ref{l_deriv} yields
\begin{align*}
\sup_{\substack{j \in \mathbb{Z}, m \in \mathbb{Z}^{d} \\ 2 Q_{j,m} \: \not \in \:  \mathcal{Q}_{SC}(f)  }} 2^{j(\frac{d}{p} - \frac{d}{u})} \Vert \partial^{1}_{i} |f| \   \vert F^{s-1}_{p,q}(2 Q_{j,m}) \Vert & \leq C_{4} \sup_{j \in \mathbb{Z}, m \in \mathbb{Z}^{d} } 2^{j(\frac{d}{p} - \frac{d}{u})} \Vert \partial^{1}_{i} f    \vert F^{s-1}_{p,q}(2 Q_{j,m}) \Vert \\
& \leq C_{5} \Vert \partial^{1}_{i} f   \vert \mathcal{E}^{s-1}_{u,p,q}(\mathbb{R}^{d})  \Vert  \\
& \leq C_{6} \Vert  f   \vert \mathcal{E}^{s}_{u,p,q}(\mathbb{R}^{d})  \Vert .
\end{align*}
So hereinafter we only have to deal with cubes that fulfill $  2 Q_{j,m} \: \in \:  \mathcal{Q}_{SC}(f)  $. That means we have to investigate the term
\begin{equation}\label{proof_d>1_SC}
\sup_{\substack{j \in \mathbb{Z}, m \in \mathbb{Z}^{d} \\ 2 Q_{j,m} \: \in \:  \mathcal{Q}_{SC}(f)  }} 2^{j(\frac{d}{p} - \frac{d}{u})} \Vert \partial^{1}_{i} |f| \   \vert F^{s-1}_{p,q}(2 Q_{j,m}) \Vert.
\end{equation}

\textit{Step 3. Pick out cubes of middle size.}

To continue we split up the term \eqref{proof_d>1_SC} in the following way.
\begin{align*}
& \sup_{\substack{j \in \mathbb{Z}, m \in \mathbb{Z}^{d} \\ 2 Q_{j,m} \: \in \:  \mathcal{Q}_{SC}(f)  }} 2^{j(\frac{d}{p} - \frac{d}{u})} \Vert \partial^{1}_{i} |f| \   \vert F^{s-1}_{p,q}(2 Q_{j,m}) \Vert \\
& \qquad \qquad \leq \sup_{\substack{j \geq R , m \in \mathbb{Z}^{d} \\ 2 Q_{j,m} \: \in \:  \mathcal{Q}_{SC}(f)  }} \ldots  + \sup_{\substack{ 0 <  j \leq R   , m \in \mathbb{Z}^{d} \\ 2 Q_{j,m} \: \in \:  \mathcal{Q}_{SC}(f)  }} \ldots  + \sup_{\substack{ j \in \mathbb{Z} \setminus \mathbb{N} , m \in \mathbb{Z}^{d} \\ 2 Q_{j,m} \: \in \:  \mathcal{Q}_{SC}(f)  }} \ldots \  .
\end{align*}
Whereas the first and the last term are like we want we should have a closer look at the second one. Therefore let $ k \in \mathbb{N} $ be a natural number with $  0 <  k \leq R  $ and let $  m \in \mathbb{Z}^{d} $. We investigate 
\begin{align*}
2^{k(\frac{d}{p} - \frac{d}{u})} \Vert \partial^{1}_{i} |f| \   \vert F^{s-1}_{p,q}(2 Q_{k,m}) \Vert.
\end{align*}
Therefore we define the set $ Z(f) = \{ x \in \R  : f(x) = 0  \} $. In addition for each $ l \in \mathbb{N}  $ we define numbers $ R_{l} $ by $ R_{l} = R + l   $. Then of course for all $ l \in \mathbb{N}  $ we have $ R_{l} > R   $ and we observe $ \lim_{l \rightarrow \infty}  R_{l} = \infty    $. Using this numbers for all $ l \in \mathbb{N}  $ we define sets
\begin{equation}\label{d>1_FAfZ}
ZF_{l}(f) = \{ x \in \R  : \dist(x,Z(f))  \leq \frac{1}{100   d} 2^{-R_{l}}   \}.
\end{equation}
Then of course we have $ 2 Q_{k,m} = [ 2 Q_{k,m} \cap ZF_{l}(f) ] \cup  [ 2 Q_{k,m} \cap ZF_{l}(f)^{c} ]      $.
Moreover since $ f $ is real valued we can define the sets
\begin{equation}\label{d>1_f+}
F^{+}(f) = \{ x \in \R : f(x) > 0   \} \qquad \mbox{and} \qquad F^{-}(f) = \{ x \in \R : f(x) < 0   \}.
\end{equation}
Now for all $ x \in ZF_{l}(f)^{c}   $ we can observe $ f(x) \not = 0   $. So it is possible to write $ ZF_{l}(f)^{c} =   [ ZF_{l}(f)^{c} \cap F^{+}(f) ] \cup  [ ZF_{l}(f)^{c} \cap F^{-}(f) ]      $. Consequently for all $ l \in \mathbb{N}  $ we obtain the disjoint decomposition
\begin{align*}
2 Q_{k,m} & = [ 2 Q_{k,m} \cap ZF_{l}(f) ] \cup  [ 2 Q_{k,m} \cap [ ZF_{l}(f)^{c} \cap F^{+}(f) ] ] \cup [ 2 Q_{k,m} \cap [ ZF_{l}(f)^{c} \cap F^{-}(f) ] ] \\
& = A_{1}^{l} \cup A_{2}^{l} \cup A_{3}^{l}.
\end{align*} 
(Notice that it is not really necessary to work with a disjoint decomposition here. Therefore it is also possible to deal with open sets that are overlapping a bit.) Now let $ l^{*} \in \mathbb{N}  $ be a fixed large natural number that will be specified later. Then we find
\begin{align*}
\Vert \partial^{1}_{i} |f| \   \vert F^{s-1}_{p,q}(2 Q_{k,m}) \Vert \leq \Vert \partial^{1}_{i} |f| \   \vert F^{s-1}_{p,q}(A_{1}^{l^{*}}) \Vert + \Vert \partial^{1}_{i} |f| \   \vert F^{s-1}_{p,q}(A_{2}^{l^{*}}) \Vert  + \Vert \partial^{1}_{i} |f| \   \vert F^{s-1}_{p,q}(A_{3}^{l^{*}}) \Vert .
\end{align*} 
Notice that for all $ x \in A_{2}^{l^{*}}  $ we have $ f(x) > 0   $. Therefore since $ A_{2}^{l^{*}} \subset 2 Q_{k,m}  $ we can write
\begin{align*}
\Vert \partial^{1}_{i} |f| \   \vert F^{s-1}_{p,q}(A_{2}^{l^{*}}) \Vert = \Vert \partial^{1}_{i} f   \vert F^{s-1}_{p,q}(A_{2}^{l^{*}}) \Vert \leq \Vert \partial^{1}_{i} f   \vert F^{s-1}_{p,q}(2 Q_{k,m}) \Vert.
\end{align*}
On the other hand for $ x \in A_{3}^{l^{*}}   $ we observe $ f(x) < 0   $. So because of $ A_{3}^{l^{*}} \subset 2 Q_{k,m}  $ we find
\begin{align*}
\Vert \partial^{1}_{i} |f| \   \vert F^{s-1}_{p,q}(A_{3}^{l^{*}}) \Vert = \Vert \partial^{1}_{i} f   \vert F^{s-1}_{p,q}(A_{3}^{l^{*}}) \Vert \leq \Vert \partial^{1}_{i} f   \vert F^{s-1}_{p,q}(2 Q_{k,m}) \Vert.
\end{align*}
Next because of $ s - 1 < \min ( 1/p , d/u )   $ we can apply proposition \ref{E_MorChar}. Consequently when we use lemma \ref{l_deriv} for all $ 0 < k \leq R   $ we get
\begin{align*}
& 2^{k(\frac{d}{p} - \frac{d}{u})} \Vert \partial^{1}_{i} |f| \   \vert F^{s-1}_{p,q}(2 Q_{k,m}) \Vert \\
& \qquad \leq 2^{k(\frac{d}{p} - \frac{d}{u})} \Vert \partial^{1}_{i} |f| \   \vert F^{s-1}_{p,q}(A_{1}^{l^{*}}) \Vert + C_{1} 2^{k(\frac{d}{p} - \frac{d}{u})} \Vert \partial^{1}_{i} f   \vert F^{s-1}_{p,q}(2 Q_{k,m}) \Vert \\
& \qquad \leq 2^{k(\frac{d}{p} - \frac{d}{u})} \Vert \partial^{1}_{i} |f| \   \vert F^{s-1}_{p,q}(A_{1}^{l^{*}}) \Vert + C_{2}  \Vert \partial^{1}_{i} f   \vert \mathcal{E}^{s-1}_{u,p,q}(\R) \Vert \\
& \qquad \leq 2^{k(\frac{d}{p} - \frac{d}{u})} \Vert \partial^{1}_{i} |f| \   \vert F^{s-1}_{p,q}(A_{1}^{l^{*}}) \Vert + C_{2}  \Vert f   \vert \mathcal{E}^{s}_{u,p,q}(\R) \Vert .
\end{align*}
Next we observe
\begin{align*}
\lim_{l \rightarrow \infty} 2 Q_{k,m} \cap ZF_{l}(f) =   2 Q_{k,m} \cap Z(f)
\end{align*} 
as sets. We need some knowledge concerning the zero set $ Z(f)  $ of real analytic functions $ f : \mathbb{R}^{d} \rightarrow \mathbb{R}   $ with $  f \not = 0  $. For this reason we collected everything we need in the appendix at the end of this paper, see lemma \ref{lem_ZOAF}. Let $ \sigma = \sigma(p,u,d) > 0  $ be a small number that will be specified later. Then from (iv) and (v) in lemma \ref{lem_ZOAF} in the appendix we learn the following. For each real analytic $ f $ with $ f \not = 0  $ and each cube $  2 Q_{k,m}  $ we find a maybe very large $ l^{*} = l^{*}(f,R,p,u,d) \in \mathbb{N}   $ such that $ 2^{R_{l^{*}}-k}   $ is much larger than $  R_{l^{*}}  $ and  $ R_{l^{*}} \  2^{ \sigma (k- R_{l^{*}} )} \leq 1  $ and such that the set $ 2 Q_{k,m} \cap ZF_{l^{*}}(f)   $ can be covered by $c(d) R_{l^{*}}  2^{(d-1)(R_{l^{*}}-k)} $ cubes of the form $ 2 Q_{R_{l^{*}},n}   $ with appropriate $ n \in  \mathbb{Z}^{d}  $. Here $ c(d)  $ only depends on $d$.  We used $ \lim_{x \rightarrow \infty}  x   2^{ \sigma (k- x )} = 0   $. Notice that under the given assumptions we always can choose $ l^{*} = l^{*}(f,R,p,u,d) $ independent of $ k $ and $ m $, see (v) in lemma \ref{lem_ZOAF}. Moreover we always can ensure that $  l^{*} < \infty  $, see (iv) in lemma \ref{lem_ZOAF} and its proof. Using such a natural number $  l^{*} $ and the associated covering we obtain
\begin{align*}
& 2^{k(\frac{d}{p} - \frac{d}{u})} \Vert \partial^{1}_{i} |f| \   \vert F^{s-1}_{p,q}(2 Q_{k,m} \cap ZF_{l^{*}}(f)) \Vert  \\
& \qquad \qquad \leq 2^{k(\frac{d}{p} - \frac{d}{u})} \sum_{n} \Vert \partial^{1}_{i} |f| \   \vert F^{s-1}_{p,q}(2 Q_{R_{l^{*}},n}) \Vert \\
& \qquad \qquad \leq C_{3} 2^{k(\frac{d}{p} - \frac{d}{u})}  R_{l^{*}} \  2^{(d-1)(R_{l^{*}}-k)} 2^{- R_{l^{*}} (\frac{d}{p} - \frac{d}{u})}  \sup_{n } 2^{R_{l^{*}}  (\frac{d}{p} - \frac{d}{u})}  \Vert \partial^{1}_{i} |f| \   \vert F^{s-1}_{p,q}(2 Q_{R_{l^{*}},n}) \Vert \\
& \qquad \qquad \leq C_{3} R_{l^{*}} \ 2^{(k- R_{l^{*}} )(\frac{d}{p} - \frac{d}{u} - d + 1)}  \sup_{\substack{j \geq R , m \in \mathbb{Z}^{d} \\ 2 Q_{j,m} \: \in \:  \mathcal{Q}_{SC}(f)  }} 2^{j(\frac{d}{p} - \frac{d}{u})}  \Vert \partial^{1}_{i} |f| \   \vert F^{s-1}_{p,q}(2 Q_{j,m}) \Vert . 
\end{align*}
In the last step we used $ R_{l^{*}} = R +  l^{*} > R  $. Now we apply the assumption $ 1/p - 1/u > 1 - 1/d   $. Because of $ d/p - d/u - d + 1  > 0    $ there exists a $ \sigma > 0    $ such that $  d/p - d/u - d + 1 - \sigma  > 0     $. Then since $ k \leq R < R_{l^{*}}  $ we get
\begin{align*}
R_{l^{*}} \ 2^{(k- R_{l^{*}} )(\frac{d}{p} - \frac{d}{u} - d + 1)} = R_{l^{*}} \  2^{ \sigma (k- R_{l^{*}} )}  2^{(k- R_{l^{*}} )(\frac{d}{p} - \frac{d}{u} - d + 1 - \sigma )} \leq 1 .
\end{align*}
Therefore we obtain
\begin{align*}
2^{k(\frac{d}{p} - \frac{d}{u})} \Vert \partial^{1}_{i} |f| \   \vert F^{s-1}_{p,q}(2 Q_{k,m} \cap ZF_{l^{*}}(f)) \Vert \leq C_{3}   \sup_{\substack{j \geq R , m \in \mathbb{Z}^{d} \\ 2 Q_{j,m} \: \in \:  \mathcal{Q}_{SC}(f)  }} 2^{j(\frac{d}{p} - \frac{d}{u})}  \Vert \partial^{1}_{i} |f| \   \vert F^{s-1}_{p,q}(2 Q_{j,m}) \Vert . 
\end{align*} 
Notice that the right-hand side is independent of $ l^{*}  $. So we do not have to deal with this parameter in what follows. All in all we find
\begin{align*}
& \sup_{\substack{j \in \mathbb{Z}, m \in \mathbb{Z}^{d} \\ 2 Q_{j,m} \: \in \:  \mathcal{Q}_{SC}(f)  }} 2^{j(\frac{d}{p} - \frac{d}{u})} \Vert \partial^{1}_{i} |f| \   \vert F^{s-1}_{p,q}(2 Q_{j,m}) \Vert \\
& \qquad \qquad \leq C_{4} \sup_{\substack{j \geq R , m \in \mathbb{Z}^{d} \\ 2 Q_{j,m} \: \in \:  \mathcal{Q}_{SC}(f)  }} \ldots  + C_{4} \sup_{\substack{ j \in \mathbb{Z} \setminus \mathbb{N} , m \in \mathbb{Z}^{d} \\ 2 Q_{j,m} \: \in \:  \mathcal{Q}_{SC}(f)  }} \ldots \ + C_{4} \Vert f   \vert \mathcal{E}^{s}_{u,p,q}(\R) \Vert   .
\end{align*}
When we define the set $ \mathbb{Z}(R) = \mathbb{Z} \setminus \mathbb{N} \  \cup \ \{ j \in \mathbb{N} : j \geq R   \} $ that means in what follows we only have to deal with the term 
\begin{equation}\label{split_of_2}
\sup_{\substack{ j \in \mathbb{Z}(R) , m \in \mathbb{Z}^{d} \\ 2 Q_{j,m} \: \in \:  \mathcal{Q}_{SC}(f)  }} 2^{j(\frac{d}{p} - \frac{d}{u})} \Vert \partial^{1}_{i} |f| \   \vert F^{s-1}_{p,q}(2 Q_{j,m}) \Vert .
\end{equation}

\textit{Step 4. Apply the result for the original Triebel-Lizorkin spaces.}

Now for all $ j \in \mathbb{Z}  $ and all $ m \in \mathbb{Z}^{d}  $ there exists a function $   f_{j,m} \in \mathbb{F}^{s}_{p,q}(\mathbb{R}^{d})  $ with $ f_{j,m}(x) = f(x)   $ for all $ x \in 2 Q_{j,m} $ such that 
\begin{align*}
 \Vert   f_{j,m}    \vert F^{s}_{p,q}(\R)   \Vert \leq 2 \Vert f \vert F^{s}_{p,q}(2 Q_{j,m}) \Vert.
\end{align*} 
Because of the definition of the norm $ \Vert \cdot \vert F^{s-1}_{p,q}(2 Q_{j,m}) \Vert  $ we can write
\begin{align*}
 \Vert \partial^{1}_{i} |f| \   \vert F^{s-1}_{p,q}(2 Q_{j,m}) \Vert & =  \Vert \partial^{1}_{i} |  f_{j,m} | \   \vert F^{s-1}_{p,q}(2 Q_{j,m}) \Vert  \leq    \Vert \partial^{1}_{i} |  f_{j,m} | \   \vert F^{s-1}_{p,q}(\R) \Vert.
\end{align*}
Next we can use a version of lemma \ref{l_deriv} for the original Triebel-Lizorkin spaces. Moreover because of $ 1 \leq p < \infty  $, $ 1 < s < 1 + 1/p   $ and the fact that $ f_{j,m} \in \mathbb{F}^{s}_{p,q}(\mathbb{R}^{d})   $ is real-valued we can apply the well-known theorem \ref{Hist_Res}. This leads to
\begin{align*}
\Vert \partial^{1}_{i} |f| \   \vert F^{s-1}_{p,q}(2 Q_{j,m}) \Vert & \leq C_{1}  \Vert  \  | f_{j,m}  |  \    \vert F^{s}_{p,q}(\R)   \Vert  \leq C_{2}   \Vert    f_{j,m}     \vert F^{s}_{p,q}(\R)   \Vert  \leq C_{3}  \Vert f \vert F^{s}_{p,q}(2 Q_{j,m}) \Vert.
\end{align*}
When we use this inequality for \eqref{split_of_2} we find
\begin{align*}
\sup_{\substack{ j \in \mathbb{Z}(R) , m \in \mathbb{Z}^{d} \\ 2 Q_{j,m} \: \in \:  \mathcal{Q}_{SC}(f)  }} 2^{j(\frac{d}{p} - \frac{d}{u})} \Vert \partial^{1}_{i} |f| \   \vert F^{s-1}_{p,q}(2 Q_{j,m}) \Vert \leq C_{4} \sup_{\substack{ j \in \mathbb{Z}(R) , m \in \mathbb{Z}^{d} \\ 2 Q_{j,m} \: \in \:  \mathcal{Q}_{SC}(f)  }} 2^{j(\frac{d}{p} - \frac{d}{u})} \Vert f \vert F^{s}_{p,q}(2 Q_{j,m}) \Vert.
\end{align*}
Hereafter we have to distinguish between small cubes and large cubes. Therefore we use the definition of the set $  \mathbb{Z}(R)  $ to write
\begin{align*}
& \sup_{\substack{ j \in \mathbb{Z}(R) , m \in \mathbb{Z}^{d} \\ 2 Q_{j,m} \: \in \:  \mathcal{Q}_{SC}(f)  }} 2^{j(\frac{d}{p} - \frac{d}{u})} \Vert f \vert F^{s}_{p,q}(2 Q_{j,m}) \Vert \\
& \qquad \leq \sup_{\substack{j \geq R , m \in \mathbb{Z}^{d} \\ 2 Q_{j,m} \: \in \:  \mathcal{Q}_{SC}(f)  }} 2^{j(\frac{d}{p} - \frac{d}{u})} \Vert f \vert F^{s}_{p,q}(2 Q_{j,m})   \Vert + \sup_{\substack{j \in \mathbb{Z} \setminus \mathbb{N} , \; m \in \mathbb{Z}^{d} \\ 2 Q_{j,m} \: \in \:  \mathcal{Q}_{SC}(f)  }} 2^{j(\frac{d}{p} - \frac{d}{u})} \Vert f \vert F^{s}_{p,q}(2 Q_{j,m}) \Vert.
\end{align*}
To continue the proof we have to deal with both terms separately. 

\textit{Step 5. Complete the proof for cubes of large and middle size.}

Here we have to deal with
\begin{align*}
\sup_{\substack{j \in \mathbb{Z} \setminus \mathbb{N} , \; m \in \mathbb{Z}^{d} \\ 2 Q_{j,m} \: \in \:  \mathcal{Q}_{SC}(f)  }} 2^{j(\frac{d}{p} - \frac{d}{u})} \Vert f \vert F^{s}_{p,q}(2 Q_{j,m}) \Vert.
\end{align*}
For that purpose we use formula (3.307) in \cite{Tr14}, see also proposition 4.21. in \cite{Tr08} and its proof. There we learn that for $ j \in \mathbb{Z} \setminus \mathbb{N}   $ and $  m \in \mathbb{Z}^{d}   $ there exists a general constant $ C_{1} $ independent of $ j $ and $ m $ such that
\begin{equation}\label{3.307}
\Vert f \vert F^{s}_{p,q}(2 Q_{j,m}) \Vert  \leq C_{1}  \Vert f \vert F^{s-1}_{p,q}(2 Q_{j,m}) \Vert + C_{1} \sum_{k = 1}^{d}   \Vert \partial^{1}_{k} f \vert F^{s-1}_{p,q}(2 Q_{j,m}) \Vert.
\end{equation}
Therefore we obtain
\begin{align*}
& \sup_{\substack{j \in \mathbb{Z} \setminus \mathbb{N}, \; m \in \mathbb{Z}^{d} \\ 2 Q_{j,m} \: \in \:  \mathcal{Q}_{SC}(f)  }}  2^{j(\frac{d}{p} - \frac{d}{u})}  \Vert f \vert F^{s}_{p,q}(2 Q_{j,m}) \Vert \\
& \qquad \leq C_{1} \sup_{j \in \mathbb{Z}, m \in \mathbb{Z}^{d} }  2^{j(\frac{d}{p} - \frac{d}{u})}  \Vert f \vert F^{s-1}_{p,q}(2 Q_{j,m}) \Vert + C_{1} \sum_{k = 1}^{d} \sup_{j \in \mathbb{Z}, m \in \mathbb{Z}^{d} }  2^{j(\frac{d}{p} - \frac{d}{u})}  \Vert \partial^{1}_{k} f \vert F^{s-1}_{p,q}(2 Q_{j,m}) \Vert.
\end{align*}
Now because of $  s - 1 < \min ( 1/p , d/u )   $ we can apply proposition \ref{E_MorChar} again. When we use lemma \ref{l_deriv} we find 
\begin{align*}
\sup_{\substack{j \in \mathbb{Z} \setminus \mathbb{N} , \; m \in \mathbb{Z}^{d} \\ 2 Q_{j,m} \: \in \:  \mathcal{Q}_{SC}(f)  }}  2^{j(\frac{d}{p} - \frac{d}{u})}  \Vert f \vert F^{s}_{p,q}(2 Q_{j,m}) \Vert & \leq C_{2}   \Vert f \vert \mathcal{E}^{s-1}_{u,p,q}(\mathbb{R}^{d})  \Vert + C_{2} \sum_{k = 1}^{d}   \Vert \partial^{1}_{k} f \vert \mathcal{E}^{s-1}_{u,p,q}(\mathbb{R}^{d})  \Vert \\
& \leq C_{3} \Vert f \vert \mathcal{E}^{s}_{u,p,q}(\mathbb{R}^{d})  \Vert. 
\end{align*}
So this step of the proof is complete. Notice that formula \eqref{3.307} also holds for $ j \in \mathbb{N}   $. But then the constant $ C_{1}  $ depends on $ j $ and tends to infinity if $ j $ tends to infinity. Whereas this is not a problem for small $ j \in \mathbb{N}  $ for the case of large $ j \in \mathbb{N}  $ we have go another way. 

\textit{Step 6. The case of small cubes.}

\textit{Substep 6.1. Construct auxiliary functions. }

Here we have to investigate the term 
\begin{equation}\label{pr_d>1_C7}
\sup_{\substack{j \geq R, m \in \mathbb{Z}^{d} \\ 2 Q_{j,m} \: \in \:  \mathcal{Q}_{SC}(f)  }}  2^{j(\frac{d}{p} - \frac{d}{u})}  \Vert f \vert F^{s}_{p,q}(2 Q_{j,m}) \Vert .
\end{equation}
Let $ j \in \mathbb{N}_{0} \cup \left\{ -1 \right\}  $ and $ c \geq 1  $. Since \eqref{pr_d>1_C7} is invariant under translation in what follows we can work with cubes denoted by $ c Q_{j}  $ that have side-length $ c \cdot 2^{-j}  $ and the center at the origin. Let $ h \in C^{\infty}_{0}(\R)  $ be a smooth function with $ h(x) = 1    $ for all $ x \in 2  Q_{0}  $ and $  h(x) = 0    $ for all $  x \not \in  8  Q_{0}  $ that fulfills $  \Vert h \vert C^{2}(\R) \Vert \leq C_{1}   $ for some constant $  C_{1} > 0  $. For $ j \in \mathbb{N}_{0}   $ we define
\begin{equation}\label{aux_fun}
h_{j}(x) = h(2^j x) \qquad \qquad \mbox{and} \qquad \qquad g_{j}(x) = f(x) \cdot h_{j}(x).
\end{equation}
In what follows we will investigate the properties of the functions $ g_{j} $.

\begin{itemize}
\item[(i)] For all $ x \in 2 Q_{j}  $ we have $  g_{j}(x) = f(x)  $. Moreover we find $ \supp g_{j} \subset 8 Q_{j}  $.

\item[(ii)] The functions $ g_{j} $ are bounded. More precisely we find
\begin{align*}
\Vert g_{j} \vert  L_{\infty}(\R) \Vert \leq \sup_{x \in 8 Q_{j}} | f(x)  | \  | h_{j}(x) | \leq  C_{1} \Vert f \vert  L_{\infty}(8 Q_{j}) \Vert.
\end{align*}

\item[(iii)] The functions $ g_{j}  $ are smooth. For the first derivatives with $ | \alpha | = 1   $ we find
\begin{align*}
\Vert D^{\alpha} g_{j} \vert  L_{\infty}(\R) \Vert & \leq \sup_{x \in 8 Q_{j}} \Big ( | D^{\alpha} f(x) | \  | h_{j}(x) | +  | D^{\alpha} h_{j}(x)|  \  | f(x)    | \Big ) \\
& \leq   C_{1} \Vert f  \vert C^{1}(8 Q_{j}) \Vert   + C_{1}  2^{j}   \  \sup_{x \in 8 Q_{j}}  | f(x)    | .
\end{align*}
In view of \eqref{pr_d>1_C7} we can assume $  2 Q_{j} \in \:  \mathcal{Q}_{SC}(f)  $. Therefore there exists a $ z \in 2 Q_{j}  $ such that $ f(z) = 0   $. Because of $  f   $ is smooth for each $ y \in 8 Q_{j}   $ by the mean value theorem we find that there exists a constant $  C_{2}  $ that is independent of $ j $ and $ f $ such that
\begin{equation}\label{step4_(iii)}
\vert f(y) \vert \leq  C_{2} \Vert f  \vert C^{1}(8 Q_{j}) \Vert \; \vert y - z   \vert \leq  C_{2} \Vert f  \vert C^{1}(8 Q_{j}) \Vert   2 d \cdot 2^{-j}. 
\end{equation}
So we can conclude $ \Vert D^{\alpha} g_{j} \vert  L_{\infty}(\R) \Vert   \leq   C_{3} \Vert f  \vert C^{1}(8 Q_{j}) \Vert $.

\item[(iv)] For the second derivatives with $ | \alpha | = 1   $ and    $ | \beta | = 1   $ when we use $ 2 Q_{j} \in \:  \mathcal{Q}_{SC}(f)   $ again we observe
\begin{align*}
& \Vert D^{\alpha} D^{\beta} g_{j} \vert  L_{\infty}(\R) \Vert \\
&  \qquad   \qquad \leq   C_{1} \Vert f  \vert C^{2}(8 Q_{j}) \Vert  +  C_{4} 2^{j} \Vert f  \vert C^{1}(8 Q_{j}) \Vert    + \sup_{x \in 8 Q_{j}}  | D^{\alpha} D^{\beta}  h_{j}(x)|  \  | f(x)    | \\
&  \qquad   \qquad \leq   C_{1} \Vert f  \vert C^{2}(8 Q_{j}) \Vert  +  C_{4} 2^{j} \Vert f  \vert C^{1}(8 Q_{j}) \Vert    +   C_{5} 2^{2j} 2^{-j}    \Vert f  \vert C^{1}(8 Q_{j}) \Vert   \\
&  \qquad   \qquad \leq   C_{6} \max(\Vert f  \vert C^{2}(8 Q_{j}) \Vert  , 2^{j} \Vert f  \vert C^{1}(8 Q_{j}) \Vert ).
\end{align*}

\end{itemize}

In what follows it will turn out that the properties (i) - (iv) of the functions $ g_{j} $ are exactly what we need to continue our proof.

\textit{Substep 6.2. Use the smoothness properties of $f$ and $g_{j}$. }

Now we use the functions $ g_{j} $ we have constructed before to deal with the term \eqref{pr_d>1_C7}. We want to apply remark 2.12. from \cite{Tr08}. From there we learn the following. Let $ U_{\lambda} = \{ x \in \R : | x_{r} |  < \lambda \}   $ with $ 0 < \lambda \leq 1   $. Then for $ s > 0$, $ 1 \leq p < \infty $ and $ 1 \leq q \leq \infty  $ we find $ \Vert  f ( \lambda \cdot ) \vert  F^{s}_{p,q}(\R)  \Vert \sim \lambda^{s - d/p} \Vert  f  \vert  F^{s}_{p,q}(\R)  \Vert $ for $ f \in F^{s}_{p,q}(\R)  $ with  $  \supp f \subset U_{\lambda} $. This is the so called local homogeneity property. When we use it in our case because of $ \supp  g_{j} \subset 8 Q_{j}   $ and the definition of the norm $  \Vert  \cdot \vert F^{s}_{p,q}(8 Q_{j}) \Vert  $ we get
\begin{align*}
 2^{j(\frac{d}{p} - \frac{d}{u})}  \Vert f \vert F^{s}_{p,q}(2 Q_{j}) \Vert & \leq 2^{j(\frac{d}{p} - \frac{d}{u})}  \Vert  g_{j} \vert F^{s}_{p,q}(8 Q_{j}) \Vert  \leq C_{1} 2^{j(s - \frac{d}{u})}   \Vert  g_{j}(2^{-j+2} \cdot ) \vert F^{s}_{p,q}( Q_{-1}) \Vert. 
\end{align*}
Because of $ p \geq 1 $, $ q \geq 1 $ and $ s > 0 $ it is possible to describe the space $ \mathbb{F}^{s}_{p,q}( Q_{-1})  $ in terms of differences, see theorem 1.118. in \cite{Tr06}. Then we obtain
\begin{align*}
& \Vert g_{j}(2^{-j+2} \cdot ) \vert F^{s}_{p,q}(Q_{-1}) \Vert \\
& \leq C_{2} \Vert g_{j}(2^{-j+2} \cdot ) \vert L_{p}(Q_{-1}) \Vert + C_{2} \Big \Vert \Big ( \int_{0}^{1} t^{-sq} \Big ( t^{-d} \int_{ V^{2}} \vert \Delta^{2}_{h} g_{j}(2^{-j+2} \cdot ) \vert dh   \Big )^{q}  \frac{dt}{t} \Big )^{\frac{1}{q}} \Big \vert L_{p}(Q_{-1}) \Big \Vert
\end{align*}
with $ V^{2} = V^{2}(x,t) = \left\{  h \in \R   : \vert h \vert < t \  \mbox{and} \ x + \tau h \in Q_{-1} \ \mbox{for} \ 0 \leq \tau \leq 2    \right\}  $. Notice that the constant $ C_{2}  $ is independent of $  j $. Now on the one hand thanks to (ii) from our construction of $g_{j}$ and a transformation of the coordinates we find
\begin{align*}
2^{j(s - \frac{d}{u})}  \Vert g_{j}(2^{-j+2} \cdot ) \vert L_{p}(Q_{-1}) \Vert & \leq C_{3} 2^{j(\frac{d}{p} - \frac{d}{u})} 2^{js} \Vert g_{j} \vert L_{p}(8 Q_{j}) \Vert  \leq C_{4} 2^{ - j \frac{d}{u}} 2^{js} \Vert f \vert  L_{\infty}(8 Q_{j}) \Vert.
\end{align*}
Like before, see formula \eqref{step4_(iii)}, for all $ y \in 8 Q_{j} $ we observe $ \vert f(y) \vert \leq  C_{5} \Vert f  \vert C^{1}(8 Q_{j}) \Vert   2 d \cdot 2^{-j}$. Hence we conclude
\begin{align*}
2^{j(s - \frac{d}{u})}  \Vert g_{j}(2^{-j+2} \cdot ) \vert L_{p}(Q_{-1}) \Vert & \leq C_{6} \Vert f  \vert C^{1}(8 Q_{j}) \Vert 2^{ - j \frac{d}{u}} 2^{js} 2^{-j}  = C_{6} \Vert f  \vert C^{1}(8 Q_{j}) \Vert  2^{j(s- 1- \frac{d}{u} )}.
\end{align*}
Now we have to deal with the term that contains differences.  When we use the well-known formula \eqref{diff_abl_c} in combination with (iv) from our construction of $g_{j}$ thanks to some transformations of the coordinates we obtain
\begin{align*}
& 2^{j(\frac{d}{p} - \frac{d}{u})} 2^{-j(\frac{d}{p} - s)} \Big \Vert \Big ( \int_{0}^{1} t^{-sq} \Big ( t^{-d} \int_{ V^{2}(x,t)} \vert ( \Delta^{2}_{h} g_{j}(2^{-j+2} \cdot ))(x) \vert dh   \Big )^{q}  \frac{dt}{t} \Big )^{\frac{1}{q}} \Big \vert L_{p}(Q_{-1}) \Big \Vert \\
& \qquad \leq C_{7} 2^{j(\frac{d}{p} - \frac{d}{u})}  \Big \Vert \Big ( \int_{0}^{c 2^{-j}} t^{-sq} \Big ( t^{-d} \int_{ |h| < t} \vert (\Delta^{2}_{h} g_{j})(x) \vert dh   \Big )^{q}  \frac{dt}{t} \Big )^{\frac{1}{q}} \Big \vert L_{p}(8 Q_{j}) \Big \Vert \\
& \qquad \leq C_{8}  \max(\Vert f  \vert C^{2}(8 Q_{j}) \Vert  , 2^{j} \Vert f  \vert C^{1}(8 Q_{j}) \Vert ) \;   2^{j(\frac{d}{p} - \frac{d}{u})}   \Big \Vert \Big ( \int_{0}^{c 2^{-j}} t^{-sq+2q}   \frac{dt}{t} \Big )^{\frac{1}{q}} \Big \vert L_{p}(8 Q_{j}) \Big \Vert \\
& \qquad \leq C_{9}   \max(\Vert f  \vert C^{2}(8 Q_{j}) \Vert  , 2^{j} \Vert f  \vert C^{1}(8 Q_{j}) \Vert ) \;  2^{j(\frac{d}{p} - \frac{d}{u})}  2^{-j(2-s)}  2^{-j \frac{d}{p}}  \\
& \qquad = C_{10}   \max( 2^{-j} \Vert f  \vert C^{2}(8 Q_{j}) \Vert  ,  \Vert f  \vert C^{1}(8 Q_{j}) \Vert ) \;   2^{j(s- 1- \frac{d}{u} )}. 
\end{align*}
Up to know we proved that for all $ j \in \mathbb{N}, m \in \mathbb{Z}^d   $ and all cubes $ 2 Q_{j,m} \: \in \:  \mathcal{Q}_{SC}(f)   $ we have
\begin{equation}\label{pr_d>1_j0}
 2^{j(\frac{d}{p} - \frac{d}{u})}  \Vert f \vert F^{s}_{p,q}(2 Q_{j,m}) \Vert \leq C_{11}    \max( 2^{-j} \Vert f  \vert C^{2}(\R) \Vert  ,  \Vert f  \vert C^{1}(\R) \Vert ) \;   2^{j(s- 1- \frac{d}{u} )} .
\end{equation}
Notice that we have $ s < 1 + d/u   $ and so $ s - 1 - d/u < 0   $. So for large $ j \in \mathbb{N}  $ that tend to $ \infty $ the right hand side tends to zero. In what follows we will make this more precise.

\textit{Substep 6.3. Use that $ \mathcal{F}f  $ has compact support.  }

Now we have to deal with the terms $ \Vert f  \vert C^{1}(\R) \Vert  $ and $ \Vert f  \vert C^{2}(\R) \Vert  $. Let us start with the first one. Let $ | \alpha | \leq 1   $ and $(\varphi_k)_{k\in \N_0 }$ be a smooth dyadic decomposition of the unity. Because of $   f \in \mathbb{E}^{s}_{u,p,q}(\mathbb{R}^{d})  $ is smooth we find
\begin{align*}
\Vert  D^{\alpha}  f  \vert  L_{\infty}(\R) \Vert & = \Big \Vert \sum_{k = 0}^{\infty}  \mathcal{F}^{-1}[\varphi_{k} \mathcal{F} ( D^{\alpha}  f) ]  \Big \vert    L_{\infty}(\R) \Big \Vert  \leq \sum_{k = 0}^{\infty}   \Vert   \mathcal{F}^{-1}[\varphi_{k} \mathcal{F} ( D^{\alpha}  f) ]   \vert    L_{\infty}(\R)  \Vert. 
\end{align*}
Now since $ \supp \varphi_{k} \mathcal{F} ( D^{\alpha}  f) \subset B(0,2^{k+1})   $ we can use formula (7) from the proof of corollary 2.3 in \cite{SawTan}, see also proposition 3.7 in \cite{SawTan}. Then we obtain 
\begin{align*}
\Vert  D^{\alpha}  f  \vert  L_{\infty}(\R) \Vert & \leq C_{1} \sum_{k = 0}^{\infty} 2^{k \frac{d}{u}}   \Vert   \mathcal{F}^{-1}[\varphi_{k} \mathcal{F} ( D^{\alpha}  f) ]   \vert    \mathcal{M}^{u}_{p}(\R)  \Vert.
\end{align*}
Recall that we have $ \supp \mathcal{F} f \subset B(0,2^{R})   $ for some $ R \in \mathbb{N}   $. By standard arguments we find that this property carries over to $  \mathcal{F} ( D^{\alpha}  f)  $. Because of this for $ \sigma > 0  $ we also can write
\begin{align*}
\Vert  D^{\alpha}  f  \vert  L_{\infty}(\R) \Vert & \leq C_{1} \sum_{k = 0}^{R + 1} 2^{k (\frac{d}{u}-s+1)} 2^{k \sigma} 2^{k(s-1-\sigma)}   \Vert   \mathcal{F}^{-1}[\varphi_{k} \mathcal{F} ( D^{\alpha}  f) ]   \vert    \mathcal{M}^{u}_{p}(\R)  \Vert \\
& \leq C_{2} 2^{R (\frac{d}{u}-s+1)} 2^{R \sigma} \sum_{k = 0}^{\infty} 2^{k(s-1-\sigma)}   \Vert   \mathcal{F}^{-1}[\varphi_{k} \mathcal{F} ( D^{\alpha}  f) ]   \vert    \mathcal{M}^{u}_{p}(\R)  \Vert \\
& \leq C_{3} 2^{R (\frac{d}{u}-s+1)} 2^{R \sigma} \Vert D^{\alpha}  f \vert \mathcal{N}^{s-1-\sigma}_{u,p,1}(\R)  \Vert \\
& \leq C_{4} 2^{R (\frac{d}{u}-s+1)} 2^{R \sigma} \Vert D^{\alpha}  f \vert \mathcal{E}^{s-1}_{u,p,q}(\R)  \Vert .
\end{align*}
In the last steps we used the definition of the Besov-Morrey spaces and proposition 3.6 from \cite{SawTan}. Let us put $ \sigma = 1/R   $. Then since $   | \alpha | \leq 1   $ with lemma \ref{l_deriv} we get
\begin{equation}\label{subs_4.3.a=1}
\Vert  D^{\alpha}  f  \vert  L_{\infty}(\R) \Vert \leq  C_{5} 2^{R (\frac{d}{u}-s+1)}  \Vert   f \vert \mathcal{E}^{s}_{u,p,q}(\R)  \Vert .
\end{equation}
Next we have to deal with $ \Vert f  \vert C^{2}(\R) \Vert   $. Therefore we have to investigate $ \Vert  D^{\beta}  f  \vert  L_{\infty}(\R) \Vert    $ with $ \vert \beta \vert = 2   $. Since $ f $ is smooth this can be done in the same way as before. One obtains
\begin{equation}\label{subs_4.3.a=2}
\Vert  D^{\beta}  f  \vert  L_{\infty}(\R) \Vert \leq  C_{6} 2^{R (\frac{d}{u}-s+1)} 2^{R}  \Vert   f \vert \mathcal{E}^{s}_{u,p,q}(\R)  \Vert .
\end{equation}
If we combine this with formula \eqref{pr_d>1_j0} and \eqref{subs_4.3.a=1} for all $ j \in \mathbb{N}, m \in \mathbb{Z}^d   $ and all cubes $ 2 Q_{j,m} \: \in \:  \mathcal{Q}_{SC}(f)   $ we find
\begin{align*}
2^{j(\frac{d}{p} - \frac{d}{u})}  \Vert f \vert F^{s}_{p,q}(2 Q_{j,m}) \Vert & \leq C_{7}    \max( 2^{-j} \Vert f  \vert C^{2}(\R) \Vert  ,  \Vert f  \vert C^{1}(\R) \Vert ) \;   2^{j(s- 1- \frac{d}{u} )} \\ 
& \leq C_{8}    \max( 2^{-j+R}   ,  1 ) 2^{R (\frac{d}{u}-s+1)}  2^{j(s- 1- \frac{d}{u} )} \Vert   f \vert \mathcal{E}^{s}_{u,p,q}(\R)  \Vert  .
\end{align*}
Now because of $ s - 1 - d/u < 0   $ in view of \eqref{pr_d>1_C7}  we obtain 
\begin{align*}
& \sup_{\substack{j \geq R , m \in \mathbb{Z}^{d} \\ 2 Q_{j,m} \: \in \:  \mathcal{Q}_{SC}(f)  }} 2^{j(\frac{d}{p} - \frac{d}{u})} \Vert f  \vert F^{s}_{p,q}(2 Q_{j,m}) \Vert \\
& \qquad  \leq C_{9} \sup_{j \geq R , m \in \mathbb{Z}^{d}}  \max( 2^{-j+R}   ,  1 ) 2^{R (\frac{d}{u}-s+1)}  2^{j(s- 1- \frac{d}{u} )} \Vert   f \vert \mathcal{E}^{s}_{u,p,q}(\R)  \Vert \\
& \qquad  \leq C_{10}  \Vert   f \vert \mathcal{E}^{s}_{u,p,q}(\R)  \Vert.
\end{align*} 
So this step and the whole proof are complete.
\end{proof}

\begin{rem}\label{rem_d>1_excon}
In the formulation of proposition \ref{d>1_E_MR} we can find two different types of conditions. On the one hand we assume $ s < \min  ( 1 + 1/p ,  1 + d/u  )$. Later we will see that this condition is necessary, see proposition \ref{d>1_Nec_MR}. On the other hand there is the restriction $ 1/p  -  1/u  > 1 -  1/d $. It is possible that this condition is of technical nature only and can be avoided by using another method for the proof. Notice that this assumption can be left away when we in addition assume $ \supp \mathcal{F}  f \subset B(0,2^{R})   $ with fixed $ 0 < R < \infty   $ in the formulation of proposition \ref{d>1_E_MR}.
\end{rem}

Under some additional conditions it is possible to prove a version of proposition \ref{d>1_E_MR} for the special case $ s = 1  $. Especially for $ 1 < p \leq u < \infty   $, $ q = 2  $ and $ s = 1  $ this is important. In this case the Triebel-Lizorkin-Morrey spaces coincide with the so-called Sobolev-Morrey spaces, see definition 9 and theorem 3.1. in \cite{Si_sur1}. We can prove the following result. 

\begin{prop}\label{s=1_E_MR}
Let $ 1 < p < u < \infty  $, $ 1 \leq q < \infty $ and $ s = 1 $. Let $ 1/p - 1/u > 1 - 1/d  $. Then there is a constant $ C > 0 $ independent of $ f \in \mathbb{E}^{1}_{u,p,q}(\mathbb{R}^{d})  $ such that we have 
\begin{equation}\label{d>1_E_eq2}
\Vert Tf  \vert \mathcal{E}^{1}_{u,p,q}(\mathbb{R}^{d})  \Vert \leq C \Vert f \vert \mathcal{E}^{1}_{u,p,q}(\mathbb{R}^{d})  \Vert
\end{equation}
for all $ f \in \mathbb{E}^{1}_{u,p,q}(\mathbb{R}^{d})   $.
\end{prop}

\begin{proof}
This result can be proved in the same way as proposition \ref{d>1_E_MR}. All arguments that are used there also hold in the situation of proposition \ref{s=1_E_MR}. The reason for this is that our main tool theorem \ref{Hist_Res} also is valid for $ s = 1 $ when we assume $ 1 < p < \infty   $. When we follow the strategy described in the proof of proposition \ref{d>1_E_MR} because of $ s - 1 = 0   $ we sometimes have to work with smoothness zero. Let us explain why this is not a problem. At first we mention that thanks to the Fatou property we can work with smooth $  C^{\infty} -  $ functions. Therefore we do not have to deal with singular distributions. Next we should notice that lemma \ref{l_deriv} also holds for smoothness zero, see corollary 3.4 in \cite{SawTan} or theorem 2.15 in \cite{TangXu}. Our tool proposition \ref{E_MorChar} is valid for $ s - 1 = 0   $ as well, see theorem 3.64 in \cite{Tr14}. Here the restrictions $ p \not = 1   $ and $ q \not = \infty    $ are required. Finally we mention that because of $ \mathcal{E}^{1}_{u,p,q}(\R) \hookrightarrow  \mathcal{E}^{1/2}_{u,p,q}(\R) \hookrightarrow \mathcal{E}^{0}_{u,p,q}(\R)    $ and proposition \ref{pro_0<s<1} instead of formula \eqref{s=1_change} we can write 
\begin{align*}
\Vert \  |f| \ \vert \mathcal{E}^{0}_{u,p,q}(\mathbb{R}^{d})  \Vert \leq  C_{1}   \Vert \  |f| \ \vert \mathcal{E}^{1/2}_{u,p,q}(\mathbb{R}^{d})  \Vert \leq  C_{2}   \Vert f \vert \mathcal{E}^{1/2}_{u,p,q}(\mathbb{R}^{d})  \Vert \leq   C_{3}   \Vert f \vert \mathcal{E}^{1}_{u,p,q}(\mathbb{R}^{d})  \Vert  .
\end{align*}
This simple observation completes the proof.
\end{proof}

\begin{rem}\label{rem_s1_p1}
Notice that proposition \ref{s=1_E_MR} does not cover the special case $ p = 1 $ for $ s = 1 $. On the other hand for the original Sobolev spaces with $ u = p = 1  $, $ q = 2  $ and $  s = 1  $ Marcus and Mizel proved in \cite{MaMiz} that $ T $ is a continuous operator, see theorem 1 in \cite{MaMiz}.
\end{rem}

\subsection{The boundedness of the operator $ T $ and Besov-Morrey spaces}\label{sec_T_BMS}

In this section we will study the boundedness properties of the operator $ T $ in the context of Besov-Morrey spaces. For that purpose our main tool will be real interpolation. There is the following result.

\begin{prop}\label{d>1_N_MR}
Let $ 1 \leq p < u < \infty  $ and $ 1 \leq q \leq \infty   $. Assume $ 1/p - 1/u > 1 - 1/d   $ in the case of $ d > 1  $. Let 
\begin{align*}
0 < s < \min \Big ( 1 + \frac{1}{p} ,  1 + \frac{d}{u} \Big ). 
\end{align*}
Then there is a constant $ C > 0 $ independent of $ f \in \mathbb{N}^{s}_{u,p,q}(\mathbb{R}^{d})  $ such that we have 
\begin{equation}\label{d>1_N_eq1}
\Vert Tf  \vert \mathcal{N}^{s}_{u,p,q}(\mathbb{R}^{d})  \Vert \leq C \Vert f \vert \mathcal{N}^{s}_{u,p,q}(\mathbb{R}^{d})  \Vert
\end{equation}
for all $ f \in \mathbb{N}^{s}_{u,p,q}(\mathbb{R}^{d})   $.
\end{prop}

\begin{proof}
For the proof we use the corresponding result for the Triebel-Lizorkin-Morrey spaces. Moreover we apply a result concerning the real interpolation of Lipschitz continuous operators that goes back to Peetre, see \cite{Pe1}. One may also consult proposition 1 in chapter 2.5.4 in \cite{RS}. Here for us it is convenient to follow the explanations given in step 5 of the proof from theorem 25.8 in \cite{Tr01}. We use the same notation as there. We put $ Tf = \vert f \vert   $ and $ A_{0} =  \mathbb{M}^{u}_{p}(\mathbb{R}^{d}) $ and $ A_{1} = \mathbb{E}^{s_{1}}_{u,p,1}(\mathbb{R}^{d})   $. From lemma \ref{E_Mor_em} we learn that we have $ \mathbb{E}^{s_{1}}_{u,p,1}(\mathbb{R}^{d}) \hookrightarrow \mathbb{M}^{u}_{p}(\mathbb{R}^{d})   $ for $ 1 \leq p \leq u < \infty $ and $ s_{1} > 0  $. Moreover from the propositions \ref{pro_0<s<1},  \ref{pro_d1s>1} and \ref{d>1_E_MR} we learn that for $ 1 \leq p < u < \infty  $ and $ 0 < s_{1} < \min ( 1 + 1/p , 1 + d/u ) \leq 2  $ with $ s_{1} \not = 1  $ we have $ \Vert Tf  \vert \mathcal{E}^{s_{1}}_{u,p,1}(\mathbb{R}^{d})  \Vert \leq C \Vert f \vert \mathcal{E}^{s_{1}}_{u,p,1}(\mathbb{R}^{d})  \Vert $ for all $ f \in \mathbb{E}^{s_{1}}_{u,p,1}(\mathbb{R}^{d}) $. Here in the case of $ d > 1  $ the assumption $ 1/p - 1/u > 1 - 1/d   $ is needed. Furthermore because of the triangle inequality we have 
\begin{align*}
\Vert |f| - |g|  \vert \mathcal{M}^{u}_{p}(\mathbb{R}^{d}) \Vert \leq \Vert f - g  \vert \mathcal{M}^{u}_{p}(\mathbb{R}^{d}) \Vert
\end{align*}
for all $ f,g \in \mathbb{M}^{u}_{p}(\mathbb{R}^{d})  $. Then like it is described in step 5 of the proof from theorem 25.8 in \cite{Tr01} for all $ 0 < \theta < 1   $ and $ 1 \leq q \leq \infty   $ we have
\begin{align*}
\Vert  Tf  \vert  ( \mathbb{M}^{u}_{p}(\mathbb{R}^{d}) ,  \mathbb{E}^{ s_{1}}_{u,p,1}(\mathbb{R}^{d})   )_{\theta , q}  \Vert \leq C \Vert  f  \vert  ( \mathbb{M}^{u}_{p}(\mathbb{R}^{d}) ,  \mathbb{E}^{ s_{1}}_{u,p,1}(\mathbb{R}^{d})   )_{\theta , q} \Vert
\end{align*}
for all $ f \in  ( \mathbb{M}^{u}_{p}(\mathbb{R}^{d}) ,  \mathbb{E}^{ s_{1}}_{u,p,1}(\mathbb{R}^{d})   )_{\theta , q}  $. Now we apply lemma \ref{re_inter} which tells us that we have
\begin{align*}
\mathbb{N}^{\theta s_{1}}_{u,p,q}(\mathbb{R}^{d}) = \Big ( \mathbb{M}^{u}_{p}(\mathbb{R}^{d}) ,  \mathbb{E}^{ s_{1}}_{u,p,1}(\mathbb{R}^{d})  \Big )_{\theta , q}
\end{align*}
with $ 0 < \theta < 1  $ and $ 1 \leq q \leq \infty  $. Because of $ 0 < s_{1} < \min ( 1 + 1/p , 1 + d/u ) $ and $ s_{1} \not = 1  $ we have $ 0 < \theta s_{1} <  \min ( 1 + 1/p , 1 + d/u )  $. So the proof is complete.
\end{proof}

\begin{rem}\label{rem_hist_m}
For the original Besov spaces there exist several different methods to prove results like proposition \ref{d>1_N_MR}. In \cite{BouMey} and in chapter 5.4.1 in \cite{RS} a Hardy-type inequality in combination with real interpolation was applied. In \cite{Os1} some tools from approximation theory for linear splines are used to prove a result for the spaces $ \mathbb{B}^s_{p,q}(\mathbb{R})  $. A third method using atoms can be found in \cite{Tr01}, see theorem 25.8. 

\end{rem}

\subsection{The boundedness of $ T^{+} $ and $ T $ : Necessary conditions}\label{subsec_Nec}

When you look at the propositions \ref{pro_d1s>1}, \ref{d>1_E_MR} and \ref{d>1_N_MR} you always will find the condition $ s < \min  ( 1 + 1/p , 1 + d/u  )$. In this section we will investigate whether this condition is also necessary. For that purpose we will deal with some special test functions. Let us start with the Triebel-Lizorkin-Morrey spaces.

\begin{prop}\label{d>1_Nec_MR}
Let $ 1 \leq p \leq u < \infty  $ and $ 1 \leq q \leq \infty  $. Let either $ s \geq 1 + 1/p $ or $ s > 1 + d/u $. Then there exists a function $ f \in  \mathbb{E}^{s}_{u,p,q}(\mathbb{R}^{d}) $ such that $ Tf \not \in \mathbb{E}^{s}_{u,p,q}(\mathbb{R}^{d})  $ and $ T^{+}f \not \in \mathbb{E}^{s}_{u,p,q}(\mathbb{R}^{d})  $. 
\end{prop}

\begin{proof}
Let $ (x_{1} , x_{2} , \ldots , x_{d}) = x \in \R  $. Then we define a real-valued function $  f \in C_{0}^{\infty}(\R)  $ such that 
\begin{equation}\label{Ex_Fu_d>1_E1}
f(x) = x_{1} \quad \mbox{for} \quad |x| < 10 d (s+2) \quad \mbox{and}  \quad  f(x) = 0 \quad \mbox{for} \quad |x| > 11 d (s+2). 
\end{equation} 
Because of lemma \ref{l_bp1} we have $ f \in  \mathbb{E}^{s}_{u,p,q}(\mathbb{R}^{d})  $. In what follows we want to prove $ Tf \not \in \mathbb{E}^{s}_{u,p,q}(\mathbb{R}^{d})   $ and $ T^{+}f \not \in \mathbb{E}^{s}_{u,p,q}(\mathbb{R}^{d})   $ simultaneously. For that purpose we write $ T^{*}  $ when we mean either $ T $ or $ T^{+} $. Because of $ p \geq 1   $ and $ q \geq 1  $ we can apply proposition \ref{pro_E_diff} with $ v = 1    $, $ a = \infty   $ and $ N \in \mathbb{N}   $ with $ s < N < s + 2   $. Notice that we always have $ N \geq 2  $. We write $ (h_{1} , h_{2} , \ldots , h_{d}) = h \in \R  $. Then we find
\begin{align*}
& \Vert    T^{*}f    \vert   \mathcal{E}^{s}_{u,p,q}(\mathbb{R}^{d})   \Vert \\
& \qquad \geq C_{1} \Big \Vert \Big (  \int_{0}^{\infty}  t^{-sq} \Big ( t^{-d} \int_{B(0,t)}\vert \Delta^{N}_{h}T^{*}f(x) \vert dh \Big )^{q}  \frac{dt}{t}  \Big )^{\frac{1}{q}} \Big \vert \mathcal{M}^{u}_{p}( \mathbb{R}^d)  \Big \Vert \\
& \qquad \geq C_{2} \sup_{\substack{P \;  \mbox{\tiny{dyadic cube}}  \\ P \subset [0,1]^{d}  }} \vert P \vert^{\frac{1}{u}-\frac{1}{p}} \Big ( \int_{P} \Big (  \int_{\frac{3}{2} x_{1}}^{2 x_{1}}  t^{-sq} \Big ( t^{-d} \int_{\substack{|h| \leq t \\ h_{1} \leq -x_{1}}}\vert \Delta^{N}_{h}T^{*}f(x) \vert dh \Big )^{q}   \frac{dt}{t}  \Big )^{\frac{p}{q}} dx \Big )^{\frac{1}{p}}.
\end{align*}
Let $ x \in  [0,1]^{d}$. Then we have $  |f(x)| = x_{1}  $ and $ \max(f(x),0) = x_{1}  $. Moreover let $ 3/2 \; x_{1} < t < 2  x_{1}   $ and $ |h| \leq t   $ with $  h_{1} \leq -x_{1} $. Then for $ l \in \left\{ 1, 2, \ldots, N  \right\}   $ we observe $ |f(x+lh)| = - x_{1} - l h_{1} $ and $ \max(f(x+lh),0) = 0  $. Recall that for $ N \geq 2  $ there are the elementary formulas 
\begin{align*}
\sum_{l = 0}^{N} (-1)^{l} { N \choose l } = 0 \qquad \mbox{and} \qquad \sum_{l = 0}^{N} (-1)^{l} { N \choose l } l = 0.
\end{align*}
Hence for the operators $ T $ and $ T^{+}  $ in the case $ N \geq 2  $ we get
\begin{align*}
\vert \Delta^{N}_{h }|f|(x) \vert  = 2 x_{1} \qquad \mbox{and} \qquad \vert \Delta^{N}_{h }T^{+}f(x) \vert  =  x_{1}.
\end{align*}
So we have almost the same outcome for $ T $ and $ T^{+} $. Therefore for both cases we can proceed in the same way now. For $3/2 \; x_{1} < t < 2  x_{1}   $ we have $  t^{-d} \geq C_{3} x_{1}^{-d}  $. Moreover there exists a constant $ C_{d} $ that depends on $ d $ such that $ \int_{|h| \leq t, \  h_{1} \leq -x_{1}} 1 dh \geq C_{d} x_{1}^{d}  $. So we obtain
\begin{align*}
\Vert T^{*}f    \vert   \mathcal{E}^{s}_{u,p,q}(\mathbb{R}^{d})   \Vert & \geq C_{4} \sup_{\substack{P \;  \mbox{\tiny{dyadic cube}}  \\ P \subset [0,1]^{d}  }} \vert P \vert^{\frac{1}{u}-\frac{1}{p}} \Big ( \int_{P} x_{1}^{p} \Big (  \int_{\frac{3}{2} x_{1}}^{2 x_{1}}  t^{-sq-1}   dt  \Big )^{\frac{p}{q}} dx \Big )^{\frac{1}{p}} \\
&  \geq C_{5} \sup_{\substack{P \;  \mbox{\tiny{dyadic cube}}  \\ P \subset [0,1]^{d}  }} \vert P \vert^{\frac{1}{u}-\frac{1}{p}} \Big ( \int_{P} x_{1}^{p - sp}  dx \Big )^{\frac{1}{p}}.
\end{align*}
In what follows we are only interested in dyadic cubes that look like $ 2^{-j}[0,1)^{d} $ with $ j \in \mathbb{N}_{0}  $. Then we can apply Fubini's theorem and get
\begin{align*}
& \Vert T^{*}f    \vert   \mathcal{E}^{s}_{u,p,q}(\mathbb{R}^{d})   \Vert \\
& \qquad \geq C_{6} \sup_{j \in \mathbb{N}_{0} } 2^{-jd( \frac{1}{u}-\frac{1}{p})} \Big ( \int_{0}^{2^{-j}} x_{1}^{p-sp}  dx_{1}  \  \int_{0}^{2^{-j}} \cdots  \int_{0}^{2^{-j}} 1 dx_{2} \cdots dx_{d}  \Big )^{\frac{1}{p}} \\
& \qquad = C_{6} \sup_{j \in \mathbb{N}_{0} } 2^{-jd( \frac{1}{u}-\frac{1}{p})} 2^{-j(d-1) \frac{1}{p}} \Big ( \int_{0}^{2^{-j}} x_{1}^{p-sp}  dx_{1}    \Big )^{\frac{1}{p}}.
\end{align*}
If we are in the case $ s \geq 1 + 1/p   $ the integral is infinite and the proof is complete. So in what follows we assume $ s < 1 + 1/p   $ but $   s > 1 + d/u  $. Then we find
\begin{align*}
\Vert T^{*}f    \vert   \mathcal{E}^{s}_{u,p,q}(\mathbb{R}^{d})   \Vert   &  \geq C_{7} \sup_{j \in \mathbb{N}_{0} } 2^{-jd( \frac{1}{u}-\frac{1}{p})} 2^{-j(d-1) \frac{1}{p}}  2^{-j(1-s+\frac{1}{p})}   = C_{7} \sup_{j \in \mathbb{N}_{0} } 2^{-j (  \frac{d}{u} - s + 1  )} .
\end{align*}
But because of $  d/u - s + 1  < 0 $ we obtain $ \Vert T^{*}f    \vert   \mathcal{E}^{s}_{u,p,q}(\mathbb{R}^{d})   \Vert = \infty   $. The proof is complete.
\end{proof}

Now we turn our attention to the Besov-Morrey spaces. Here also for the critical border $ s = \min ( 1 + 1/p , 1 + d/u )    $ we obtain an almost complete result.

\begin{prop}\label{d>1_Nec_MR_N}
Let $ 1 \leq p \leq u < \infty  $ and $ 1 \leq q \leq \infty  $. Moreover we are in one of the following situations. 
\begin{itemize}
\item[(i)] We have $  s > \min ( 1 + \frac{1}{p} , 1 + \frac{d}{u} )   $.

\item[(ii)] We have $  s = \min ( 1 + \frac{1}{p} , 1 + \frac{d}{u} )  $ and $ q \not = \infty   $.

\item[(iii)] We have $ d = 1  $ with $ s = 1 + \frac{1}{u}  $ and $ q = \infty  $.

\end{itemize}
Then there exists a function $ f \in  \mathbb{N}^{s}_{u,p,q}(\mathbb{R}^{d}) $ such that $ Tf \not \in \mathbb{N}^{s}_{u,p,q}(\mathbb{R}^{d})  $ and $ T^{+}f \not \in \mathbb{N}^{s}_{u,p,q}(\mathbb{R}^{d})  $. 
\end{prop}

\begin{proof}
\textit{Step 1.} At first we look at the cases $ s > 1 + d/u  $ and $ s = 1 + d/u   $ with $ q \not = \infty   $. Here we work with the same function $ f \in C_{0}^{\infty}(\R)   $ as in the proof of proposition \ref{d>1_Nec_MR}, see formula \eqref{Ex_Fu_d>1_E1}. We proceed like there and apply proposition \ref{pro_N_diff}  with $ v = 1   $, $a = \infty   $ and $ N \geq 2   $. Then we find $  Tf \not \in  \mathbb{N}^{s}_{u,p,q}(\mathbb{R}^{d})  $ and $ T^{+}f \not \in  \mathbb{N}^{s}_{u,p,q}(\mathbb{R}^{d})   $. Notice that for $ q \not = \infty  $ we also obtain a result for $ s = 1 + d/u   $. The reason for this is that in the norm that can be found in proposition \ref{pro_N_diff} the integral concerning $ t $ is outside of the Morrey norm. 

\textit{Step 2.} Now we look at the case $ s > 1 + 1/p  $. We work with the same function $ f \in C_{0}^{\infty}(\R)   $ as in the proof of proposition \ref{d>1_Nec_MR}. Then because of lemma \ref{l_bp1} we have $ f \in  \mathbb{N}^{s}_{u,p,q}(\mathbb{R}^{d})  $. Let $ \epsilon > 0  $ such that $ s > s - \epsilon >   1 + 1/p   $. We have $ \mathbb{N}^{s}_{u,p,q}(\mathbb{R}^{d}) \hookrightarrow \mathbb{N}^{s - \epsilon}_{u,p,p}(\mathbb{R}^{d}) \hookrightarrow \mathbb{E}^{s - \epsilon}_{u,p,p}(\mathbb{R}^{d})$. Now we can apply proposition \ref{d>1_Nec_MR} and its proof. So we obtain $ Tf \not \in  \mathbb{E}^{s - \epsilon}_{u,p,p}(\mathbb{R}^{d})   $ and $ T^{+}f \not \in  \mathbb{E}^{s - \epsilon}_{u,p,p}(\mathbb{R}^{d})   $. Consequently we get $ Tf \not \in  \mathbb{N}^{s}_{u,p,q}(\mathbb{R}^{d})  $ and $ T^{+}f \not \in  \mathbb{N}^{s}_{u,p,q}(\mathbb{R}^{d})  $. 
 
\textit{Step 3.} Next we have to deal with $ s = 1 + 1/p $ and $ 0 < q < \infty  $. Again we work with the function $ f \in C_{0}^{\infty}(\R)  $ from the proof of proposition \ref{d>1_Nec_MR}, see formula \eqref{Ex_Fu_d>1_E1}. Of course we have $ f \in  \mathbb{N}^{1 + 1/p}_{u,p,q}(\mathbb{R}^{d}) $. In what follows we will prove $ Tf \not \in \mathbb{N}^{1 + 1/p}_{u,p,q}(\mathbb{R}^{d})  $ and $ T^{+}f \not \in \mathbb{N}^{1 + 1/p}_{u,p,q}(\mathbb{R}^{d})  $ simultaneously. For that purpose we write $ T^{*}  $ when we mean either $ T $ or $ T^{+} $. We use proposition \ref{pro_N_diff} with $ a = \infty  $, $v = 1   $ and $ N > 1 + 1/p  $. For the supremum in the Morrey norm we choose the smallest ball $ B^{*}  $ with $ [0,1]^{d} \subset B^{*}   $. We write $ x' = (x_{2}, x_{3}, \ldots , x_{d} ) \in \mathbb{R}^{d - 1}  $. Then we find
\begin{align*}
\Vert T^{*}f \vert \mathcal{N}^{1 + \frac{1}{p}}_{u,p,q}(\mathbb{R}^{d}) \Vert & \geq C_{1}   \Big ( \int_{0}^{\infty} t^{-sq-dq} \Big \Vert  \Big ( \int_{B(0,t)}\vert \Delta^{N}_{h}T^{*}f(x) \vert dh \Big ) \Big \vert  \mathcal{M}^{u}_{p}( \mathbb{R}^d) \Big \Vert^{q}  \frac{dt}{t} \Big )^{\frac{1}{q}} \\ 
& \geq C_{2}   \Big ( \int_{0}^{\infty} t^{-sq-dq}  \Big ( \int_{[0,1]^{d}}  \Big ( \int_{B(0,t)}\vert \Delta^{N}_{h}T^{*}f(x) \vert dh \Big )^{p} dx \Big )^{\frac{q}{p}}  \frac{dt}{t} \Big )^{\frac{1}{q}} \\ 
& \geq C_{2}   \Big ( \int_{0}^{1} t^{-sq-dq}  \Big ( \int_{0}^{\frac{t}{2}} \int_{[0,1]^{d-1}}  \Big ( \int_{\substack{ |h| \leq t  \\ h_{1} \leq - \frac{t}{2}  }}\vert \Delta^{N}_{h}T^{*}f(x) \vert dh \Big )^{p} dx' \ dx_{1} \Big )^{\frac{q}{p}}  \frac{dt}{t} \Big )^{\frac{1}{q}}.
\end{align*}
Now for $ t \in [0,1]  $ and $ x \in [0,\frac{t}{2}] \times [0,1]^{d - 1}    $ we have $ |f(x)| = x_{1}  $ and $ \max(f(x),0) = x_{1}  $. Moreover for $ |h| \leq t   $ with $ h_{1} \leq - t/2   $ and $ l \in \left\{ 1, 2, \ldots , N  \right\}   $ we have $ |f(x+lh)|= -x_{1} - l h_{1}   $ and $ \max(f(x+lh),0) = 0  $. Because of this like in the proof of proposition \ref{d>1_Nec_MR} for $ N \geq 2  $ we find
\begin{align*}
\vert \Delta^{N}_{h }|f|(x) \vert  = 2 x_{1} \qquad \mbox{and} \qquad \vert \Delta^{N}_{h }T^{+}f(x) \vert  =  x_{1}.
\end{align*}
So we have almost the same outcome for $ T $ and $ T^{+} $. Therefore for both cases we can proceed in the same way now. We obtain
\begin{align*}
\Vert T^{*}f \vert \mathcal{N}^{1 + \frac{1}{p}}_{u,p,q}(\mathbb{R}^{d}) \Vert & \geq C_{3}   \Big ( \int_{0}^{1} t^{-sq-dq}  \Big ( \int_{0}^{\frac{t}{2}} x_{1}^{p} \int_{[0,1]^{d-1}}  \Big ( \int_{\substack{ |h| \leq t  \\ h_{1} \leq - \frac{t}{2}  }} 1 dh \Big )^{p} dx' \ dx_{1} \Big )^{\frac{q}{p}}  \frac{dt}{t} \Big )^{\frac{1}{q}} \\
& \geq C_{4}   \Big ( \int_{0}^{1} t^{-sq}  \Big ( \int_{0}^{\frac{t}{2}} x_{1}^{p} \int_{[0,1]^{d-1}} 1  dx' \ dx_{1} \Big )^{\frac{q}{p}}  \frac{dt}{t} \Big )^{\frac{1}{q}} \\
& \geq C_{5}   \Big ( \int_{0}^{1} t^{-sq+q+\frac{q}{p}-1}  dt \Big )^{\frac{1}{q}} = \infty.
\end{align*}
In the last step we used $ s = 1 + 1/p  $ and $0 < q < \infty   $. Hence this step of the proof is complete.

\textit{Step 4.}
Now we look at the case $ d = 1  $ with $ s = 1 + 1/u $ and $ q = \infty   $. We will work with a function that can be found in lemma 2 in chapter 5.4.1 in \cite{RS}, see also the proposition in \cite{BouMey}. Let $ \varphi \in \mathbb{S}(\mathbb{R})  $ be a real-valued and odd function with $ \supp \mathcal{F}\varphi \subset [-1,1]   $ and $ \varphi(x) = x   $ for $ -1 \leq x \leq 1   $. We define
\begin{equation}\label{s_bord_EXF}
g(x) = \sum_{j = 0}^{\infty} 2^{-j} \varphi(2^{j}x).
\end{equation}
From lemma 2 in chapter 5.4.1 in \cite{RS} we learn that we have $ g \in \mathbb{B}^{1 + 1/u}_{u,\infty}(\mathbb{R})  $. Now because of $ \mathbb{B}^{1 + 1/u}_{u,\infty}(\mathbb{R}) = \mathbb{N}^{1 + 1/u}_{u,u,\infty}(\mathbb{R}) \hookrightarrow \mathbb{N}^{1 + 1/u}_{u,p,\infty}(\mathbb{R}) $, see formula \eqref{Lu_Mo_Em}, we also find $ g \in \mathbb{N}^{1 + 1/u}_{u,p,\infty}(\mathbb{R})   $. In what follows we will prove that we have $ Tg \not \in \mathbb{N}^{1 + 1/u}_{u,p,\infty}(\mathbb{R})   $ and $ T^{+}g \not \in \mathbb{N}^{1 + 1/u}_{u,p,\infty}(\mathbb{R})   $ simultaneously. For that purpose as before we write $ T^{*}  $ when we mean either $ T $ or $ T^{+} $. We use proposition \ref{pro_N_diff} with $ a = \infty  $ and $ v = \infty  $. This is possible because of $1/p \leq 1 < 1 + 1/u  $. Because of $ 1 + 1/u < 2  $ we can put $ N = 2  $. Then we find
\begin{align*}
\Vert T^{*}g  \vert   \mathcal{N}^{1 + \frac{1}{u}}_{u,p,\infty}(\mathbb{R})   \Vert & \geq C_{1}    \sup_{0 \leq t < \infty} t^{- 1 - \frac{1}{u} } \Big \Vert  \sup_{|h| \leq t}\vert \Delta^{2}_{h}T^{*}g(x) \vert  \Big \vert  \mathcal{M}^{u}_{p}( \mathbb{R}) \Big \Vert \\
& \geq C_{1} \sup_{0 \leq t < \infty} t^{- 1 - \frac{1}{u} } \sup_{\substack{a<b \\ |a| \leq t  ,|b| \leq t }} \vert a - b \vert^{\frac{1}{u}-\frac{1}{p}} \Big ( \int_{a}^{b}  \vert  \Delta^{2}_{-x}T^{*}g(x)   \vert^{p} dx  \Big )^{\frac{1}{p}}.  
\end{align*} 
From the proof of lemma 2 in chapter 5.4.1 in \cite{RS} we know that we have $ g(0) = 0  $ and $ \vert g(x) \vert =  \vert g(-x) \vert $. Because of $ g $ is odd we find
\begin{align*}
\Delta^{2}_{-x}|g|(x)  =  2 |g(x)| \qquad \mbox{and} \qquad \Delta^{2}_{-x}T^{+}g(x) = |g(x)|.
\end{align*}
So we have almost the same outcome for $ T $ and $ T^{+} $. Then in both cases we obtain 
\begin{align*}
\Vert T^{*}g  \vert   \mathcal{N}^{1 + \frac{1}{u}}_{u,p,\infty}(\mathbb{R})   \Vert & \geq C_{2} \sup_{0 \leq t < \infty} t^{- 1 - \frac{1}{u} } \sup_{\substack{a<b \\ |a| \leq t  ,|b| \leq t }} \vert a - b \vert^{\frac{1}{u}-\frac{1}{p}} \Big ( \int_{a}^{b}  \vert  g(x)   \vert^{p} dx  \Big )^{\frac{1}{p}} \\ 
& \geq C_{2} \sup_{0 \leq t < \infty} t^{- 1 - \frac{1}{u} }   t^{\frac{1}{u}-\frac{1}{p}} \Big ( \int_{0}^{t}  \vert  g(x)   \vert^{p} dx  \Big )^{\frac{1}{p}} \\  
& \geq C_{3} \sup_{0 < t < 1} t^{- 2 }   \int_{0}^{t} \Big \vert \sum_{j = 0}^{\infty} 2^{-j} \varphi(2^{j}x)  \Big  \vert dx. 
\end{align*} 
In the last step we used the H\"older inequality. Now for $0 < t < 1  $ let $ L(t) \in \mathbb{N}_{0}  $ be the biggest natural number such that $ L(t) < \frac{\ln(\frac{1}{t})}{\ln(2)}   $. Then for $ j \in \mathbb{N}_{0}  $ with $ j \leq L(t)   $ and $ x \leq t  $ we have $  2^{j} x \leq 1   $. Now because of the definition of the function $ \varphi  $ we obtain
\begin{align*}
\Vert T^{*}g  \vert   \mathcal{N}^{1 + \frac{1}{u}}_{u,p,\infty}(\mathbb{R})   \Vert & \geq C_{3} \sup_{0 < t < 1} t^{- 2 } \int_{0}^{t} \Big \vert \sum_{j = 0}^{L(t) - 1} x + \sum_{j = L(t) }^{\infty} 2^{-j} \varphi(2^{j}x)   \Big  \vert dx   \\
& \geq C_{4} \sup_{0 < t < 1} t^{- 2 } \Big ( L(t) \frac{ t^{2}}{2}  - K t   \sum_{j = L(t) }^{\infty} 2^{-j}  \Big ) .
\end{align*} 
In the last step we used that $ \varphi \in \mathbb{S}(\mathbb{R})  $ is bounded by a constant $ K < \infty  $. Now we calculate
$ \sum_{j = L(t) }^{\infty} 2^{-j} \leq C_{5} 2^{- L(t)} \leq C_{6} t    $. Hence we get
\begin{align*}
\Vert T^{*}g  \vert   \mathcal{N}^{1 + \frac{1}{u}}_{u,p,\infty}(\mathbb{R})   \Vert   \geq C_{7} \sup_{0 < t < 1} t^{- 2 } \Big ( L(t)  \frac{ t^{2}}{2}  - K t^{2}   \Big )  \geq C_{8} \lim_{t \downarrow 0} \Big  ( \frac{1}{2} \ln \Big (\frac{1}{t} \Big )   - K  \Big )  = \infty.
\end{align*} 
So the proof is complete.
\end{proof} 

\textbf{Proof of theorem \ref{MR_N_d>1} and theorem \ref{MR_E_d>1}.}
To prove the main theorems \ref{MR_N_d>1} and \ref{MR_E_d>1} we just have to combine the propositions \ref{pro_0<s<1}, \ref{d>1_E_MR}, \ref{s=1_E_MR}, \ref{d>1_N_MR}, \ref{d>1_Nec_MR} and \ref{d>1_Nec_MR_N}.

\section{Further properties and outstanding issues}\label{sec_5}

In this paper we investigated the boundedness and the acting property of the operators $ T^{+}  $ and $ T  $ in the context of the spaces $ \mathbb{A}^{s}_{u,p,q}(\mathbb{R}^{d})  $. Both things are at least partly understood now. But one may also ask whether the operators $ T^{+}  $ and $  T $ are continuous or even Lipschitz continuous on $ \mathbb{A}^{s}_{u,p,q}(\R)  $. In general for $ s > 0 $ Lipschitz continuity can not be expected. For the special case $ p = u  $ we refer to \cite{Tr01}, see theorem 25.14. Most likely it is possible to use the ideas from there also for $ p < u  $. On the other hand to prove a satisfactory result concerning continuity seems to be a difficult problem. In the following list we will collect some open problems concerning the operators $ T^{+} $ and $ T $.

\begin{itemize}
\item[(i)] The first question concerns the mapping properties of the operator $ T^{+}  $ in the context of the Triebel-Lizorkin-Morrey spaces on the critical border $ s = 1 + d/u  $ with $ u > dp  $, see theorem \ref{MR_E_d>1}. Do we have $ T^{+} (\mathbb{E}^{1 + d/u}_{u,p,q}(\mathbb{R}^{d}) ) \subset \mathbb{E}^{1 + d/u}_{u,p,q}(\mathbb{R}^{d})  $ ? 

\item[(ii)] The next query is related to the mapping properties of  $ T^{+}    $ in the case $ d > 1   $ and $ 1 \leq s < \min  ( 1 + 1/p ,  1 + d/u  )    $. Is it possible to omit the assumption $  1/p - 1/u > 1 - 1/d   $ you can find in the main results theorem \ref{MR_N_d>1} and theorem \ref{MR_E_d>1} ?

\item[(iii)] The spaces $ \mathbb{N}^{s}_{u,p,q}(\mathbb{R}^{d})  $ and $ \mathbb{E}^{s}_{u,p,q}(\mathbb{R}^{d})  $ are also well-defined for $ 0 < p \leq u < \infty   $ and $ 0 < q \leq \infty   $. One may discuss the mapping properties of $ T^{+}  $ and $  T $ in this more general setting. Some results concerning the special case $ p = u  $ can be found in theorem 25.8 in \cite{Tr01}.

\item[(iv)] The next issue concerns the continuity of $ T^{+} $ and $ T $. Under which conditions on the parameters $ s,p,u,q $ and $ d $ the operator $ T^{+} : \mathbb{A}^{s}_{u,p,q}(\mathbb{R}^{d}) \rightarrow \mathbb{A}^{s}_{u,p,q}(\mathbb{R}^{d})  $ is continuous? Notice that for the special case $ p = u  $ and $ 0 < s \leq 1  $ some positive results are already known, see \cite{MaMiz}, \cite{MuNa} and theorem 3 in chapter 5.5.2 in \cite{RS}. 

\end{itemize}

\section{Appendix: The zero set of real analytic functions}

In what follows we will collect some facts concerning the zero set of a real analytic function. They are used in the proof of proposition \ref{d>1_E_MR}. 

\begin{lem}\label{lem_ZOAF}
Let $ f : \mathbb{R}^d \rightarrow \mathbb{R}   $ be an analytic function. Then we know the following.
\begin{itemize}
\item[(i)] Let $ f \not = 0  $. Then the set $ Z(f) = \{ x \in \R  : f(x) = 0  \}  $ is the union of countably many compact sets $ K_{j}  $ with $ \lambda_{d - 1} (K_{j}) < \infty  $ for all $ j \in \mathbb{N} $. Here with $ \lambda_{d - 1}   $ we denote the $ (d-1) -$dimensional Lebesgue measure. Moreover the Hausdorff dimension of $ Z(f)  $ does not exceed $ d - 1  $. 
\item[(ii)] Let $ d = 1  $. Let $ (x_{n})_{n \in \mathbb{N}} \subset \mathbb{R}  $ be a sequence with $ \lim_{n \rightarrow \infty} x_{n} = x    $ that fulfills $  x_{n} \not = x   $ and $ f( x_{n}) = 0   $ for all $ n  \in \mathbb{N}   $. Then we obtain $ f = 0  $.

\item[(iii)] Let $ f \not = 0  $. Suppose that $ f(0',x_{d})   $ has a zero of multiplicity $ m \in \mathbb{N}  $ at $ x_{d} = 0   $. Then there exist open intervals $ I_{1}, I_{2}, \ldots , I_{d} \subset \mathbb{R}   $ centered at $0$ such that $  f(x', \cdot )   $ has for each $ x' \in \mathbb{R}^{d-1}  $ with $ x' \in I_{1} \times I_{2} \times \ldots \times I_{d-1}   $ not more than $ m $ zeros in $ I_{d} $ counted according to their multiplicities. Moreover the multiplicity $ m  $ is always finite (after a suitable rotation of $f$ maybe).

\item[(iv)] Let $ R \in \mathbb{N}   $, $ v \in  \mathbb{Z}^{d}   $ and $ k \in \mathbb{N}  $ with $ 0 <  k \leq R    $. Let $ Q_{k,v}   $ be a dyadic cube and $ f \not = 0   $. Then there exist \\ 
$ \bullet   $ a natural number $ n(f,k,v) \in \mathbb{N}   $ that depends on $ f, k, v $, \\
$ \bullet   $ a constant $ c(d)  $ that only depends on $ d $, \\
$  \bullet     $ a number  $ r \in \mathbb{N}   $ with $ r $ much larger than $ R $ and $ 2^{r-k}   $ much larger than $  c(d)   n(f,k,v)  $  \\ 
such that the set $  Z(f) \cap  2 Q_{k,v}   $ can be covered by $ c(d) \  n(f,k,v) \ 2^{(d-1)(r-k)} $ $d-$dimensional cubes with side-length $ 2^{-r}   $.
 
\item[(v)] Assume we are in the same situation as described in (iv). Then the number $ n(f,k,v) = n(f) \in \mathbb{N}  $ can be chosen independent of $ k $ and $ v $.
\end{itemize}
\end{lem} 

\begin{proof}
\textit{Proof of (i) - (iii).} Fact (i) can be deduced from theorem 14.4.9. in \cite{Rud}. For the result concerning the Hausdorff dimension we also refer to \cite{mit}. A definition of the Hausdorff dimension can be found in definition 14.4.1. in \cite{Rud}, and also in \cite{Fal} and \cite{Mat95}. Fact (ii) is a classical result from complex analysis.
Fact (iii) can be derived from lemma 14.1.2.(i) in \cite{Rud}. The result concerning the multiplicity also can be found in the proof of claim 2 in \cite{mit}.
 
\textit{Proof of (iv).}
To prove (iv) let us start with the case $ d = 1  $. From (i) and (ii) we learn that $  Z(f) \cap  2 Q_{k,v}   $ consists of a finite number $ n_{0}(f,k,v) \in \mathbb{N}   $ of isolated points in $ \mathbb{R}  $. Therefore for each large $ r \in \mathbb{N}  $ the set $  Z(f) \cap  2 Q_{k,v}   $ can be covered by $ n_{0}(f,k,v)   $ intervals with side-length $ 2^{-r}   $.

Now we look at the case $ d > 1  $. Let $ z_{1} \in Z(f) \cap \overline{2 Q_{k,v}}   $ such that the multiplicity $ m_{1} \in \mathbb{N}  $ of the zero is as big as possible. From fact (iii) we know $ m_{1} < \infty  $ (maybe after a rotation of $f$). If there exist two or more zeros with the same maximal multiplicity we choose that one that allows us to find the biggest cube $ Q_{1}  $ appropriate to the description given now. From (iii) we learn that there exist open intervals $ I_{1}, I_{2}, \ldots , I_{d} \subset \mathbb{R}   $ with $ | I_{i} | = 2^{-r_{1}}  $ for all $i$ and for some $ r_{1} \geq k  $ such that $ z_{1}  $ is in the center of $ Q_{1} =  I_{1} \times I_{2} \times \ldots \times I_{d}    $ and such that $  f(z', \cdot )   $ has for each $ z' \in \mathbb{R}^{d-1}  $ with $ z' \in I_{1} \times I_{2} \times \ldots \times I_{d-1} = Q_{1}'   $ not more than $ m_{1} $ zeros in $ I_{d} $ counted according to their multiplicities. In view of (i) that means the set $ Z(f) \cap Q_{1}   $ consists of not more than $ m_{1} $ manifolds of dimension $ d - 1 $ that meet at $ z_{1} $ and maybe also somewhere else. Choose the cube $ Q_{1} \subset \mathbb{R}^{d}  $ as large as possible. That means for each $ \epsilon > 0  $ there exists a set $ Z_{1 , \epsilon} \subset (1 + \epsilon) Q_{1}'  $ with $ \lambda_{d - 1}(Z_{1 , \epsilon}) > 0  $ such that in $ Z_{1 , \epsilon}  \times (1 + \epsilon) I_{d}  $ the function $  f(z', \cdot )   $ has more than $ m_{1} $ zeros for some $ z' \in Z_{1 , \epsilon}   $. We find that the set $ ( Z(f) \cap  2 Q_{k,v} ) \cap Q_{1}  $ can be covered by not more than $ c(d) m_{1} 2^{(d-1)(t_{1} - r_{1})}  $ cubes of dimension $ d $ with side-length $ 2^{-t_{1}}  $ for some $ t_{1} \in \mathbb{N}  $ with $ t_{1} $ much larger than $ r_{1}   $ such that $ c(d) m_{1} $ is much smaller than $ 2^{t_{1} - r_{1}}  $. To see this we can argue as follows. 

At first assume $ z_{1}  $ has multiplicity one. Then from the Implicit Function Theorem it follows that the set $ Z(f) \cap Q_{1}   $ is the graph of a $ C^{\infty} -  $ function of $ d-1  $ variables, see also remark 14.1.4. in \cite{Rud}. Therefore the set $ Z(f) \cap Q_{1}   $ can be interpreted as a smooth manifold of dimension $ d-1  $, see for example chapter 16.1 in \cite{Tr97}. Consequently we also find that $ Z(f) \cap Q_{1}   $ is a so-called $ (d-1) - $ set, see definition 3.1 in \cite{Tr97} and remark 16.3 in \cite{Tr97}. Hence $ Z(f) \cap Q_{1}   $ has Minkowski dimension (box counting dimension) of $ d - 1  $, see remark 3.5 in \cite{Tr97}. For details concerning the Minkowski dimension we refer to \cite{Mat95}, pages 76-81, and to \cite{Fal}.  From the definition of the Minkowski dimension it follows that $ Z(f) \cap Q_{1}   $ can be covered by $ c(d) 2^{(d-1)(t_{1} - r_{1})}   $ cubes of side-length $ 2^{-t_{1}}   $ when $ t_{1} $ is large enough.  For multiplicity $  m_{1} > 1 $ we have to decompose the set $ Z(f) \cap Q_{1}   $ in not more than $  m_{1} $ smooth manifolds that we cover separately. See also theorem 14.1.3. in \cite{Rud}.  

Now take $  z_{2} \in Z(f) \cap \overline{ 2 Q_{k,v} } $ such that $ z_{2} \not \in Q_{1}    $ and with multiplicity $ m_{2} \in \mathbb{N}  $ as big as possible. Of course we have $ m_{2} \leq m_{1}   $. We proceed exactly as before and obtain a cube $ Q_{2}  $ and numbers $ m_{2}, r_{2}, t_{2}    $ as well as a covering of $ ( Z(f) \cap  2 Q_{k,v} ) \cap Q_{2}   $. We choose the cube $ Q_{2} \subset \mathbb{R}^{d} $ as large as possible (in the sense described before) but in such a way that $ ( Q_{1} \cap Z(f)   ) \cap Q_{2} = \emptyset    $. This process will be continued till the whole set $ Z(f) \cap  2 Q_{k,v}   $ is covered. If there are two or more zeros with the same multiplicity $ m_{i}  $ we always continue with that one that allows us to find the biggest cube $ Q_{i}  $.  At the end we obtain a sequence of cubes $ \{ Q_{i} \}_{i}  $ and sequences $ \{ m_{i} \}_{i}, \{ r_{i} \}_{i}, \{ t_{i} \}_{i}   $. Below in picture Fig. 1 we tried to illustrate one step of the algorithm we just described for $ d = 2  $. Notice that the iteration ends after a finite number of $ w \in \mathbb{N}  $ of steps. 

To see this we can argue as follows. Assume $ w = \infty  $. Then since $ 2 Q_{k,v}  $ is bounded we find a subsequence $  \{ Q_{i_{l}} \}_{l}   $ of cubes such that $ \lim_{l \rightarrow \infty }  | Q_{i_{l}} | = 0    $. Because of the definition of the algorithm and the fact that $ Z(f)  $ consists of countably many compact sets, see (i), that implies the following. There exist two sequences of sets $  \{ A_{l} \}_{l}   $ and $  \{ B_{l} \}_{l}  $ with

$ \bullet $ $  A_{l} = ( Z(f) \cap \overline{  Q_{i_{l}} } ) $ and $  B_{l}  = ( Z(f) \cap ( \overline{ 2 Q_{i_{l}} \setminus  Q_{i_{l}} } )  ) $ for all $ l \in \mathbb{N}  $; 

$ \bullet $ $ \lambda_{d - 1}( A_{l}  ) > 0   $ and $  B_{l} \not =   \emptyset   $ for all $ l \in \mathbb{N}  $;

$ \bullet $ for each $ A_{l}  $ there is a generating zero $ z_{i_{l}}  $ with multiplicity $ m_{i_{l}}  $;

$ \bullet $ we know $\infty > m_{i_{1}} \geq \ldots \geq m_{i_{l}} \geq m_{i_{l + 1}} \geq \ldots   $ for all $ l \in \mathbb{N}   $;

$ \bullet $ the set $  A_{l} \cup B_{l}   $ is not connected for all large $ l $ due to the definition of the algorithm;

$ \bullet $ $  \lim_{l \rightarrow \infty} \dist(A_{l} , B_{l}) = 0    $. 

But in the limiting case $ l \rightarrow \infty   $ the last 3 points and (iii) generate a contradiction. An increase of multiplicity in a late step of the algorithm (forced by the last point and (iii)) is forbidden and an infinite multiplicity does not exist. So our assumption must be wrong and we find $ w < \infty   $. We tried to illustrate this argument in picture Fig. 2 below for $ d = 2 $.

Now we have to unify the size of the very small cubes we use for the covering.  Therefore because of $ \max _{1 \leq i \leq w} m_{i} = m_{1} $ we put $ n(f,k,v) = m_{1} w    $. Moreover we choose  $ t^{*} \geq \max_{1 \leq i \leq w} t_{i} $ such that $ 2^{t^{*} - k}   $ is much bigger than $ c(d) m_{1} w $ . So due to $ \min_{1 \leq i \leq w} r_{i} \geq k   $ we can cover the set $ Z(f) \cap  2 Q_{k,v}  $ with $ c(d) m_{1} w 2^{(d-1)(t^{*} - k)}      $ cubes with side-length $ 2^{- t^{*}} $. 

\textit{Proof of (v).} To see this at first recall that we have $ 0 < k \leq R   $. Because of this it is enough to work with $ k = 1  $ in the proof of (iv) to identify a possible number $ n(f,k,v)   $. To prove the independence from $ v \in \mathbb{Z}^{d}  $ we choose $ v^{*} \in \mathbb{Z}^{d}  $ such that $ n(f,v^{*})   $ is maximal. We already explained $  n(f,v^{*})  < \infty  $. But of course the number $ n(f,v^{*})  $ works for each $ v \in \mathbb{Z}^{d}  $. So $ n(f,v) = n(f)   $ only depends on $f$.
\end{proof}

\begin{figure}[h]

\begin{minipage}[b]{0.1\textwidth}

\includegraphics{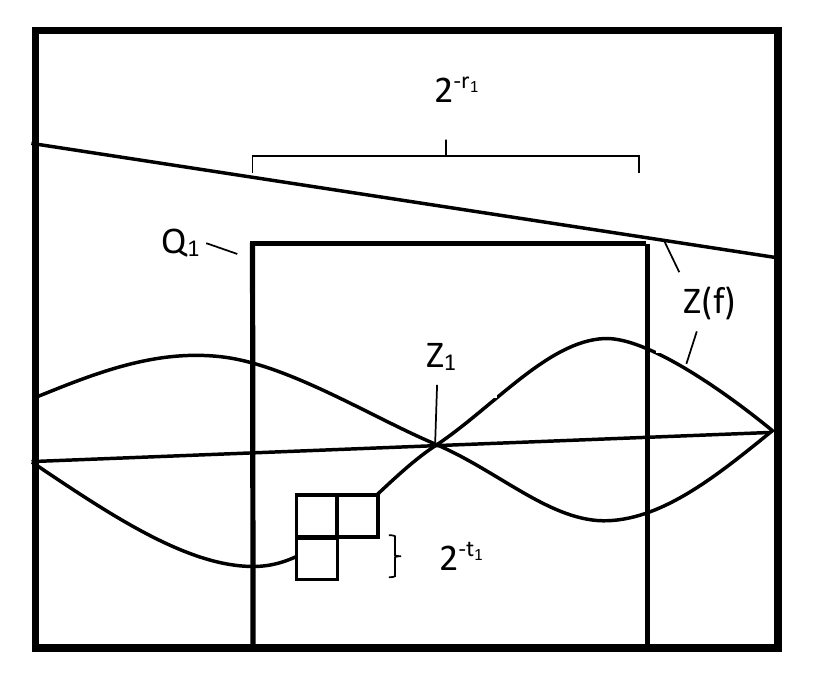}%

\end{minipage}\hfill
\begin{minipage}[b]{0.4\textwidth}
\textbf{Fig. 1. One step of the algorithm.} In this picture we try to illustrate a typical situation for $ d = 2  $. The notation is the same as in the proof of (iv). $ Z(f) $ is the zero set of a real analytic function. $ z_{1} $ is a zero of maximal multiplicity that allows to find a cube $ Q_{1}  $ that is as large as possible apposite to the algorithm. The small cubes with side-length $ 2^{-t_{1}}  $ deliver a covering for $ Z(f) \cap Q_{1} $.

\vspace{1,3 cm}

\end{minipage}

\end{figure}

\begin{figure}[h]

\begin{minipage}[b]{0.1\textwidth}

\includegraphics{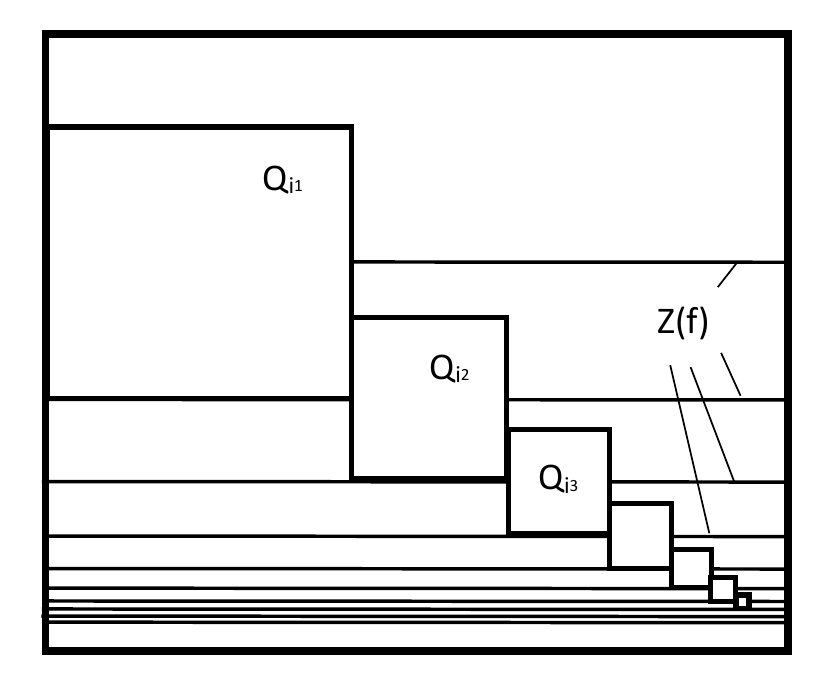}%

\end{minipage}\hfill
\begin{minipage}[b]{0.4\textwidth}
\textbf{Fig. 2. The reason for $ w < \infty   $.} In this picture we try to illustrate why the algorithm must break after a finite number of steps in the case $ d = 2  $. Assume the number of steps is $ w = \infty  $. That means the algorithm produces a subsequence of cubes $ \{ Q_{i_{l}} \}_{l}  $ that become smaller and smaller. In the limiting case this implies the existence of $ (d -1) -  $ dimensional manifolds $  A_{\infty} \subset Z(f)   $ and $  B_{\infty} \subset Z(f)   $ that are not connected but have a distance of 0. That either contradicts (iii) or results in $ f = 0 $.     

\vspace{0,1 cm}

\end{minipage}

\end{figure}


\vspace{1 cm}

\textbf{Acknowledgements }

The author is funded by a Landesgraduiertenstipendium which is a scholarship from the Friedrich-Schiller university and the Free State of Thuringia. The author would like to thank his supervisor professor Winfried Sickel for his tips and hints.


\end{document}